\documentclass[11pt]{article}
\usepackage{color}
\usepackage{amsfonts}
\usepackage{graphicx}
\usepackage{amsmath,enumerate}
\usepackage{cases}
\usepackage{amsthm}
\usepackage{amssymb}
\usepackage{amsfonts}
\usepackage{latexsym,amssymb,amsfonts}
\usepackage{mathrsfs}
\usepackage{graphicx}
\usepackage{psfrag}
\usepackage{amsmath, amsbsy}
\usepackage{amsopn, amstext}
\usepackage{amsmath,multicol,amsopn}
\usepackage{latexsym,amssymb}

\usepackage{bm}
\usepackage{cases}
\usepackage{amsfonts}
\usepackage{latexsym,enumerate,amscd,verbatim}
\usepackage{mathrsfs}

\allowdisplaybreaks
\numberwithin{equation}{section}

\newtheorem{theorem}{Theorem}[section]
\newtheorem{lemma}{Lemma}[section]
\newtheorem{proposition}{Proposition}[section]
\newtheorem{corollary}{Corollary}[section]
\newtheorem{definition}{Definition}[section]

\newtheorem{remark}{Remark}[section]

\def\3n{\negthinspace \negthinspace \negthinspace }
\def\2n{\negthinspace \negthinspace }
\def\1n{\negthinspace }

\makeatletter

\def\mn{\mathbb{N}}

\def\mf{\mathcal{F}}

\def\mb{\mathcal{B}}

\def\cd{\cdot}

\def\dbE{{\mathbb{E}}}
\def\dbF{{\mathbb{F}}}

\def\dbH{{\mathbb{H}}}
\def\dbI{{\mathbb{I}}}

\def\dbK{{\mathbb{K}}}

\def\dbN{{\mathbb{N}}}

\def\dbP{{\mathbb{P}}}
\def\dbQ{{\mathbb{Q}}}
\def\dbR{{\mathbb{R}}}

\def\a{\alpha}
\def\b{\beta}

\def\d{\delta}
\def\e{\varepsilon}

\def\k{\kappa}
\def\l{\lambda}

\def\si{\sigma}
\def\t{\tau}
\def\f{\varphi}
\def\th{\theta}
\def\o{\omega}
\def\p{\phi}

\def\3n{\negthinspace \negthinspace \negthinspace }
\def\2n{\negthinspace \negthinspace }
\def\1n{\negthinspace }
\def\ns{\noalign{\smallskip} }

\def\ds{\displaystyle}
\def\wh{\widehat}
\def\G{\Gamma}

\def\L{\Lambda}
\def\Si{\Sigma}

\def\O{\Omega}
\def\Om{\Omega}
\def\cA{{\cal A}}
\def\cB{{\cal B}}
\def\cC{{\cal C}}
\def\cD{{\cal D}}
\def\cE{{\cal E}}
\def\cF{{\cal F}}
\def\cG{{\cal G}}

\def\cI{{\cal I}}
\def\cJ{{\cal J}}
\def\cK{{\cal K}}
\def\cL{{\cal L}}
\def\cM{{\cal M}}
\def\cN{{\cal N}}

\def\cP{{\cal P}}
\def\cQ{{\cal Q}}

\def\cT{{\cal T}}
\def\cU{{\cal U}}
\def\cV{{\cal V}}
\def\cW{{\cal W}}

\def\cY{{\cal Y}}
\def\cZ{{\cal Z}}

\def\mE{{\mathbb{E}}}

\def\no{\noindent}

\def\ss{\smallskip}

\def\q{\quad}
\def\qq{\qquad}
\def\hb{\hbox}

\def\limsup{\mathop{\overline{\rm lim}}}

\def\lan{\big\langle}
\def\ran{\big\rangle}

\def\pa{\partial}

\def\wt{\widetilde}
\def\cd{\cdot}
\def\cds{\cdots}

\def\ae{\hbox{\rm a.e.{ }}}
\def\as{\hbox{\rm a.s.{ }}}

\def\span{\hbox{\rm span$\,$}}

\def\({\Big (}
\def\){\Big )}
\def\[{\Big[}
\def\]{\Big]}

\def\={\buildrel \triangle \over =}

\def\resp{{\it resp. }}

\def\dbE{{\mathbb{E}}}
\def\dbF{{\mathbb{F}}}

\def\dbH{{\mathbb{H}}}
\def\dbI{{\mathbb{I}}}

\def\dbK{{\mathbb{K}}}

\def\dbN{{\mathbb{N}}}

\def\dbP{{\mathbb{P}}}
\def\dbQ{{\mathbb{Q}}}
\def\dbR{{\mathbb{R}}}

\def\ds{\displaystyle}

\def\={\buildrel \triangle \over =}

\def\resp{{\it resp. }}
\def\a{\alpha}
\def\b{\beta}

\def\d{\delta}
\def\e{\varepsilon}

\def\k{\kappa}
\def\l{\lambda}

\def\si{\sigma}
\def\t{\times}
\def\f{\varphi}
\def\th{\theta}

\def\p{\phi}
\def\ns{\noalign{\ss} }

\def\pa{\partial}
\def\G{\Gamma}

\def\L{\Lambda}
\def\Si{\Sigma}

\def\Om{\Omega}

\def\cA{{\cal A}}
\def\cB{{\cal B}}
\def\cC{{\cal C}}
\def\cD{{\cal D}}
\def\cE{{\cal E}}
\def\cF{{\cal F}}
\def\cG{{\cal G}}

\def\cI{{\cal I}}
\def\cJ{{\cal J}}
\def\cK{{\cal K}}
\def\cL{{\cal L}}
\def\cM{{\cal M}}
\def\cN{{\cal N}}

\def\cP{{\cal P}}
\def\cQ{{\cal Q}}

\def\cT{{\cal T}}
\def\cU{{\cal U}}
\def\cV{{\cal V}}
\def\cW{{\cal W}}

\def\cY{{\cal Y}}
\def\cZ{{\cal Z}}

\def\mE{{\mathbb{E}}}

\def\no{\noindent}

\def\ss{\smallskip}

\def\q{\quad}
\def\qq{\qquad}
\def\hb{\hbox}

\def\t{\tau}

\def\bu{\bar{u}}
\def\bx{\bar{x}}

\def\T{T^{b}}
\def\TT{T^{b(2)}}
\def\eps{\varepsilon}
\def\t{\tau}

\newcommand{\inner}[2]{\left\langle#1,#2\right\rangle}

\renewcommand{\@seccntformat}[1]{\csname the#1\endcsname.\hspace{0.5em}}
\makeatother

\begin{document}

\title{\bf First and Second Order Necessary Optimality
Conditions for Controlled Stochastic Evolution
Equations with Control and State Constraints}

\author{H\'{e}l\`{e}ne Frankowska \thanks{CNRS,
IMJ-PRG, UMR 7586, Sorbonne Universit\'{e}, case
247, 4 place Jussieu, 75252 Paris, France. The
research of this author is partially supported
by the Gaspard  Monge Program for Optimisation
and Operational Research, Jacques Hadamard
Mathematical Foundation (FMJH) and the AFOSR
grant FA 9550-18-1-0254. {\small\it E-mail:}
{\small\tt helene.frankowska@imj-prg.fr}. } \q
and \q Qi L\"u\thanks{School of Mathematics,
Sichuan University, Chengdu, 610064, China. The
research of this author is supported by the NSF
of China under grant 11471231, the Fundamental
Research Funds for the Central Universities in
China under grant 2015SCU04A02 and Grant
MTM2014-52347 of the MICINN, Spain. {\small\it
E-mail:} {\small\tt lu@scu.edu.cn}. The authors
are grateful to Dr. Haisen Zhang for the
fruitful discussion which played an important
role in proving the main results of this paper.}
}

\date{}

\maketitle

\begin{abstract}
The purpose of this paper is to establish  first
and  second order necessary optimality
conditions for optimal control problems of
stochastic evolution equations with control and
state constraints. The control acts both in the
drift and diffusion terms and the control region
is a nonempty closed subset of a separable
Hilbert space. We employ some classical
set-valued analysis tools and theories of the
transposition solution of vector-valued backward
stochastic evolution equations and the
relaxed-transposition solution of
operator-valued backward stochastic evolution
equations  to derive these optimality
conditions. The correction part of the second
order adjoint equation, which does not appear in
the first order optimality condition, plays a
fundamental role in the second order optimality
condition.
\end{abstract}

\vspace{+0.3em}


\vspace{+0.3em}


\section{Introduction}

Let $T>0$ and $(\Om,\cF,\dbF,\dbP)$ a complete
filtered probability space  with the
c\`{a}dl\`{a}g (right continuous with left
limits) filtration $\dbF=\{\cF_t\}_{t\in[0,T]}$,
on which a cylindral  Brownian motion
$\{W(t)\}_{t\in[0,T]}$ taking values in a
separable Hilbert space $V$ is defined. Let  $H$
be a separable Hilbert space and $A$  be an
unbounded linear operator  generating a
contractive $C_0$-semigroup $\{S(t)\}_{t\geq 0}$
on $H$. For a nonempty closed subset  $U$  of a
separable Hilbert space $H_1$
define\vspace{-0.1cm}
$$
\cU \triangleq \Big\{u(\cdot):\,[0,T]\to U\;
\Big|\; u(\cd) \in L^2_\dbF(0,T;H_1)
\Big\}\vspace{-0.1cm}
$$
and consider the following controlled stochastic
evolution equation (SEE for
short):\vspace{-0.1cm}
\begin{equation}\label{controlsys}
\left\{
\begin{array}{ll}
dx(t)=\big(Ax(t)+a(t,x(t),u(t))\big)dt
+b(t,x(t),u(t))dW(t) & \mbox{ in }(0,T],\\
x(0)=\nu_0\in H,
\end{array}
\right.\vspace{-0.1cm}
\end{equation}
where $u \in \cU $.
A process $x(\cd)\equiv x(\cd\,;\nu_0,u)\in
L^{2}_\dbF(\O;C([0,T];H))$ is called a mild
solution of \eqref{controlsys} if\vspace{-0.2cm}
$$
x(t)=S(t)\nu_0\! +\! \int_0^t\!
S(t\!-\!s)a(s,x(s),u(s))ds + \!\int_0^t\!
S(t\!-\!s)b(s,x(s),u(s))dW(s),\;\dbP\mbox{-a.s.,}\;\forall\,
t\in\![0,T].\vspace{-0.2cm}
$$

Many controlled stochastic partial differential
equations, such as controlled stochastic
wave/heat/\-Schr\"odinger equations, can be
regarded as a special case of the system
\eqref{controlsys}.

\vspace{0.2cm}

Let $\cK_a$ be a nonempty closed subset of $H$,
and $h:\Om\times H\to \dbR$, $g^j:H\to \dbR$
($j=0,\cds,n$). Define a Mayer type cost functional
$\cJ(\cdot)$ (for the control system
\eqref{controlsys}) as\vspace{-0.2cm}
\begin{equation}\label{costfunction}
\cJ(u(\cdot), \nu_0)= \mE h(x(T))
\end{equation}
with the state constraint\vspace{-0.2cm}
\begin{equation}\label{constraints1}
\mE g^0(x(t))\leq 0,\q \mbox{ for all }t\in
[0,T],
\end{equation}
and the initial-final states
constraints\vspace{-0.2cm}
\begin{equation}\label{constraints}
\nu_0\in \cK_a,\qq \mE g^j(x(T))\leq 0,\q
j=1,\cds,n.
\end{equation}

The set of admissible controls at  the initial
datum $\nu_0$ is given by\vspace{-0.1cm}
$$
\begin{array}{ll}\ds
\cU_{ad}^{\nu_0} \=\Big\{ u\in \cU \ \big|\
\mbox{the corresponding solution $x(\cd)$ of
\eqref{controlsys} satisfies
\eqref{constraints1} and \eqref{constraints}}
\Big\}
\end{array}\vspace{-0.1cm}
$$
and the one of admissible trajectory-control
pairs by \vspace{-0.2cm}
$$
\begin{array}{ll}\ds
\cP_{ad} \=\Big\{(x(\cd),u(\cd))\ \big|\ u\in
\cU_{ad}^{\nu_0}\;\;  \mbox{\rm  for some } \nu_0 \in  \cK_a \Big\}.
\end{array}\vspace{-0.1cm}
$$
Under the usual assumptions, \eqref{controlsys}
has exactly one (mild) solution $x(\cdot,
\nu_0)$ with initial value $\nu_{0}\in \cK_a$,
which is called an admissible state.

We present the optimal control problem for the
system \eqref{controlsys} as follows:

\medskip

\no {\bf Problem (OP)} {\it Find  $(\bar
\nu_0,\bar u(\cdot))\in \cK_a\times
\cU_{ad}^{\nu_0}$ such that
\begin{equation}\label{jk2}
\cJ (\bar \nu_0,\bar u(\cdot)) =
\inf_{(\nu_0,u(\cdot))\in \cK_a\times
\cU_{ad}^{\nu_0}} \cJ
(\nu_0,u(\cdot)).\vspace{-0.21cm}
\end{equation}
}

In(1.5), $\bar u(\cdot)$ is said to be an
optimal control and $\bar x(\cdot)$  the
corresponding optimal state.  $(\bar
x(\cdot),\bar u(\cdot))$ is called an {\it
optimal pair} and $(\bar
\nu_{0},\bar{x}(\cd),\bar u(\cd))$ is called an
{\it optimal triple}.

Our purpose is to establish  first and  second
order necessary optimality conditions for {\bf
Problem (OP)}.

We could also  consider a more general
Bolza-type cost functional\vspace{-0.1cm}
$$
\cJ(u(\cdot), \nu_0)=\mE\[\int_0^T \tilde
h(t,x(t),u(t))dt + h(x(T))\].\vspace{-0.1cm}
$$
However, it is well known that such optimal
control problem can be reduced to {\bf Problem
(OP)} by considering an extended control
system:\vspace{-0.1cm}
\begin{equation}\label{8.18-eq21}
\left\{
\begin{array}{ll}
dx(t)=\big(Ax(t)+a(t,x(t),u(t))\big)dt+b(t,x(t),u(t))dW(t) &\mbox{ in }(0,T],\\
\ns\ds d\tilde x(t) = \tilde h(t,x(t),u(t))dt
&\mbox{ in }[0,T],
\\ x(0)=\nu_0\in H,\q \tilde x(0)=0
\end{array}\right.\vspace{-0.1cm}
\end{equation}
with the Mayer type cost
functional\vspace{-0.21cm}
$$
\cJ(u(\cdot),\nu_0)= h(x(T))+\tilde
x(T),\vspace{-0.21cm}
$$
under  constraints\vspace{-0.1cm}
$$
\mE g^0(x(t))\leq 0,\q \mbox{ for all }t\in
[0,T], \q \nu_0\in \cK_a,\q \mE g^j(x(T))\leq
0,\q j=1,\cds,n.\vspace{-0.1cm}
$$

It is one of the important issues in optimal
control theory  to establish necessary
optimality conditions for optimal controls,
which is useful for characterizing optimal
controls  or  solving the optimal control
problems numerically. Since the seminal work
\cite{PC},  necessary optimality conditions are
studied extensively for different kinds of
control systems. We refer the readers to
\cite{FZZ,FZZ2,HP1,LY,Troltzsch,YZ,ZZ} and the
rich references therein for the first and second
order necessary optimality conditions for
systems governed by ordinary differential
equations, by partial differential equations and
by stochastic differential equations.

It is natural  to seek to extend the  theory of
necessary optimality conditions to those
infinite dimensional SEEs. The main motivation
is to study the optimal control of systems
governed by stochastic partial differential
equations, which are useful models for many
processes in natural sciences (see
\cite{Carmona,Kotelenez} and the rich references
therein).

We refer to \cite{Bensoussan2} for a pioneering
work on first order necessary optimality
condition (Pontryagin-type maximum principle)
and subsequent extensions \cite{HP1,TL,Zhou} and
so on. Nevertheless, for a long time, almost all
of the works on the necessary conditions for
optimal controls of infinite dimensional SEEs
addressed only the case that the diffusion term
does NOT depend on the control variable (i.e.,
the function $b(\cd,\cd,\cd)$ in
\eqref{controlsys} is independent of $u$). As
far as we know, the stochastic maximum principle
for general infinite dimensional nonlinear
stochastic systems with control-dependent
diffusion coefficients and possibly nonconvex
control domains had been a longstanding problem
till the very recent papers
(\cite{DM,FHT,LZ1,LZ,LZ2}). In these papers
first order necessary optimality conditions for
controlled SEEs are established by several
authors with no constraint on the state.
Further, in \cite{lv2016,LZZ}, some second order
necessary optimality conditions for controlled
SEEs are obtained, provided that there is no
constraint on the state and $U$ is convex. As
far as we know, there are no results on first or
second order necessary optimality conditions for
controlled SEEs with state constraints and for a
nonconvex set $U$.

Compared with
\cite{DM,FHT,lv2016,LZZ,LZ1,LZ,LZ2}, the main
novelty of the present work is in employing some
sharp tools of set-valued analysis with the
following advantages:
\begin{itemize}
  \item only one adjoint equation is needed to
get a first order necessary optimality condition
even when the diffusion term is control
dependent and $U$ is nonconvex;
  \item  two second order necessary optimality
  conditions
are obtained by using two adjoint equations;
  \item state constraints are presented.
\end{itemize}

The rest of this paper is organized as follows:
in Section \ref{pre}, we introduce some
notations and assumptions  and recall some
concepts and results from the set-valued
analysis to be used in this paper; Section
\ref{first} is devoted to establishing  first
order necessary optimality conditions; at last,
in Section \ref{second}, we obtain two
integral-type second order necessary optimality
conditions.


\section{Preliminaries}\label{pre}
\subsection{Notations and assumptions}
Let $X$ be a Banach space.  For each $t\in[0,T]$
and $r\in [1,\infty)$, denote by
$L_{\cF_t}^r(\Om;X)$ the Banach space of all
(strongly) $\cF_t$-measurable random variables
$\xi:\Om\to X$ such that $\mathbb{E}|\xi|_X^r <
\infty$, with the norm
$|\xi|_{L_{\cF_t}^r(\Om;X)}\=\big(\mE|\xi|_X^r\big)^{1/r}$.
Write $D_{\dbF}([0,T];L^{r}(\Om;X))$ for the
Banach space of all $X$-valued, $r$th power
integrable $\dbF$-adapted processes $\f(\cdot)$
such that $\f:[0,T] \to L^{r}_{\cF_T}(\Omega;X)$
is c\`adl\`ag, with the norm $
|\f(\cd)|_{D_{\dbF}([0,T];L^{r}(\Omega;X))} =
\sup_{t\in
[0,T]}\left(\mE|\f(t)|_X^r\right)^{1/r}$. Write
$C_{\dbF}([0,T];L^{r}(\Omega;X))$ for the Banach
space of all $X$-valued, $\dbF$-adapted
processes $\f(\cdot)$ such that $\f:[0,T] \to
L^{r}_{\cF_T}(\Omega;X)$ is continuous, with the
norm inherited from $D_{\dbF}([0,T];$
$L^{r}(\Omega;X))$.

Denote by $D([0,T];X)$ the Banach space of all
$X$-valued c\`adl\`ag functions $\f(\cd)$ such
that $ \sup_{t\in [0,T]}|\f(t)|_X <\infty$, with
the norm $|\f|_{D([0,T];X)}= \sup_{t\in
[0,T]}|\f(t)|_X$; by $L^2_\dbF(\O;D([0,T];X))$
the Banach space of all $X$-valued
$\dbF$-adapted c\`adl\`ag processes $\f(\cd)$
such that $\mE\big(\sup_{t\in
[0,T]}|\f(t)|_X\big)^{2}<\infty$, with the norm
$|\f|_{L^2_\dbF(\O;D([0,T];X))}=
\big[\mE\big(\sup_{t\in
[0,T]}|\f(t)|_X\big)^{2}\big]^{\frac{1}{2}}$; by
$L^2_\dbF(\O;C([0,T];X))$ the space of all
$\dbF$-adapted continuous processes $\f(\cd)$
such that $\mE\big(\sup_{t\in
[0,T]}|\f(t)|_X\big)^{2}<\infty$, with the norm
inherited from $L^2_\dbF(\O;D([0,T];X))$; by
$L^{2}_{\dbF}(\Omega;BV([0,T]; X))$  the Banach
space of all $X$-valued, $\dbF$-adapted
processes $\f(\cd)$   whose sample paths have
bounded variation, $\dbP$-a.s., such that
$\mE~|\f(\cd,\omega)|_{BV[0,T]}^2 <+\infty$,
with the norm
$\big(\mE|\f(\cd,\omega)|_{BV[0,T]}^2\big)^{\frac{1}{2}}$;
and by $L^{2}_{\dbF}(\Omega;BV_0([0,T]; X))$ the
space  of processes $\f\in
L^{2}_{\dbF}(\Omega;BV([0,T]; X))$ satisfying
$\f(0)=0$, with the norm inherited from
$L^{2}_{\dbF}(\Omega;BV([0,T]; X))$. For any
$\f\in L^{2}_{\dbF}(\Omega;BV_0([0,T]; X))$, one
can find a $\tilde\f\in
L^{2}_{\dbF}(\Omega;BV_0([0,T]; X))\cap
L^2_\dbF(\O;D([0,T];X))$ such that $\f=\tilde
\f$ for a.e. $(t,\omega)\in [0,T]\times\Omega$.
Hence, in this paper, without loss of
generality, any $\f\in
L^{2}_{\dbF}(\Omega;BV_0([0,T]; X))$ can be
considered as an element in
$L^2_\dbF(\O;D([0,T];X))$.

Fix any $r_1,r_2\in[1,\infty]$.
Put\vspace{-0.1cm}
$$
\begin{array}{ll}
\ds
L^{r_1}_\dbF(\Omega;L^{r_2}(0,T;X))=\Big\{\f:(0,T)\times\Om\to
X\;\Big|\; \f(\cd)\hb{
is $\dbF$-adapted and } \dbE\(\int_0^T|\f(t)|_X^{r_2}dt\)^{\frac{r_1}{r_2}}<\infty\Big\},\\
\ns\ds
L^{r_2}_\dbF(0,T;L^{r_1}(\Omega;X))=\Big\{\f:(0,T)\times\Om\to
X\;\Big|\; \f(\cd)\hb{ is $\dbF$-adapted and }
\int_0^T\(\dbE|\f(t)|_X^{r_1}\)^{\frac{r_2}
{r_1}}dt<\infty\Big\}.
\end{array}
$$
Clearly, the above two sets are Banach spaces
with the following norms
respectively\vspace{-0.2cm}
$$
|\phi(\cd)|_{L^{r_1}_\dbF(\Om;L^{r_2}(0,T;X))}\=
\[\mE\(\int_0^T
|\phi(t)|_X^{r_2}dt\)^{\frac{r_1}{r_2}}\]^{\frac{1}{r_1}}
$$
and\vspace{-0.32cm}
$$
|\phi(\cd)|_{L^{r_2}_\dbF(0,T;L^{r_1}(\Om;X))}\=
\[\int_0^T\big(
\mE|\phi(t)|_X^{r_1}dt\)^{\frac{r_2}{r_1}}\]^{\frac{1}{r_2}}.
$$
If $r_1=r_2$, we simply write
$L^{r_1}_\dbF(0,T;X)$ for the above spaces. As
usual, if there is no danger of confusion, we
omit the $\omega$ ($\in \Omega$) argument in the
notations of functions and operators.

Let $H$ be a separable Hilbert space and $A$ be
an unbounded linear operator (with the domain
$D(A)$) on $H$, which generates a contractive
$C_0$-semigroup $\{S(t)\}_{t\geq 0}$ on $H$. It
is well known that $D(A)$ is a Hilbert space
with the usual graph norm. By $A^*$, we denote
the adjoint operator of $A$, which generates
the adjoint $C_0$-semigroup $\{S^*(t)\}_{t\geq
0}$. Denote by $\cL_2$ the space of all
Hilbert-Schmidt operators from $V$ to $H$, which
is a Hilbert space with the canonical norm.

Throughout this paper, we use $C$ to denote a
generic constant, which may change from line to
line.

Let us introduce the following condition:

\medskip

\no{\bf (AS1)} {\it
$a(\cd,\cd,\cd,\cd):[0,T]\times H\times
H_1\times\Om\to H$ and
$b(\cd,\cd,\cd,\cd):[0,T]\times H\times
H_1\times\Om\to \cL_2$ are two maps such that:
i) For any $(x,u)\in H\times H_1$,
$a(\cd,x,u,\cd):[0,T]\times\Om\to H$ and
$b(\cd,x,u,\cd):[0,T]\times\Om\to \cL_2$ are
$\cB([0,T])\times \cF$ measurable and
$\dbF$-adapted; ii) For any $(t,x,\o)\in
[0,T]\times H\times\O$, $a(t,x,\cd,\o):H_1\to H$
and $b(t,x,\cd,\o):H_1\to \cL_2$ are continuous,
and\vspace{-0.1cm}
\begin{equation}\label{ab0}
\left\{\1n
\begin{array}{ll}\ds
|a(t,x_1,u,\o)\! -\! a(t,x_2,u,\o)|_{H} +
|b(t,x_1,u,\o)\! -\! b(t,x_2,u,\o)|_{\cL_2}
\!\leq\! C |x_1-x_2|_{H},\\
\ns\ds\hspace{5.6cm} \forall\, (t,x_1,x_2,u,\o)
\in [0,T] \times  H \times
H \times  H_1 \times \Om ,\\
\ns\ds |a(t,0,u,\o)|_{H} + |b(t,0,u,\o)|_{\cL_2}
\leq C, \qq \forall\, (t,u,\o)\in [0,T] \times
H_1\times\Om.
\end{array}
\right.
\end{equation}}
We have the following result:
\begin{lemma}\label{well lemma s1}
Let {\bf (AS1)} hold. Then the equation
\eqref{controlsys} admits a unique mild
solution. Furthermore, for some $C>0$ and all
$\nu_0\in H$,\vspace{-0.2cm}
$$
|x(\cd)|_{L^{2}_\dbF(\O;C([0,T];H))} \leq
C\big(1+|\nu_0|_{H}\big).\vspace{-0.1cm}
$$
\end{lemma}
\no The proof of Lemma \ref{well lemma s1} can
be found in \cite[Chapter 7]{Prato}.


\subsection{Set-valued
analysis} For readers' convenience, we collect
some basic facts from set-valued analysis. More
information can be found in \cite{Aubin90}.


Let $Z$ (\resp $\wt Z$) be a Banach (\resp
separable Banach) space with the norm $|\cdot
|_{Z}$ (\resp $|\cd|_{\wt Z}$). Denote by $Z^*$
(\resp $\wt Z^*$) the dual space of $Z$ (\resp
$\wt Z$). For any subset $\cK\subset Z$, denote
by  ${\rm int} \cK$ and ${\rm cl} \cK$ the
interior and closure of $\cK$, respectively.
$\cK$ is called a cone if $\alpha z\in \cK$ for
every $\alpha\ge0$ and $z\in \cK$. Define the
distance between a point $z\in Z$ and $\cK$
as\vspace{-0.2cm}
$$
dist\,(z,\cK)\= \inf_{y\in \cK}
|y-z|_{Z}\vspace{-0.2cm}
$$
and the metric projection of $z$ onto $\cK$
as\vspace{-0.2cm}
$$
\Pi_{\cK}(z)\= \{y\in \cK\ |\
|y-x|_{Z}=dist\,(z,\cK)\}.
$$
\begin{definition}
For $z\in \cK$, the Clarke tangent cone
$\cC_{\cK}(z)$ to $\cK$ at $z$ is
$$
\cC_{\cK}(z)\=\Big\{v\in Z\ \Big|\ \lim_{\e\to
0^+, y\in \cK,y\to z} \frac{dist\,(y+\e
v,\cK)}{\e}=0 \Big\}.
$$
and the adjacent cone  $\T_{\cK}(z)$ to $\cK$ at
$z$ is
$$
\T_{\cK}(z)\=\Big\{v\in Z\ \Big|\ \lim_{\e\to
0^+} \frac{dist\,(z+\e v,\cK)}{\e}=0 \Big\}.
$$
\end{definition}
$\cC_{\cK}(z)$ is a closed convex cone in $Z$
and $\cC_{\cK}(z)\subset \T_{\cK}(z)$. When
$\cK$ is convex, $ \cC_{\cK}(z)=
\T_{\cK}(z)={\rm cl} \{\a(\hat z-z)|\a\geq 0,\;
\hat z\in \cK\}$.

\begin{definition}
For $z\in \cK$ and $v\in \T_{\cK}(z)$, the
second order adjacent subset to $\cK$ at $(z,v)$
is defined by\vspace{-0.2cm}
$$
\TT_{\cK}(z,v)\=\Big\{h\in Z\ \Big|\ \lim_{\e\to 0^+}
\frac{dist\,(z+\e v+\e^{2}h,\cK)}{\e^2}=0 \Big\}.
$$
\end{definition}

The dual cone of the Clarke tangent cone
$\cC_{\cK}(z)$, denoted by $\cN^C_{\cK}(z)$, is
called the normal cone of $\cK$ at $z$,
i.e.,\vspace{-0.2cm}
$$
\cN^C_{\cK}(z)\=\Big\{\xi\in Z^*\ \Big|\
\inner{\xi}{v}_{Z^*,Z}\le 0,\ \forall\; v\in
\cC_{\cK}(z) \Big\}.
$$
\begin{definition}
Let $(\Xi, \Si)$ be a measurable space, and
$F:\Xi\leadsto Z$ be a set-valued map. For any
$\xi\in \Xi$, $F(\xi)$ is called the value of
$F$ at $\xi$. The domain of $F$ is $
Dom\,(F)\=\{\xi\in \Xi\ |\ F(\xi)\neq
\emptyset\}. $ $F$ is called measurable if
$F^{-1}(B)\=\{\xi\in \Xi\ |\ F(\xi)\cap B\neq
\emptyset\}\in \Si$ for any $B\in \mb(Z)$, where
$\mb(Z)$ is the Borel $\si$-algebra on $Z$.
\end{definition}
%

%
%
%

%
\begin{lemma}\label{lm7}\cite[Lemma 2.7]{FZZ1}
Suppose that $(\Xi,\Si,\mu)$ is a complete
finite measure space,   $p\ge 1$ and $\cK$ is a
closed nonempty subset of $\wt Z$.
Put\vspace{-0.2cm}
\begin{equation}\label{7.18-eq1}
\dbK\= \big\{\varphi(\cdot)\in L^p(\Xi,\Si,\mu;
\wt Z)\ \big| \ \varphi(\xi)\in \cK, \
\mu\hbox{-a.e.}\ \xi\in \Xi\big\}.
\end{equation}
Then for any $\varphi(\cdot)\in\dbK$, the
set-valued map $\T_{\dbK}(\varphi(\cdot))$:
$\xi\leadsto \T_{\cK}(\varphi(\xi))$ and
$\cC_\dbK(\f(\cd)):\xi\rightsquigarrow
\cC_\cK(\f(\xi))$ are $\Si$-measurable,
and\vspace{-0.2cm}
$$
\big\{v(\cd)\in L^p(\Xi,\Si,\mu;Z)\,|\ v(\xi)\in
T_\cK^b(\f(\xi)),\ \mu\mbox{-a.e. }\xi\in \Xi
\big\}\subset T_\dbK^b(\f(\cd)),
$$
$$
\big\{v(\cd)\in L^p(\Xi,\Si,\mu;Z)\,|\ v(\xi)\in
\cC_\cK(\f(\xi)),\ \mu\mbox{-a.e. }\xi\in \Xi
\big\}\subset \cC_\dbK(\f(\cd)).
$$
\end{lemma}

The following result provides a criteria for the
measurability of set-valued maps.

\medskip
\begin{lemma}\cite[Theorem 8.1.4]{Aubin90}\label{lm2}
Let $(\Xi,\Si,\mu)$ be a complete
$\sigma$-finite measure space and $F$ be a
set-valued map from $\Xi$ to $\wt Z$ with
nonempty closed images. Then $F$ is measurable
if and only if the graph of $F$ belongs to
$\Si\otimes\mb(\wt Z)$.
\end{lemma}
\begin{definition}
We call a  map $\zeta:(\Omega,\mf)\leadsto Z$ a
set-valued random variable if it is measurable.

We call a map $\Psi:[0,T]\times\Omega\leadsto Z$
a measurable set-valued stochastic process if
$\Psi$ is $\mathcal{B}([0,T])\otimes
\mathcal{F}$-measurable.

We say that a measurable set-valued stochastic
process $\Psi$ is $\dbF$-adapted if
$\Psi(t,\cd)$ is $\cF_{t}$-measurable for all
$t\in [0,T]$.
\end{definition}

Let\vspace{-0.2cm}
\begin{equation}\label{adapted sigma field}
\cG\=\big\{B\in \mb([0,T])\otimes\cF\ \big|\
B_{t}\in \cF_{t},\ \forall\; t\in
[0,T]\big\},\vspace{-0.2cm}
\end{equation}
where $B_{t}\=\{\omega\in \Omega\ |\
(t,\omega)\in B\}$ is the section of $B$.
Obviously, $\cG$ is a sub-$\sigma$-algebra of
$\cB([0,T])\otimes\cF$. Denote by $\mathbf{m}$
the Lebesgue measure on $[0,T]$.   The measure
space
$([0,T]\times\Omega,\cG,\mathbf{m}\times\dbP)$
may be incomplete. Let us give a completed
version of it.

Let $\wt\cG$ be the collection of $B\subset
[0,T]\times\Omega$ for which there exist $B_1,\
B_2\in\cG$ such that $ B_1\subset B\subset B_2$
and $(\mathbf{m}\times\dbP)(B_2\setminus
B_1)=0$.
One can define a function $\wt\mu$ on $\wt\cG$
as $\wt\mu(B)=[\mathbf{m}\times\dbP](B_1)$ for
any $B\in \wt\cG$. By Proposition 1.5.1 in
\cite{Cohn13}, the measure space
$([0,T]\times\Omega,\wt\cG,\wt\mu)$ is a
completion of
$([0,T]\times\Omega,\cG,\mathbf{m}\times\dbP)$.

Define\vspace{-0.1cm}
\begin{equation*}\label{7.30-eq1}
\begin{array}{ll}
\ds
\cL^{2}_{\dbF}(0,T;H_1)\!\=\!\Big\{y:[0,T]\!\times\!\Omega\!\rightarrow\!
H_1 \big|
y(\cd)\ \hb{is $\wt\cG$-measurable,} \int_{[0,T] \times \Omega}|y(s,\o)|_{H_1}^2d\wt\mu(s,\o)\!<\!\infty\Big\},\\
\ns\ds \wt\cU_{ad}^{\nu_0}
\=\!\Big\{u:[0,T]\!\times\!\Omega\rightarrow
H_1\ \big|\ u(\cd)\
\hb{is $\wt\cG$-measurable,}\ u(t) \in U, \;\wt\mu\hb{-a.e.}, \mbox{ the corresponding}\\
\ns\ds\qq\qq\qq\qq\qq\qq \mbox{solution $x(\cd)$
of \eqref{controlsys} satisfies
\eqref{constraints1} and
\eqref{constraints}}\Big\}.
\end{array}\vspace{-0.1cm}
\end{equation*}
Clearly, $\cU_{ad}^{\nu_0}
\subset\wt\cU_{ad}^{\nu_0}$ and
$L^{2}_{\dbF}(0,T;H_1)\subset
\cL^{2}_{\dbF}(0,T;H_1).$

\vspace{0.1cm}

Let $\Xi=[0,T]\times\Omega$, $\mu=\wt\mu$ and
$Z=H_1$. From Lemma \ref{lm7}, we deduce the
following result.
\begin{corollary}\label{Lem-adjacent}
For any $u(\cdot)\in \wt\cU_{ad}^{\nu_0}$,
$\cC_{U}(u(\cdot)):[0,T]\times\Omega\rightsquigarrow
H_1$ is $\wt\cG$-measurable and \vspace{-0.2cm}
\begin{equation}\label{cT-bar-u*}
\cT_u\=\big\{v(\cdot)\in
\cL^{2}_{\dbF}(0,T;H_1)\ \big|\ v(t)\in
\cC_{U}(u(t)),\ \wt\mu \mbox{-a.e.}\big\}\subset
\cC_{\wt\cU_{ad}^{\nu_0}}(u(\cdot)).
\end{equation}
\end{corollary}

The next result concerns the completion of a
measure space, which is a corollary of
Proposition 1.5.1 in \cite{Cohn13}.
\begin{lemma}\label{lm3}
Let $(\Xi, \Si, \mu)$ be a $\sigma$-finite
measure space with the completion $(\Xi, \wt\Si,
\wt\mu)$, and $f$ be a $\wt\Si$-measurable
function from $\Xi$ to $Z$. Then there exists a
$\Si$-measurable function $g$ such that
$\wt\mu(g(\xi)\neq f(\xi))=0$.
\end{lemma}

Due to Lemma \ref{lm3}, in what follows, we
omit\; $\wt{ }$\; to simplify notation.

\begin{lemma}\label{lm10}
Let $H$ be a separable Hilbert space. A
set-valued stochastic process
$F:[0,T]\times\Omega\leadsto H$ is
$\cB([0,T])\otimes\cF$-measurable and
$\dbF$-adapted if and only if $F$ is
$\cG$-measurable.
\end{lemma}
\begin{proof}
Since $H$ is separable, it has an orthonormal
basis $\{e_k\}_{k=1}^\infty$. Denote by $\G_k$
the projection operator from $H$ to
$H_k\=\span\{e_k\}$. Let $F_k(\cd)=\lan
F(\cd),e_k\ran_H$. From \cite[p.
96]{Kisielewicz}, we know that the set-valued
stochastic process
$F_k:[0,T]\times\Omega\leadsto \dbR$ is
$\cB([0,T])\otimes\mf$-measurable and
$\dbF$-adapted if and only if $F_k$ is
$\cG$-measurable. Then Lemma \ref{lm10} follows
from the fact that $\ds F(\cd) =
\sum_{k=1}^\infty F_k(\cd) e_k$.
\end{proof}

Next, we recall the notion of measurable
selection for a set-valued map.

\begin{definition} \label{Def-mea-sel}
Let $(\Xi,\Si)$ be a measurable space and $\wh
Z$ a complete separable metric space. Let $F$ be
a set-valued map  from $\Xi$ to $\wh Z$. A
measurable map $f:\Xi\to \wh Z$ is called a
measurable selection of $F$ if $f(\xi)\in
F(\xi)$ for all $\xi\in\Xi$.
\end{definition}
A result concerning the measurable selection is
given below.

\begin{lemma} \cite[Theorem 8.1.3]{Aubin90}\label{Pro-mea-sel}
Let $\wh Z$ be a complete separable metric
space, $(\Xi,\Si)$ a measurable space, and
$F:\Xi\leadsto \wh Z$ a measurable set-valued
map with nonempty closed values. Then there
exists a measurable selection of $F$.
\end{lemma}

The following result is a special case of
\cite[Corollary 8.2.13]{Aubin90}.
\begin{lemma}\label{metric projection}
Suppose that $(\Xi,\Si,\mu)$ is a complete
$\sigma$-finite measure space,  $\cK$ is a
closed nonempty subset in $\wt Z$ and
$\varphi(\cdot)$ is a $\Si$-measurable map from
$\Xi$ to $\wt Z$. Then the projection map
$\xi\rightsquigarrow \Pi_{\cK}(\varphi(\xi))$ is
$\Si$-measurable. If
$\Pi_{\cK}(\varphi(\xi))\neq\emptyset$ for all
$\xi\in\Xi$, then there exists a
$\Si$-measurable, $\wt Z$-valued selection
$\psi(\cdot)$ such that
$|\psi(\xi)-\varphi(\xi)|_{\wt
Z}=dist\,(\varphi(\xi), \cK)$, $\mu$-a.e.
\end{lemma}

At last, let us recall some results concerning
convex cones.

\begin{definition}
For a cone $\cK$ in  $Z$, the convex closed cone
$\cK^-= \{\xi\in Z^*| \xi(z) \leq 0 \mbox{ for
all }z\in \cK\}$ is called the dual cone of
$\cK$.
\end{definition}
\vspace{-0.2cm}
\begin{lemma}\label{lm6}\cite[Lemma 2.4]{FZZ1}
Let $m\in \dbN$. Let $\cK_1,\cds,\cK_m$ be
convex cones in $\wt Z$ and $\ds\bigcap_{j=1}^m
{\rm int}\cK_j\neq\emptyset$. Then for any
convex cone $\cK_0$ such that $\ds
\cK_0\bigcap\Big(\bigcap_{j=1}^m
\cK_j\Big)\neq\emptyset$, we have
$\ds\(\bigcap_{j=0}^m \cK_j\)^-=\sum_{j=0}^m
\cK_j^-$.
\end{lemma}
\begin{definition}
We call $\cK$  a nonempty closed polyhedra in
$Z$ if for some $n\in\dbN$, $\{
z_1^*,\cds,z_n^*\}\subset  Z^*\setminus\{0\}$
and $\{b_1,\cds,b_n\}\subset
\dbR$,\vspace{-0.2cm}
$$
\cK\=\{y\in \wt Z|\ \langle y_j,y\rangle_{\wt Z}
+ b_j\leq 0,\; \forall\ j=1,\cds,n\}.
$$
\end{definition}
\begin{lemma}\label{lm5}\cite[Lemma 2.5]{FZZ1}
Let $\wh Z$ be a Hilbert space. Let $\cK$ be a
nonempty closed polyhedra in $\wh Z$. Then, for
any $0\neq\xi\in \wh Z$ such that $\sup_{y\in
\cK}\langle \xi,y\rangle_{\wh Z}<+\infty$, this
supremum is attained at some $\bar y\in \pa
\cK$. Furthermore, $\xi\in \sum_{j\in In(\bar
y)}\dbR^+ y_j$, where\vspace{-0.2cm}
$$
In(\bar y)\=\big\{j\in\{1,\cds,n\}|\langle
y_j,y\rangle_{\wh Z} + b_j=0\big\}.
$$
\end{lemma}

\begin{lemma}\label{lm9}
Let $M_0, M_1,\ldots,M_n$ be nonempty  convex
subsets of $Z$ such that $M_j$ is open for all
$j \in\{ 1,\cdots,n\}$. Then\vspace{-0.2cm}
\begin{equation}\label{abstr}
M_0 \cap  M_1\cap\ldots\cap M_n =\emptyset
\end{equation}
if and only if there are $z_0^*,
z_1^*,\cdots,z_n^*\in Z^*$, not vanishing
simultaneously, such that\vspace{-0.2cm}
\begin{equation}\label{abstr1} z_0^* + z_1^*+\cdots+z_n^*=0,  \quad
\inf_{z\in M_0} z_0^*(z) + \inf_{z\in M_1}
z_1^*(z)+\cdots+\inf_{z\in M_n} z_n^*(z)\geq
0.\vspace{-0.32cm}
\end{equation}
Furthermore, if (\ref{abstr1}) holds true  and
for some $j\in \{0,\ldots,n\}$ there is  a
nonempty cone $\cK_j \subset Z$  and $z_j \in Z$
such that  $z_j + \cK_j \subset M_j$, then
$-z_j^* \in \cK_j^-$.
\end{lemma}
Proof of the above lemma can be found in
\cite{HFNO2017a}.


\section{First order necessary
conditions}\label{first}


This section is devoted to  establishing a first
order necessary optimality condition for {\bf
Problem (OP)}. Let us first impose the following
assumptions:

\medskip

{\bf (AS2)}  {\em For a.e. $(t,\omega)\in [0,
T]\times\Omega$, the functions $a(t, \cdot,
\cdot,\omega):\ H\times H_1\to H$ and $b(t,
\cdot, \cdot,\omega):\ H\times H_1\to \cL_2$ are
differentiable, and $
(a_{x}(t,x,u,\omega),a_{u}(t,x,u,\omega))$  and
$(b_{x}(t,x,u,\omega),$ $b_{u}(t,x,u,\omega))$
are uniformly continuous with respect to $x\in
H$ and $u\in U$. For any $p\geq 1$, there exists
a nonnegative $\eta\in L^{2}_{\dbF}(0,T;\dbR)$
such that for a.e. $(t,\omega)\in [0,
T]\times\Omega$ and for all $x\in H$ and $u\in
H_1$,}
$$
\left\{\2n
\begin{array}{ll}
|a(t,0, u,\omega)|_H + |b(t,0,
u,\omega)|_{\cL_2} \leq
C (\eta(t,\omega)+|u|_{H_1}), \\
\ns\ds
|a_{x}(t,x,u,\omega)|_{\cL(H)}\!+\!|a_{u}(t,x,u,\omega)|_{\cL(H_1;H)}\!+\!|b_{x}(t,x,u,\omega)|_{\cL(H;\cL_2)}\!+\!|b_{u}(t,x,u,\omega)|_{\cL(H_1;\cL_2)}\le
C.
\end{array}
\right.
$$

\medskip

{\bf (AS3)}   {\em The functional
$h(\cdot,\omega):\ H \to \dbR$ is
differentiable, $\dbP$-a.s.,  and there exists
an $\eta \in L^2_{\cF_T}(\Omega)$ such that for
any $x,\ \tilde{x}\in H$,}\vspace{-0.1cm}
$$
\left\{
\begin{array}{ll}\ds
|h(x,\omega)|\le C(\eta(\omega)^2+|x|_H^2),\qq
|h_{x}(0,\omega)|_H \le C\eta(\omega), \q
\dbP\mbox{-a.s.,}\\
\ns\ds|h_{x}(x,\omega)-h_{x}(\tilde{x},\omega)|_H\le
C|x-\tilde{x}|_H, \q \dbP\mbox{-a.s.}
\end{array}\vspace{-0.2cm}\right.
$$

\medskip

{\bf (AS4)}   {\em For $j=0,\cds,n$, the
functional $g^j:\ H \to \dbR$ is differentiable,
and   for any $x,\ \tilde{x}\in
H$,}\vspace{-0.2cm}
$$
\begin{array}{ll}\ds
|g^j(x)|\le C(1+|x|_H^2),\qq
|g^j_{x}(x)-g^j_{x}(\tilde{x})|_H\le
C|x-\tilde{x}|_H.\vspace{-0.2cm}
\end{array}
$$

Let $\Phi$ be a set-valued stochastic process
satisfying
\begin{enumerate}
  \item $\Phi$  is
$\cB([0,T])\otimes\cF$-measurable and
$\dbF$-adapted;
  \item  for  a.e. $(t,\omega)\in
[0,T]\times \Omega$, $\Phi(t,\omega)$ is a
nonempty closed convex cone in $H_1$;
  \item $\Phi(t,\omega)\subset T^b_{U}(\bar
u(t,\omega))$, for a.e. $(t,\omega)\in
[0,T]\times \Omega$.
\end{enumerate}

Let\vspace{-0.2cm}
$$
\cT_\Phi(\bar u)\=\Big\{u(\cdot)\in
L^2_{\dbF}(0,T;H_1)\,\big|\, u(t,\omega)\in
\Phi(t,\omega),\; \mbox{a.e. } (t,\omega)\in
[0,T]\times \Omega\Big\}.
$$
Clearly, $\cT_\Phi(\bar u)$ is a closed convex
cone in $L^2_{\dbF}(0,T;H_1)$. Since $0 \in
\cT_\Phi(\bar u)$, $\cT_\Phi(\bar u)$ is
nonempty. By Lemma \ref{lm7}, we can choose
$\Phi(t,\omega)=\cC_{U}(\bar u(t,\omega))$.
However, in general, there may exist a
$\Phi(t,\omega)$ as above such that
$\cC_{U}(\bar u(t,\omega))\subsetneq
\Phi(t,\omega) \subset  T^b_{U}(\bar
u(t,\omega))$.

For $\varphi$ equal to $a$, $b$, $f$, $g$ or
$h$, write\vspace{-0.2cm}
$$
\varphi_{1}[t]=\varphi_{x}(t,\bar{x}(t),\bar{u}(t)),\quad
\varphi_{2}[t]=\varphi_{u}(t,\bar{x}(t),\bar{u}(t)).\vspace{-0.2cm}$$
Consider the following linearized stochastic
control system:\vspace{-0.1cm}
\begin{equation}\label{first vari equ}
\left\{\3n
\begin{array}{ll}
dx_{1}(t)= \big(Ax_{1}(t)+a_{1}[t] x_{1}(t)+a_{2}[t]u_1(t)\big)dt+\big(b_{1}[t] x_{1}(t)+ b_{2}[t]u_1(t)\big)dW(t) \mbox{ in }(0,T], \\
x_{1}(0)=\nu_1.
\end{array}\right.\vspace{-0.1cm}
\end{equation}
It is a classical result that, under
\textbf{(AS1)}, for any $u_1\in \cT_\Phi(\bar
u)$ and $\nu_1\in T^b_{\cK_{a}}(\bar x_{0})$,
(\ref{first vari equ}) admits a unique solution
$x_{1}(\cdot)\in L_{\dbF}^{2}(\Omega; C([0,T];
H))$ (e.g.\cite[Chapter 6]{Prato}).

By Lemma \ref{lm7}, $\cT_\Phi(\bar u)\subset
T^b_{\cU}(\bar u)$. For any $\e>0$, choose
$\nu_1^{\e}\in H$ and $v^{\e}\in
L^{2}_{\dbF}(0,T;H_1)$ such that\vspace{-0.2cm}
$$
\nu^{\e}_{0}\=\bar \nu_0+\e \nu_1^{\e}\in
\cK_a,\q u^{\e}\=\bar u+\e
v^{\e}\in\mathcal{U}_{ad}
$$
and\vspace{-0.32cm}
$$
\nu_1^{\e}\to  \nu_1 \mbox{ in }H \mbox{ and }
v^{\e}\to u_1\mbox{ in }L^{2}_{\dbF}(0,T;H_1)
\mbox{ as }\e\to 0^+.
$$
Let $x^{\e}(\cd)$ be the solution of
(\ref{controlsys}) corresponding to the control
$u^{\e}(\cd)$ and the initial datum
$\nu_{0}^{\e}$, and put\vspace{-0.2cm}
\begin{equation}\label{7.2-eq5}
\delta x^{\e}(\cd)=x^{\e}(\cd)-\bar x(\cd),
\qquad r_{1}^{\e}(\cd)\= \dfrac{\delta
x^{\e}(\cd)- \e x_{1}(\cd)}{\e}.
\end{equation}

We have the following results:
\begin{lemma}\label{estimate one of varie qu}
If {\bf (AS1)}--{\bf (AS2)} hold, then for
$p\geq 2$,\vspace{-0.1cm}
\begin{equation}\label{8.18-eq41}
|x_{1}|_{L^\infty_\dbF(0,T;L^p(\Om;H))}\le
C\big(|\nu_1|_H+|u_1|_{L^p_\dbF(\Om;L^2(0,T;H))}\big),\vspace{-0.12cm}
\end{equation}
\begin{equation}\label{8.18-eq42}
|\delta x^{\e}|_{L^\infty_\dbF(0,T;L^p(\Om;H))}=
O(\e),\vspace{-0.12cm}
\end{equation}
and\vspace{-0.2cm}
\begin{equation}\label{r1 to 0}
\lim_{\e\to
0^+}|r_{1}^{\e}|_{L^\infty_\dbF(0,T;L^p(\Om;H))}=0.
\end{equation}
\end{lemma}
Proof of Lemma \ref{estimate one of varie qu} is
provided in Appendix A.

\medskip

Next, we give a result  which is very useful to
get the first order pointwise necessary
condition.
\begin{lemma} \label{lm8}
Let  $\tilde{u}(\cdot)\in \cU_{ad}^{\nu_0}$, and
$F:[0,T]\times\Omega\to H_1$ be an
$\dbF$-adapted process such that\vspace{-0.2cm}
$$
\mE\int_{0}^{T}\inner{ F(t )}{v(t )}_{H_1}dt\leq
0,\quad \forall\; v(\cdot)\in
\cC_{\cU_{ad}^{\nu_0}}(\tilde{u}(\cdot)).
$$
Then, for a.e. $(t,\omega)\in
[0,T]\times\Omega$, $
\inner{F(t,\omega)}{v}_{H_1} \le0$,  $\forall\;
v\in \cC_{U}(\tilde {u}(t,\omega))$.
\end{lemma}
Proof of Lemma \ref{lm8} is postponed to
Appendix C.
\begin{lemma}\label{lm12}
For each bounded linear functional $\Lambda$ on
$L^{2}_{\dbF}(\Omega; C([0,T]; H))$, there
exists a process $ \psi\in
L^{2}_{\dbF}(\Omega;BV_0([0,T]; H))$ such
that\vspace{-0.1cm}
\begin{equation}\label{(1)}
\Lambda(z(\cdot))=\mE\int_{0}^{T}\inner{z(t)}{d\psi(t)}_H,\q
\forall\; z(\cdot)\in L_{\dbF}^{2}(\Omega;
C([0,T]; H)),\vspace{-0.1cm}
\end{equation}
and\vspace{-0.2cm}
\begin{equation}\label{(2)}
|\Lambda|_{L_{\dbF}^{2}(\Omega;
C([0,T];H))^{*}}= |
\psi|_{L^{2}_{\dbF}(\Omega;BV([0,T];H))}.
\end{equation}
\end{lemma}
Proof  of Lemma  \ref{lm12} is given in Appendix
D.

\medskip

Let $\cT_{\cK_{a}}(\bar \nu_0)$ be a nonempty
closed convex cone contained in
$T^b_{\cK_{a}}(\bar \nu_0)$. Put\vspace{-0.2cm}
\begin{equation}\label{1st variation set}
\cG^{(1)}\=\!\big\{x_{1}(\cdot)\!\in\!L_{\dbF}^{2}(\Omega;
C([0,T]; H)) \,\big|\, x_{1}(\cd) \mbox{ solves
(\ref{first vari equ}) with}\, u_1\!\in\!
\cT_{\Phi}(\bar u)\mbox{ and }
\nu_1\!\in\!\cT_{\cK_{a}}(\bar \nu_0)
\big\},\vspace{-0.1cm}
\end{equation}
\begin{equation}\label{index I0 bx}
\cI_0(\bar x)\=\{t\in[0,T]\,|\, \mE g^0(\bar
x(t))=0\},
\end{equation}
$$
\cI(\bar x)=\big\{j\in\{1,\cdots, n\}\,|\, \mE
g^{j}(\bar x(T))=0 \big\},
$$
\begin{equation}\label{Q1}
\cQ^{(1)}=\big\{z(\cdot) \in
L_{\dbF}^{2}(\Omega; C([0,T];H))\;\big|\;
\mE\inner{g^0_{x}(\bx(t))}{z(t)}_{H}<0,\;
\forall\; t\in \cI_0(\bar x)\big\},
\end{equation}
\begin{equation}\label{Ed1.1}
\cE^{(1,j)}\=\big\{z(\cdot) \in
L_{\dbF}^{2}(\Omega; C([0,T];H))\;\big|\;
\mE\inner{g^{j}_{x}(\bar
x(T))}{z(T)}_{H}<0\big\}, \qq \forall j\in
\cI(\bar x).\vspace{-0.12cm}
\end{equation}
\begin{equation}\label{Ed1}
\cE^{(1)}\= \bigcap_{j\in \cI(\bar
x)}\cE^{(1,j)},\vspace{-0.12cm}
\end{equation}
\begin{equation}\label{EdjT}
\cE^{(1,j)}_T\=\big\{\zeta \in
L_{\cF_T}^{2}(\Omega;H)\;\big|\;
\mE\inner{g^{j}_{x}(\bar
x(T))}{\zeta}_{H}<0\big\},\q j\in \cI(\bar
x),\vspace{-0.12cm}
\end{equation}
\begin{equation}\label{EdT}
\cE^{(1)}_T\=\bigcap_{j\in \cI(\bar
x)}\cE^{(1,j)}_T,\vspace{-0.12cm}
\end{equation}
$$
\cL^{(1)}\=\{z(\cdot) \in L_{\dbF}^{2}(\Omega;
C([0,T];H))|\, \mE \langle h_x(\bar
x(T)),z(T)\rangle_{H}<0\},\vspace{-0.12cm}
$$
and\vspace{-0.12cm}
$$
\cL^{(1)}_T\=\{z\in L^2_{\cF_T}(\Om;H)\,|\, \mE
\langle h_x(\bar x(T)),z\rangle_{H}<0\}.
$$

Since  $\cT_\Phi(\bar u)$ and
$\cT_{\cK_{a}}(\bar \nu_0)$ are nonempty convex
cones, $\cG^{(1)}$ is a nonempty convex cone in
$L_{\dbF}^{2}(\Omega; C([0,T];H))$.

If $\cI_0(\bar x)=\emptyset$ ({\it resp.}
$\cI(\bar x)=\emptyset$), we set
$\cQ^{(1)}=L_{\dbF}^{2}(\Omega; C([0,T];H))$
(\resp $\cE^{(1)}=L_{\dbF}^{2}(\Omega;C([0,T];$
$ H))$). If $h_x(\bar x(T))=0$, $\dbP$-a.s.,
then $\cL^{(1)}=\emptyset$ and
$\cL^{(1)}_T=\emptyset$.

Define a map $\Gamma: L_{\dbF}^{2}(\Omega;
C([0,T]; H))\to  L^{2}_{\mf_{T}}(\Omega;H)$
as\vspace{-0.2cm}
\begin{equation}\label{gamma C to L2}
\Gamma(z)=z(T), \q \forall\, z(\cd)\in
L_{\dbF}^{2}(\Omega; C([0,T];
H)).\vspace{-0.2cm}
\end{equation}
Denote by  $\Gamma^*$ the adjoint operator of
$\Gamma$. Clearly, $\Gamma$ is surjective. From
\eqref{Ed1.1} to \eqref{EdT}, we see
that\vspace{-0.12cm}
\begin{equation}\label{11.20-eq4}
\cE^{(1,j)}_T = \G(\cE^{(1,j)}), \qq  j\in
\cI(\bar x),\qq \cE^{(1)}_T =
\G(\cE^{(1)}).\vspace{-0.12cm}
\end{equation}

If $\cQ^{(1)}$ and $\cE^{(1)}$ are nonempty,
then\vspace{-0.1cm}
$$
{\rm cl} \cQ^{(1)} = \big\{z(\cdot)\in
L_{\dbF}^{2}(\Omega; C([0,T]; H))\,\big|\,
\mE\inner{g_{x}(\bx(t))}{ z(t)}_H\leq 0,\;
\forall\; t\in \cI_0(\bar
x)\big\},\vspace{-0.2cm}
$$
and\vspace{-0.2cm}
$$
{\rm cl} \cE^{(1)} =\big\{z(\cdot) \in
L_{\dbF}^{2}(\Omega; C([0,T]; H))\;\big|\;
\mE\inner{g^{j}_{x}(\bar x(T))}{z(T)}_H\le 0,\;
\forall\; j\in \cI(\bar x)\big\}.
$$
\begin{lemma}\label{lm11}
$\cQ^{(1)}$ is an open convex cone in
$L_{\dbF}^{2}(\Omega; C([0,T];H))$.
\end{lemma}
\begin{proof}
Clearly, $\cQ^{(1)}$ is a cone. It is sufficient
to prove that it is open.

Let $z(\cd)\in\cQ^{(1)}$. Since $\bar x(\cd)\in
L_{\dbF}^{2}(\Omega; C([0,T]; H))$, $\cI_0(\bar
x)$ is a compact subset of $[0,T]$. This,
together with the fact that
$\mE\inner{g_{x}(\bar x(\cd))}{z(\cd)}_H$ is
continuous with respect to $t$, implies that
there exists a constant $\rho>0$ such
that\vspace{-0.2cm}
$$
\mE\inner{g_{x}(\bar x(t))}{z(t)}_H<-\rho,\qq
\forall\, t\in \cI_0(\bar x).
$$
Let\vspace{-0.32cm}
$$
\d=\frac{\rho}{2|g_{x}(\bar
x(\cd))|_{L^\infty_\dbF(0,T;L^2(\Om;H))}}.
$$
Then for any $\eta\in L_{\dbF}^{2}(\Omega;
C([0,T]; H))$ with
$\|\eta\|_{L_{\dbF}^{2}(\Omega; C([0,T]; H))}\le
\delta$,\vspace{-0.1cm}
$$
\mE\inner{g_{x}(\bar x(t))}{z(t)+\eta(t)}_H
<-\frac{\rho}{2},\q \forall\, t\in \cI_0(\bar
x).\vspace{-0.2cm}
$$
This proves that $z\in {\rm int}\cQ^{(1)}$.
\end{proof}

Now we introduce the first order adjoint
equation for \eqref{first vari
equ}:\vspace{-0.1cm}
\begin{equation}\label{first ajoint equ}
\left\{
\begin{array}{ll}
dy(t)=-\big(A^*y(t)+a_{1}[t]^{*}y(t)
+b_{1}[t]^{*}Y(t)
\big)dt+d\psi(t)+Y(t)dW(t) &\mbox{ in }[0,T), \\
y(T)=y_T,
\end{array}
\right.\vspace{-0.1cm}
\end{equation}
where $y_T\in L^2_{\cF_T}(\Om;H)$ and $ \psi\in
L^{2}_{\dbF}(\Omega;BV_0([0,T]; H))$.

Since neither the usual natural filtration
condition nor the quasi-left continuity is
assumed for the filtration $\dbF$ in this paper,
one cannot apply the existence results for mild
or weak solution of infinite dimensional BSEEs
(e.g. \cite{HP2,MY}) to obtain the
well-posedness of the equation \eqref{first
ajoint equ}. Thus, we use the notion of
transposition solution here. To this end,
consider the following (forward)
SEE:\vspace{-0.1cm}
\begin{eqnarray}\label{fsystem2}
\left\{
\begin{array}{lll}\ds
d\phi(s) = \big(A\phi(s) + f_1(s)\big)ds +  f_2(s) dW(s) &\mbox{ in }(t,T],\\
\ns\ds \phi(t)=\eta,
\end{array}
\right.\vspace{-0.1cm}
\end{eqnarray}
where $t\in[0,T]$, $f_1\in
L^1_{\dbF}(t,T;L^{2}(\O;H))$, $f_2\in
L^2_{\dbF}(t,T;\cL_2)$, $\eta\in
L^{2}_{\cF_t}(\O;H)$ (See \cite[Chapter
6]{Prato} for the well-posedness of
\eqref{fsystem2} in the sense of mild solution).
We now introduce the following notion.

\begin{definition}\label{definition1}
We call $(y(\cdot), Y(\cdot)) \in
D_{\dbF}([0,T];L^{2}(\O;H)) \times
L^2_{\dbF}(0,T;\cL_2)$  a transposition solution
of \eqref{first ajoint equ} if for any $t\in
[0,T]$, $f_1(\cdot)\in
L^1_{\dbF}(t,T;L^2(\O;H))$, $f_2(\cdot)\in
L^2_{\dbF}(t,T;\cL_2)$, $\eta\in
L^2_{\cF_t}(\O;H)$ and the corresponding
solution $\phi\in L^2_{\dbF}(\O;C([t,T];H))$ to
the equation (\ref{fsystem2}), we
have\vspace{-0.2cm}
\begin{equation}\label{eq def solzz}
\begin{array}{ll}\ds
\q\dbE \big\langle \phi(T),y(T)\big\rangle_{H}+
\dbE\int_t^T \big\langle
\phi(s),a_{1}[s]^{*}y(s)
+b_{1}[s]^{*}Y(s))\big\rangle_Hds \\
\ns\ds = \dbE
\big\langle\eta,y(t)\big\rangle_H\! +
\dbE\int_t^T\! \big\langle
f_1(s),y(s)\big\rangle_H ds + \dbE\int_t^T\!
\big\langle f_2(s),Y(s)\big\rangle_{\cL_2} ds+
\dbE\int_t^T\! \big\langle \phi(s),d\psi(s)
\big\rangle_H.
\end{array}
\end{equation}
\end{definition}
\begin{lemma}\label{lm13}
Assume that {\bf(AS1)}--{\bf(AS2)} hold  and $
\psi\in L^{2}_{\dbF}(\Omega;BV_0([0,T]; H))$.
Then the equation \eqref{first ajoint equ}
admits a unique transposition solution $(y,Y)\in
D_{\dbF}([0,T];L^{2}(\Omega; H))\times
L_{\dbF}^{2}(0,T;\cL_2)$.
\end{lemma}

If $\psi=0$ and $W(\cd)$ is a one dimensional
Brownian motion, Lemma \ref{lm13} is proved in
\cite[Chapter 3]{LZ1}. The proof for the case
$\psi\neq 0$ is similar. We only give a sketch
in Appendix \ref{sec-ap-E}.

Define the Hamiltonian\vspace{-0.2cm}
\begin{equation}\label{Hamiltonianconvex}
\dbH(t,x,u, p,q,\omega)
\=\inner{p}{a(t,x,u,\omega)}_H+
\inner{q}{b(t,x,u,\omega)}_{\cL_2},
\end{equation}
where $(t,x,u,p,q,\omega)\in [0,T]\times H\times
H_1\times H\times \cL_2\times\Omega$.

\medskip

Now we state a first order necessary optimality
condition  in the integral form.

\begin{theorem}\label{th1}
Let {\bf(AS1)}--{\bf(AS4)} hold and $(\bar
x(\cd),\bar u(\cd),\bar \nu_{0})$ be an  optimal
triple for {\bf Problem (OP)}. If
$\mE|g^0_{x}(\bar x(t))|_H\neq 0$ for any $t\in
\cI_0(\bar x)$, then there exist $\lambda_{0}\in
\{0,1\}$, $\lambda_{j} \geq 0$ for  $j \in
\cI(\bar x)$ and $\psi\in
\big(\cQ^{(1)}\big)^{-}$ with $\psi(0)=0$
satisfying\vspace{-0.2cm}
\begin{equation}\label{6.23-eq2}
\lambda_{0}+\sum_{j\in \cI(\bar
x)}\lambda_j+|\psi|_{L^2_\dbF(\Om;BV(0,T;H))}\neq
0,\vspace{-0.2cm}
\end{equation}
such that the corresponding transposition
solution $(y(\cd),Y(\cd))$ of the first order
adjoint equation (\ref{first ajoint equ}) with
$\ds y(T)=-\l_0 h_x(\bar x(T))-\sum_{j\in
\cI(\bar x)}\l_j g^j_x(\bar x(T))$ verifies
that\vspace{-0.2cm}
\begin{equation}\label{first order integraltype condition}
\mE\inner{y(0)}{\nu}_H+\mE\int_{0}^{T}
\inner{\dbH_{u}[t]}{v(t)}_{H_1}dt\le
0,\;\;\forall\; \nu\in\cT_{\cK_{a}}(\bar
\nu_{0}),\;\; \forall\; v(\cdot)\in\cT_\Phi(\bar
u),
\end{equation}
where $ \dbH_{u}[t]=\dbH_{u}(t,\bar x(t),\bar
u(t), y(t),Y(t),\omega)$. In addition, if
$\cQ^{(1)} \cap \cG^{(1)}\cap \cE^{(1)}
\neq\emptyset$, the above holds with
$\lambda_0=1$.
\end{theorem}

\begin{proof} We first claim that\vspace{-0.2cm}
\begin{equation}\label{contodicition phiy133T}
\cG^{(1)}\cap
\cQ^{(1)}\cap\cE^{(1)}\cap\cL^{(1)}=\emptyset.\vspace{-0.2cm}
\end{equation}
If this is not the case, then there would exist
$\tilde x_{1}(\cd)\in \cG^{(1)}\cap
\cQ^{(1)}\cap\cE^{(1)}$ such that\vspace{-0.1cm}
\begin{equation}\label{contodicition phiy1T}
\mE \inner{h_{x}(\bx(T))}{\tilde
x_{1}(T)}_H<0.\vspace{-0.1cm}
\end{equation}
Let
$\tilde{\nu}_{1}\in\cT_{\cK_{a}}(\bar\nu_{0})$
be the initial datum and $\tilde u_1(\cdot)\in
\cT_\Phi(\bu(\cdot))$ the control corresponding
to $\tilde x_{1}(\cd)$. Let $\mu^\eps \in H$
with $|\mu^\eps|=o(\eps)$ and $\eta^\eps(\cd)\in
L_{\dbF}^{2}(0,T; H_1)$ with
$|\eta^\eps|_{L_{\dbF}^{2}(0,T; H_1)}=o(\eps)$
be such that\vspace{-0.2cm}
$$
\nu^{\eps}_{0}\=\bar\nu_0+\eps\tilde\nu_1+\mu^\eps\in
\cK_a,\qq u^{\eps}(\cd)\=\bu(\cd)+\eps\tilde
u_1(\cd)+\eta^\eps(\cd)\in
\mathcal{U}_{ad}.\vspace{-0.2cm}
$$
Let $x^{\eps}(\cd)$ be the solution of the
control system (\ref{controlsys}) with  the
initial datum $\nu^{\eps}_{0}$ and the control
$u^{\eps}(\cd)$.

Since  $\tilde x_{1}(\cd)\in \cQ^{(1)}$, we know
that $\mE\inner{g_{x}^0(\bx(\cd))}{\tilde
x_{1}(\cd)}_H$ is continuous with respect to
$t$. This, together with the compactness of
$\cI^0(\bar x)$, implies that there exists
$\rho_{0}>0$ such that\vspace{-0.2cm}
$$
\mE\inner{g_{x}^0(\bx(t))}{\tilde x_{1}(t)}_H<
-\rho_{0} \mbox{  for every } t\in \cI^0(\bar
x).
$$
Moreover, there exists $\delta>0$ (independent
of $t\in \cI^0(\bar x)$) such
that\vspace{-0.2cm}
$$
\mE\inner{g^0_{x}(\bx(s))}{\tilde x_{1}(s)}_H<
-\frac{\rho_{0}}{2}, \forall\, s\in
(t-\delta,t+\delta)\cap [0,T] \mbox{ and } t\in
\cI^0(\bar x).
$$
By Lemma \ref{estimate one of varie qu}, there
is an $\eps_{0}>0$  such that for every $\eps
\in [0,\eps_{0}]$,\vspace{-0.2cm}
\begin{equation}\label{bar y1 admissible eq1}
\begin{array}{ll}\ds
\mE g^0(x^{\eps}(s))\3n&=\ds\mE g(\bx(s))
+\eps \mE\inner{g_{x}(\bx(s))}{\tilde x_{1}(s)}_H
+o(\eps) \\
\ns&\leq\ds\eps\,\mE\inner{g_{x}(\bx(s))}{\tilde x_{1}(s)}_H+o(\eps) \\
\ns&<\ds-\frac{\eps\rho_{0}}{4}<0,\q \forall\;
s\in (t-\delta,t+\delta)\cap[0,T],\;\; t\in
\cI^0(\bar x).
\end{array}\vspace{-0.32cm}
\end{equation}
Since $\cI_{\delta}^{c}\=[0,T]\setminus
\bigcup_{t\in \cI^0(\bar x)}(t-\delta,t+\delta)$
is compact, there exist $\rho_{1}>0$ and
$\eps_{1}>0$ such that for any $\eps \in
[0,\eps_{1}]$,
\begin{equation}\label{bar y1 admissible eq2}
\begin{array}{ll}\ds
\mE g(x^{\eps}(t))\3n&=\ds\mE g(\bx(t))
+\eps\mE\inner{g_{x}(\bx(t))}{\tilde x_{1}(t)}_H+o(\eps) \\
\ns &<\ds-\rho_{1}+\eps\mE\inner{g_{x}(\bx(t))}{\tilde x_{1}(t)}_H+o(\eps) \\
\ns&<\ds-\frac{\rho_{1}}{2}<0,\q \forall\; t\in
\cI_{\delta}^{c}.
\end{array}
\end{equation}
By (\ref{bar y1 admissible eq1}) and (\ref{bar
y1 admissible eq2}),   $x^{\eps}(\cd)$ satisfies
the state constraint \eqref{constraints1} for
$\eps< \min \{\eps_{0},\eps_{1}\}$.

Since $\tilde x_{1}(T)\in \cE^{(1)}_T$, $
\mE\langle g^{j}_{x}(\bx(T)),\tilde
x_{1}(T)\rangle_H<0 \mbox{ for every } j\in
\cI(\bx). $ Similar to the proof of \eqref{bar
y1 admissible eq2},  for every sufficiently
small $\eps$, $x^{\eps}(\cd)$ satisfies the
final state constraint \eqref{constraints}, and
$(x^{\eps}(\cd), u^{\eps}(\cd))\in \cP_{ad}$.
Following (\ref{contodicition phiy1T}), there
exists $\rho_{2}>0$ such that for all
sufficiently small $\eps$,\vspace{-0.2cm}
$$
\begin{array}{ll}
\mE h(x^{\eps}(T))\3n&=\ds\mE h(\bx(T))+\eps\mE\inner{h_{x}(\bx(T))}{\tilde x_{1}(T)}_H+o(\eps)\\
\ns&<\ds\mE h(\bx(T))-\eps\rho_{2}+o(\eps)< \mE
h(\bx(T)),
\end{array}\vspace{-0.2cm}
$$
contradicting  the optimality of
$(\bx(\cd),\bu(\cd))$. This completes the proof
of (\ref{contodicition phiy133T}).

\vspace{+0.25cm}

To finish the proof, we consider three different
cases.

\vspace{+0.25cm}

{\bf Case 1}: $\cQ^{(1)}\cap
\cG^{(1)}=\emptyset$.

Noting that $\cQ^{(1)}$ is nonempty, open and
convex, and $\cG^{(1)}$ is nonempty and convex,
by the Hahn-Banach separation theorem and Lemma
\ref{lm12}, there exists a nonzero $\psi(\cd)\in
L^{2}_{\dbF}(\Omega;$ $BV_0([0,T]; H))$ such
that
$$
\sup_{z\in
\cQ^{(1)}}\mE\int_{0}^{T}\inner{z(t)}{d
\psi(t)}_H \le \inf_{\tilde z\in \cG^{(1)}}
\mE\int_{0}^{T}\inner{\tilde z(t)}{d \psi(t)}_H.
$$
Since $\cQ^{(1)}$ and $\cG^{(1)}$ are
cones,\vspace{-0.2cm}
$$
0=\sup_{z\in
\cQ^{(1)}}\mE\int_{0}^{T}\inner{z(t)}{d
\psi(t)}_H = \inf_{\tilde z\in \cG^{(1)}}
\mE\int_{0}^{T}\inner{\tilde z(t)}{d \psi(t)}_H.
$$
Therefore, $ \psi\in \big(\cQ^{(1)}\big)^-$ and
$-\psi\in \big(\cG^{(1)}\big)^-$. Consequently,
for all $z_1(\cd)\in\cG^{(1)}$,\vspace{-0.2cm}
\begin{equation}\label{integral by part psi y1 case1}
\mE\int_{0}^{T}\inner{z_{1}(t)}{d\psi(t)}_H \geq
0.\vspace{-0.2cm}
\end{equation}
Furthermore, it follows from the definition of
the transposition solution to \eqref{first
ajoint equ} that for every $x_1$ solving
\eqref{first vari equ} with $u_1\in
\cT_\Phi(\bar u)$ and $\nu_1\in
T^b_{\cK_{a}}(\bar x_{0})$,
\begin{eqnarray}\label{duality y1 p1}
&&\3n\3n\!\!\mE\inner{y(T)}{x_{1}(T)}_H-
\inner{y(0)}{\nu_1}_H \nonumber\\
&&\3n\3n\!\!=
\mE\int_{0}^{T}\!\!\Big(\inner{y(t)}{a_{1}[t]x_{1}(t)}_H
+\inner{y(t)}{a_{2}[t]u_1(t)}_H-
\inner{a_{1}[t]^* y(t) }{x_{1}(t)}_H -\!
\inner{b_{1}[t]^*Y(t)}{x_{1}\!(t)}_H\!\nonumber\\
&&\3n\3n\!\!\qq\qq\!\!\!  +\!
\inner{Y(t)}{b_{1}[t]x_{1}(t)}_{\cL_2} \!+\!
\inner{Y(t)}{b_{2}[t]u_1\!(t)}_{\cL_2}\!\Big)dt\!
+\!\mE\!\int_{0}^{T}\!\!
\inner{x_1(t)}{d\psi(t)}_H
 \\
&&\!\!\3n\3n= \mE\int_{0}^{T}\Big(
\inner{y(t)}{a_{2}[t]u_1(t)}_H +
\inner{Y(t)}{b_{2}[t]u_1(t)}_{\cL_2} \Big)dt
+\mE\int_{0}^{T}
\inner{x_1(t)}{d\psi(t)}_H.\nonumber
\end{eqnarray}
Set $\lambda_{0}=0$, $\lambda_{j}=0$, $j\in
\cI(\bx)$ and $y(T)=0$. Then, \eqref{6.23-eq2}
holds and (\ref{first order integraltype
condition}) follows from (\ref{integral by part
psi y1 case1}) and (\ref{duality y1 p1}).

\vspace{+0.2cm}

{\bf Case 2}: $\cQ^{(1)}\cap
\cG^{(1)}\neq\emptyset$ and $\cQ^{(1)}\cap
\cG^{(1)}\cap\cE^{(1)}=\emptyset$.

\vspace{+0.2cm}

If $\cE^{(1)}=\emptyset$, we claim that for each
$j\in \cI(\bar x)$, there exists $\lambda_{j}\ge
0$ such that\vspace{-0.12cm}
\begin{equation}\label{11.20-eq3}
\sum\limits_{j\in \cI(\bar x)} \lambda_{j}>0,\qq
\sum\limits_{j\in \cI(\bar x)}
\lambda_{j}g^{j}_{x}(\bx(T))=0.\vspace{-0.2cm}
\end{equation}
Indeed, if there is a $j_0\in \cI(\bar x)$ such
that $g^{j_0}_{x}(\bx(T))=0$, then we can take
$\l_{j_0}=1$ and $\l_j=0$ for all $j\in \cI(\bar
x)\setminus \{j_0\}$. In this context,
\eqref{11.20-eq3} hold.

If $g^{j}_{x}(\bx(T))\neq0$ for all $j\in
\cI(\bar x)$,  then $\cE^{(1,j)} \neq\emptyset$
for all $j\in \cI(\bar x)$ since $\Gamma$ is
surjective (recall \eqref{Ed1.1} for the
definition of $\cE^{(1,j)}$). From
\eqref{11.20-eq4}, we find that $\cE^{(1,j)}_T
\neq\emptyset$ for all $j\in \cI(\bar x)$. On
the other hand, since $\cE^{(1)}=\bigcap_{j\in
\cI(\bar x)}\cE^{(1,j)}=\emptyset$, by
\eqref{11.20-eq4}, we get that
$\cE^{(1)}_T=\bigcap_{j\in \cI(\bar
x)}\cE^{(1,j)}_T=\emptyset$. Then one can find a
$j_0\in \cI(\bar x)$ and a subset
$\cI_{j_0}\subset \cI(\bar x)\setminus\{j_0\}$
such that $\bigcap_{j\in \cI_{j_0}}
\cE^{(1,j)}_T\neq\emptyset$ and\vspace{-0.12cm}
$$
\cE^{(1,j_0)}_T\bigcap\(\bigcap_{j\in \cI_{j_0}}
\cE^{(1,j)}_T\)=\emptyset.\vspace{-0.252cm}
$$
By the Hahn-Banach separation theorem, there
exists a nonzero $\xi\in L^2_{\cF_T}(\Omega;H)$
such that\vspace{-0.12cm}
$$
\sup_{\eta\in  \cE^{(1,j_0)}_T}\mE\langle\xi,
\eta \rangle_H\leq \inf_{\eta\in \bigcap_{j\in
\cI_{j_0}} \cE^{(1,j)}_T}\mE\langle\xi, \eta
\rangle_H.\vspace{-0.12cm}
$$
Noting that $\cE^{(1,j)}_T$ ($j\in \cI(\bar x)$)
is a cone,  $\xi\in \big(\cE^{(1,j_0)}_T\big)^-$
and $-\xi\in \big(\bigcap_{j\in \cI_{j_0}}
\cE^{(1,j)}_T\big)^-$. By Lemma \ref{lm6}, $\xi
= \l_{j_0}g_x^{j_0} (\bar x(T))$ for some
$\l_{j_0}
> 0$. Further, for every $j \in \cI_{j_0}$, there
exists   $\l_j \geq 0$ such that
$\ds-\xi=\sum_{j\in \cI_{j_0}}\l_jg_x^j (\bar
x(T))$. Let $\l_j= 0$ for $j \in \cI(\bar
x)\setminus(\cI_{j_0}\cup\{j_0\})$, we get
\eqref{11.20-eq3}.

By taking $\lambda_{0}=0$, $\psi= 0$ and
$y(T)=0$, we have \eqref{6.23-eq2} and the
condition (\ref{first order integraltype
condition}) holds trivially with $(y,Y)\equiv
0$.

If $\cE^{(1)}\neq\emptyset$, then
$\Gamma\big(\cQ^{(1)}\cap \cG^{(1)}\big)\cap
\cE^{(1)}_{T}=\emptyset$. By the Hahn-Banach
theorem, there exists a nonzero $\xi\!\in\!
L_{\cF_{T}}^{2}(\Omega;H)$ such
that\vspace{-0.2cm}
$$\sup_{\alpha\in \Gamma(\cQ^{(1)}\cap \cG^{(1)})}\mE\inner{\xi}{\alpha}_H \leq
\inf_{\beta\in \cE^{(1)}_{T}}
\mE\inner{\xi}{\beta}_H.$$
Since both $\Gamma(\cQ^{(1)}\cap \cG^{(1)})$ and
$\cE^{(1)}_{T}$ are cones, $$ 0=\sup_{\alpha\in
\Gamma(\cQ^{(1)}\cap \cG^{(1)})}
\mE\inner{\xi}{\alpha}_H = \inf_{\beta\in
\cE^{(1)}_{T}} \mE\inner{\xi}{\beta}_H. $$
Therefore, $ \xi \in \big(\Gamma(\cQ^{(1)}\cap
\cG^{(1)})\big)^-$ and $-\xi \in
\big(\cE^{(1)}_{T}\big)^-. $

By Lemma \ref{lm6}, for each $j\in \cI(\bar x)$,
there exists $\lambda_j\geq 0$ such
that\vspace{-0.12cm}
$$
\sum_{j\in \cI(\bar x)}\lambda_j>0, \qq -\xi=
\sum_{j\in \cI(\bar x)}\lambda_{j}
g_{x}^{j}(\bx(T)).\vspace{-0.2cm}
$$
Since $ 0\ge\mE\inner{\xi}{\Gamma(z)}_H$  for
all $z\in \cQ^{(1)}\cap \cG^{(1)}$, we have that
$\Gamma^{*}(\xi)\in\big(\cQ^{(1)}\cap\cG^{(1)}\big)^-$.
By Lemma \ref{lm6}, there exists
$\psi\in\big(\cQ^{(1)}\big)^-$ with $\psi(0)=0$
such that $\Gamma^{*}(\xi)-\psi \in
\big(\cG^{(1)}\big)^-$.  Thus,  for all
$z(\cd)\in \cG^{(1)}$,\vspace{-0.2cm}
\begin{equation}\label{integral by part psi y1 caseb}
0 \ge \mE \inner{\xi}{z(T)}_H
-\mE\int_{0}^{T}\inner{z(t)}{d
\psi(t)}_H.\vspace{-0.2cm}
\end{equation}
Let $\lambda_{0}=0$. Since $\xi\neq 0$,
\eqref{6.23-eq2} holds. Set $\ds
y(T)=-\sum_{j\in \cI(\bar x)}\lambda_{j}
g_{x}^{j}(\bx(T))$.
By (\ref{duality y1 p1}) and (\ref{integral by
part psi y1 caseb}), we obtain (\ref{first order
integraltype condition}).

\vspace{+0.6em}

{\bf Case 3}: $\cQ^{(1)}\cap
\cG^{(1)}\cap\cE^{(1)}\neq\emptyset$.

In this case, it holds that
$\Gamma\big(\cQ^{(1)}\cap \cG^{(1)}\big)\cap
\cE^{(1)}_{T}\neq\emptyset.
$
By (\ref{contodicition
phiy133T}),\vspace{-0.1cm}
$$
\mE\inner{h_{x}(\bx(T))}{z(T)}_H\geq 0 , \q
\forall z(\cd)\in \cQ^{(1)}\cap \cG^{(1)}\cap
\cE^{(1)}.\vspace{-0.1cm}
$$
This yields that\vspace{-0.2cm}
\begin{equation*}
\mE\inner{h_{x}(\bx(T))}{\zeta}_H \geq 0,\; \;\;
\forall\, \zeta\in \Gamma\big(\cQ^{(1)}\cap
\cG^{(1)}\big)\cap \cE^{(1)}_{T}.
\end{equation*}
Consequently,\vspace{-0.2cm}
$$
-h_{x}(\bx(T))\in \big[\Gamma\big(\cQ^{(1)}\cap
\cG^{(1)}\big)\cap \cE^{(1)}_{T}\big]^{-}.
$$
By Lemma \ref{lm6},\vspace{-0.2cm}
$$
\big[\Gamma\big(\cQ^{(1)}\cap \cG^{(1)}\big)\cap
\cE^{(1)}_{T}\big]^{-}
=\big[\Gamma\big(\cQ^{(1)}\cap
\cG^{(1)}\big)\big]^{-}+\big(\cE^{(1)}_{T}\big)^{-}.
$$
Then, for each $j\in \cI(\bar x)$, there exists
$\lambda_j\ge 0$ such that\vspace{-0.2cm}
$$
\xi\=\sum\limits_{j\in \cI(\bar
x)}\lambda_{j}g^{j}_{x}(\bx(T)) \in
\big(\cE^{(1)}_{T}\big)^{-}\vspace{-0.2cm}
$$
and that\vspace{-0.2cm}
$$
-h_{x}(\bx(T))-\sum\limits_{j\in \cI(\bar
x)}\lambda_{j}g^{j}_{x}(\bx(T))\in
\big[\Gamma\big(\cQ^{(1)}\cap
\cG^{(1)}\big)\big]^{-}.\vspace{-0.2cm}
$$
Therefore,\vspace{-0.32cm}
$$
\Gamma^{*}\Big(-h_{x}(\bx(T))-\sum\limits_{j\in
\cI(\bar x)}\lambda_{j}g^{j}_{x}(\bx(T))\Big)\in
\big(\cQ^{(1)}\cap \cG^{(1)}\big)^{-}=
\big(\cQ^{(1)}\big)^{-}+
\big(\cG^{(1)}\big)^{-}.
$$

Let  $\psi\in\big(\cQ^{(1)}\big)^-$ with
$\psi(0)=0$ be such that\vspace{-0.2cm}
$$
\Gamma^{*}\Big(-h_{x}(\bx(T))-\sum\limits_{j\in
\cI(\bar
x)}\lambda_{j}g^{j}_{x}(\bx(T))\Big)-\psi\in
\big(\cG^{(1)}\big)^{-}.\vspace{-0.1cm}
$$
Set $\lambda_{0}=1$ and $
y(T)=-h_{x}(\bx(T))-\sum\limits_{j\in \cI(\bar
x)}\lambda_{j}g^{j}_{x}(\bx(T)).
$
Then, \eqref{6.23-eq2} holds and for all $z\in
\cG^{(1)}$,\vspace{-0.2cm}
\begin{equation}\label{integral by part psi y1 case2}
0 \geq -\mE
\inner{h_{x}(\bx(T)))}{z(T)}_H-\sum_{j\in
\cI(\bar x)}\lambda_{j}\mE
\inner{g^{j}_{x}(\bx(T))}{z(T)}_H -
\mE\int_{0}^{T}\inner{z(t)}{d\psi(t)}_H.\vspace{-0.2cm}
\end{equation}
Combining (\ref{integral by part psi y1 case2})
with (\ref{duality y1 p1}),  we obtain
(\ref{first order integraltype condition}). This
completes the proof of Theorem \ref{th1}.
\end{proof}

Let $\Phi(t,\omega)=\cC_{U}(\bu(t,\omega))$, for
a.e. $ (t,\omega)\in[0,T]\times\Omega$ and
$\cT_{\cK_a}(\bar \nu_0)=\cC_{\cK_a}(\bar
\nu_0)$. From Theorem \ref{th1} and Lemma
\ref{lm8}, it is easy to obtain the following
pointwise first order necessary condition.

\vspace{0.1cm}
\begin{theorem}\label{th2}
Let {\bf(AS1)}--{\bf(AS4)} hold and $(\bar
x(\cd),\bar u(\cd),\bar \nu_{0})$ be an optimal
triple for {\bf Problem (OP)}  such that
$\mE|g^0_{x}(\bx(t))|_H \neq 0$ for any $t\in
\cI^0(\bar x)$. Then  for $(y,Y)$ as in Theorem
\ref{th1},\vspace{-0.2cm}
\begin{equation}\label{th2-eq1}
y(0)\in  \cN^C_{\cK_a}(\bar \nu_0),\q
\dbH_{u}[t] \in \cN^C_{U}(\bar u(t)),\ a.e.\
t\in [0,T],\;\dbP\mbox{-a.s.}
\end{equation}
\end{theorem}
\begin{remark}
If both the control set $U$ and the initial
state constraint set $\cK_a$ are convex, then
$\cN_{U}^C(\bu)$ and $\cN_{\cK_a}^C(\bx_{0})$
are simply the normal cones of convex analysis.
\end{remark}
\begin{remark}
Let\vspace{-0.2cm}
$$
\begin{array}{ll}\ds
{\cal H}(t,x,u,\omega) \3n&\ds= \dbH(t,x,u,
y(t),Y(t),\omega)
 - \frac{1}{2}\inner{ P(t)b(t,\bx(t),\bu(t),\omega)}{b(t,\bx(t),\bu(t),\omega)}_{\cL_2}\\
\ns &\ds\q +  \frac{1}{2}\big\langle
P(t)\big(b(t,x,u,\omega)
 - b(t,\bx(t),\bu(t),\omega)\big),
b(t,x,u,\omega)-b(t,\bx(t),\bu(t),\omega)\big\rangle_{\cL_2},
\end{array}\vspace{-0.2cm}
$$
where $P(\cd)$ is the first element of the
solution of the second order adjoint process
with respect to $(\bx(\cd),\bu(\cd),\bar\nu_0)$
(defined by (\ref{op-bsystem3}) in Section 4).
If there is no state constraint, the stochastic
maximum principle (e.g. \cite{LZ1,LZ}) says
that, if $(\bx(\cd),\bu(\cd),\bar\nu_0)$ is an
optimal triple, then\vspace{-0.12cm}
\begin{equation}\label{maximum principle}
{\cal H}(t,\bx(t),\bu(t)) =\max_{v\in U} {\cal
H}(t,\bx(t),v),\quad \ a.e.\; t \in [0,T],
\;\;\dbP\mbox{-a.s.}\vspace{-0.2cm}
\end{equation}
This implies that\vspace{-0.32cm}
$$
\inner{\dbH_{u}(t,\omega)}{v}_{H_1}\le 0,\quad
\forall\, v\in \cC_{U}(\bu(t,\omega)),\ a.e.\
(t,\omega)\in [0,T]\times\Omega,\vspace{-0.2cm}
$$
i.e., the second condition in \eqref{th2-eq1}
holds. However, to derive \eqref{maximum
principle}, one has to assume that $a$, $b$  and
$h$ are $C^2$ with respect to the variable $x$.
Therefore, in practice, under some usual
structural assumptions on $U$, it is more
convenient to use the condition \eqref{th2-eq1}
directly.
\end{remark}

As for the deterministic optimal control
problems with state constraints, we call the
first order necessary condition \eqref{first
order integraltype condition} {\em normal} if
the Lagrange multiplier $\lambda_{0}\neq 0$. By
Theorem \ref{th1}, this is the case when
$\cG^{(1)}\cap\cQ^{(1)} \cap\cE^{(1)}
\neq\emptyset$. Let us give some  conditions to
guarantee it. To this end, we first introduce
the following equation:\vspace{-0.1cm}
\begin{equation}\label{first ajoint equ1}
\left\{
\begin{array}{ll}
d\tilde y(t)=- \big(A^*\tilde
y(t)+a_{1}[t]^{*}\tilde y(t) +b_{1}[t]^{*}\wt
Y(t) + \a(t)
\big)ds +\wt Y(t)dW(t) &\mbox{ in }[0,T), \\
\tilde y(T)=0,
\end{array}
\right.\vspace{-0.1cm}
\end{equation}
where  $\a(\cd)\in L^2_{\dbF}(0,T;H)$. The
equation \eqref{first ajoint equ1} is a special
case  of \eqref{first ajoint equ}, where
$d\psi(\cd)=\a(\cd)$.

Let us make the following assumptions:

\vspace{0.2cm}

{\bf (AAS1)}  $\a(\cd)=0$ whenever
$a_{1}(\cd)^{*}\tilde y[\cd]+b_{1}[\cd]^{*}\wt
Y(\cd) =0$.

\vspace{0.2cm}

{\bf (AAS2)} $\cC_{U}(\bu(t,\omega))=H_1$, for
a.e. $ (t,\omega)\in[0,T]\times\Omega$.

\vspace{0.2cm}

{\bf (AAS3)} There is a $\b(\cd)\in
C_\dbF([0,T];L^2(\Om;H))$ such
that\vspace{-0.1cm}
$$
\begin{cases}
\mE\inner{g^{0}_{x}(\bar x(t))}{\b(t)}_{H}<0,\q
\forall\, t\in \cI^0(\bar x),\\
\ns\ds\mE\inner{g^{j}_{x}(\bar
x(T))}{\b(T)}_{H}<0,\; \forall\, j\in \cI(\bar
x).
\end{cases}\vspace{-0.1cm}
$$
\begin{remark}
{\bf (AAS1)} is a condition about the unique
continuation for the solution of \eqref{first
ajoint equ1}. It means that if
$a_{1}[\cd]^{*}\tilde y(\cd)+b_{1}[\cd]^{*}\wt
Y(\cd) =0$, then the nonhomogeneous term
$\a(\cd)$ must be zero. A sufficient condition
for {\bf (AAS1)} is that $a_{1}[\cd]^{*}$ is
injective and $b_{1}[\cd]^{*}=0$.
\end{remark}
\begin{remark}
{\bf (AAS2)} means that $\cT_\Phi(\bar
u)=L^2_\dbF(0,T;H_1)$. This, together with {\bf
(AAS1)}, guarantees that the solution set of
\eqref{first vari equ} is rich enough for us to
choose one belonging to $\cQ^{(1)}
\cap\cE^{(1)}$. {\bf (AAS2)} holds for some
trivial cases. For example, $U=H_1$ or $\bar
u(t,\omega)\in{\rm int}U$, $\dbP$-a.s. for a.e.
$t\in [0,T]$. Note that we put state constraints
\eqref{constraints1} and \eqref{constraints} in
the control problem. Hence, even for $U=H_1$,
the optimal control problem is not trivial. We
believe that for some concrete control problem,
both {\bf (AAS1)} and {\bf (AAS2)} can be
dropped.   A possible way to do it is to follow
the idea in the proof of Proposition 3.3 in
\cite{Frankowska2018a}. The detailed analysis is
beyond the scope of this paper and will be
investigated in future work.
\end{remark}
\begin{remark}
From the definition of $\cG^{(1)}$, $\cQ^{(1)}$
and $\cE^{(1)}$, it is clear that {\bf (AAS3)}
is necessary for $\cG^{(1)}\cap\cQ^{(1)}
\cap\cE^{(1)} \neq\emptyset$.
\end{remark}
\begin{proposition}\label{prop1}
Let {\bf(AS1)}--{\bf(AS4)} and
{\bf(AAS1)}--{\bf(AAS3)} hold. Then
$\cG^{(1)}\cap\cQ^{(1)} \cap\cE^{(1)}
\neq\emptyset$.
\end{proposition}
\begin{proof}
We divide the proof into two
steps.

{\bf Step 1}. It follows from {\bf(AAS2)} that
$\cT_\Phi(\bar u)=L^2_\dbF(0,T;H_1)$. Define a
map $\Pi:\cT_\Phi(\bar u)\to L^2_\dbF(0,T;H)$ in
the following way:\vspace{-0.1cm}
$$
\Pi(u_1)(\cd)=x_1(\cd),\vspace{-0.1cm}
$$
where $x_1(\cd)$ is the solution of \eqref{first
vari equ} for some $u_1(\cd)\in \cT_\Phi(\bar
u)$.

We claim that\vspace{-0.2cm}
\begin{equation}\label{8.3-eq5}
\Pi(\cT_\Phi(\bar u)) \mbox{ is dense in }
L^2_\dbF(0,T;H).
\end{equation}
Let us prove \eqref{8.3-eq5} by a contradiction
argument. Without loss of generality, we assume
that $\nu_1=0$. If \eqref{8.3-eq5} was false,
then there would exist a nonzero $\b_0(\cd)\in
L^2_\dbF(0,T;H)$ such that for any $u_1(\cd)\in
\cT_\Phi(\bar u)$,\vspace{-0.1cm}
\begin{equation}\label{8.3-eq4}
 \mE\int_0^T \lan x_1(t),\b_0(t) \ran_{H} dt=0.\vspace{-0.1cm}
\end{equation}
Let  $\a=\b_0$. By the definition of the
transposition solution of \eqref{first ajoint
equ1}, we have that for any $u_1(\cd)\in
\cT_\Phi(\bar u)$,\vspace{-0.285cm}
\begin{eqnarray}\label{8.3-eq3}
0 \!=\! \mE\!\int_0^T\!\! \lan x_1(t),\b_0(t)
\ran_{H} dt \!=\!\mE\!\int_0^T\!\! \lan
u_1(t),a_2(t)^*\tilde y(t)\ran_{H_1}
ds+\mE\!\int_0^T\!\! \lan u_1(t),b_2(t)^*\wt
Y(t)\ran_{H_1} dt.\vspace{-0.321cm}
\end{eqnarray}
This, together with the choice of $u_1(\cd)$,
implies that $a_{1}[\cd]^{*}\tilde
y(\cd)+b_{1}[\cd]^{*}\wt Y(\cd) =0$ for a.e.
$t\in [0,T]$. By {\bf (AAS1)}, we see  $\a=0$ in
$L^2_\dbF(0,T;H)$, a contradiction.
Consequently, \eqref{8.3-eq5} holds.

\vspace{0.1cm}

{\bf Step 2}. Since $\cI^0(\bar x)$ is compact,
by {\bf (AAS3)}, one can find a $\b(\cd)\in
C_\dbF([0,T];L^2(\Om;H))$ such that, there are
$\e_0>0$ and $M_0>0$ so that\vspace{-0.1cm}
\begin{equation}\label{11.20-eq10}
\begin{cases}
\mE\inner{g^{0}_{x}(\bar
x(t))}{\b(t)}_{H}<-\e_0,\q |g^{0}_{x}(\bar
x(t))|_{L^2_{\cF_t}(\Om;H)}\leq M_0, \q
\forall\, t\in \cI^0(\bar x),\\
\ns\ds\mE\inner{g^{j}_{x}(\bar
x(T))}{\b(T)}_{H}<-\e_0,\q |g^{j}_{x}(\bar
x(T))|_{L^2_{\cF_t}(\Om;H)}\leq M_0,\q \forall\,
j\in \cI(\bar x).
\end{cases}\vspace{-0.1cm}
\end{equation}

It follows from \eqref{8.3-eq5} that for every
$k\in\dbN$, there is $u_{1,k}\in \cT_\Phi(\bar
u)$ such that the corresponding solution
$x_{1,k}=\Pi(u_{1,k})$ satisfies
that\vspace{-0.4332cm}
\begin{equation*}\label{8.3-eq1}
|x_{1,k}-
\b|_{L^2_\dbF(0,T;H)}<\frac{1}{k}.\vspace{-0.1cm}
\end{equation*}
Consequently, there is a subsequence
$\{u_{1,k_j}\}_{j=1}^\infty$ of
$\{u_{1,k}\}_{k=1}^\infty$ such
that\vspace{-0.1cm}
\begin{equation}\label{8.3-eq6}
\lim_{j\to\infty}x_{1,k_j}(t)=\b(t) \mbox{ in
}L^2_{\cF_T}(\Om;H),\q \mbox{ for a.e. }t\in
[0,T].\vspace{-0.2cm}
\end{equation}
Since both $x_{1,k_j}(\cd)$ and $\b(\cd)$ belong
to $C_\dbF([0,T];L^2(\Om;H))$, we get from
\eqref{8.3-eq6} that\vspace{-0.2cm}
\begin{equation*}\label{8.3-eq7*}
\lim_{j\to\infty}x_{1,k_j}(\cd)=\b(\cd)\; \mbox{
in }\;C_\dbF([0,T];L^2(\Om;H)).\vspace{-0.2cm}
\end{equation*}
Hence, there exists $N\in\dbN$ such
that\vspace{-0.2cm}
\begin{equation*}\label{8.3-eq8*}
|x_{1,N}(t)-\b(t)|_{L^2_{\cF_T}(\Om;H)}<
\frac{\e_0}{2M_0}\; \mbox{ for all }t\in
[0,T].\vspace{-0.1cm}
\end{equation*}
This, together with {\bf (AAS3)} and
\eqref{11.20-eq10}, implies that\vspace{-0.1cm}
$$
\begin{array}{ll}\ds
\mE\inner{g^{0}_{x}(\bar
x(t))}{x_1(t)}_{H}\3n&\ds
=\mE\inner{g^{0}_{x}(\bar
x(t))}{x_1(t)-\b(t)}_{H}+
\mE\inner{g^{0}_{x}(\bar x(t))}{\b(t)}_{H}\\
\ns&\ds \leq M_0 \times \frac{\e_0}{2M_0} - \e_0
<0,\qq\qq \forall\, t\in \cI^0(\bar x)
\end{array}\vspace{-0.21cm}
$$
and
$$
\begin{array}{ll}\ds
\mE\inner{g^{j}_{x}(\bar
x(T))}{x_1(T)}_{H}\3n&\ds=\mE\inner{g^{j}_{x}(\bar
x(T))}{x_1(T)-\b(T)}_{H}+
\mE\inner{g^{j}_{x}(\bar
x(T))}{\b(T)}_{H}\\
\ns&\ds \leq  M_0 \times \frac{\e_0}{2M_0} -
\e_0 <0,\qq\qq \forall\, j\in \cI(\bar x).
\end{array}\vspace{-0.1cm}
$$
This completes the proof.
\end{proof}


\section{Second order necessary conditions}\label{second}


In this section, we establish second order
necessary conditions for the optimal triple of
{\bf Problem (OP)}. In addition to {\bf
(AS1)}--{\bf (AS4)}, we impose the following:

\vspace{0.2cm}

{\bf (AS5)}  {\em For   a.e. $(t,\omega)\in [0,
T]\times\Omega$, the operators $a(t, \cdot,
\cdot,\omega):\ H\times H_1\to H$ and $b(t,
\cdot, \cdot,\omega):\ H\times H_1\to \cL_2$ are
$C^2$, and $a_{xu}(t,x,u,\omega)$ and
$b_{xu}(t,x,u,\omega)$ are uniformly continuous
with respect to $x\in H$ and $u\in H_1$, and
}\vspace{-0.2cm}
$$
|a_{xu}(t,x,u,\omega)|_{\cL(H\times H_1;H)} +
|b_{xu}(t,x,u,\omega)|_{\cL(H\times H_1;\cL_2)}
\leq C,\q \forall\; (x,u)\in H\times
H_1.\vspace{-0.1cm}
$$

\vspace{0.1cm}

{\bf (AS6)}  {\em The functional
$h(\cdot,\omega):\ H \to \dbR$ is $C^2$,
$\dbP$-a.s., and for any $x,\ \tilde{x}\in
H$,}\vspace{-0.2cm}
$$
|h_{xx}(x,\omega)|_{\cL(H\times H;\dbR)} \le C,
\q
|h_{xx}(x,\omega)-h_{xx}(\tilde{x},\omega)|_{\cL(H\times
H;\dbR)}\le C|x-\tilde{x}|_H.\vspace{-0.2cm}
$$

\vspace{0.2cm}

{\bf (AS7)}  {\em For $j=0,1,\cds,n$, the
functional $g^j(\cdot):\ H \to \dbR$ is $C^2$,
and for any $x,\ \tilde{x}\in
H$,}\vspace{-0.2cm}
$$
|g^j_{xx}(x,\omega)|_{\cL(H\times H;\dbR)} \leq
C, \q
|g^j_{xx}(x,\omega)-g^j_{xx}(\tilde{x},\omega)|_{\cL(H\times
H;\dbR)}\leq C|x-\tilde{x}|_H.\vspace{-0.2cm}
$$

\vspace{0.1cm}

{\bf (AS8)}  {\em The optimal control }$\bar
u\in\cV\= \cU\cap L^4_{\dbF}(0,T;H_1)$.

\vspace{0.2cm}

In what follows,  $\mathcal{V}$ is viewed as a
subset of $L^{4}_{\dbF}(0,T;H_1)$ in the
definitions of $T^b_{\mathcal{V}}(\bu)$ and
$T^{b(2)}_{\mathcal{V}}(\bu,v)$.

\vspace{0.2cm}

{\bf (AS9)}  {\em $(\Om,\cF_T,\dbP)$ is
separable.}

\begin{remark}
Recall that $(\Om,\cF_T,\dbP)$ is separable  if
there exists a countable family
$\cD\subset\cF_T$ such that, for any $\e>0$ and
$B\in \cF_T$ one can find $B_1\in\cD$ with
$\dbP\big((B\setminus B_1)\cup (B_1\setminus
B)\big)<\e$. Probability space enjoying such
kind of property is called a standard
probability space. Except some artificial
examples, almost all frequently used probability
spaces are standard ones(e.g. \cite{Rohlin}).
From \cite[Section 13.4]{BBT},  if {\bf (AS9)}
holds,  then $L^{p}_{\cF_T}(\Om)$ ($1\leq
p<\infty$) is separable.
\end{remark}

Consider the following $\cL(H)$-valued
BSEE\footnote{Throughout this paper, for any
operator-valued process (\resp random variable)
$R$, we denote by $R^*$ its pointwisely dual
operator-valued process (\resp random variable),
e.g., if $R\in L^{r_1}_\dbF(0,T; L^{r_2}(\Omega;
\cL(H)))$, then $R^*\in L^{r_1}_\dbF(0,T;
L^{r_2}(\Omega; \cL(H)))$, and
$\|R\|_{L^{r_1}_\dbF(0,T; L^{r_2}(\Omega;
\cL(H)))}=\|R^*\|_{L^{r_1}_\dbF(0,T;
L^{r_2}(\Omega; \cL(H)))}$.}:\vspace{-0.1cm}
\begin{eqnarray}\label{op-bsystem3}
\2n\left\{\2n
\begin{array}{ll} \ds dP \! =\!  -
(A^* \!+\! J^* )P dt \! - \! P(A \!+\! J )dt
\!-\!K^*PKdt \! - (K^* Q \!+\! Q K)dt
\!+\!   Fdt \! + \! Q dW(t) \;\mbox{ in } [0,T),\\
\ns\ds P(T) = P_T,
\end{array}
\right.\vspace{-0.1cm}
\end{eqnarray}
where $F\in L^1_\dbF(0,T;L^2(\Omega;\cL(H)))$,
$P_T\in L^2_{\cF_T}(\Omega;\cL(H))$, $J\in
L^4_\dbF(0,T; L^\infty(\Omega; \cL(H)))$  and
$K\in L^4_\dbF(0,T; L^\infty(\Omega;
\cL(H;\cL_2)))$. In \eqref{op-bsystem3}, the
unknown (or solution) is a pair $(P,Q)$.

Let us first recall the definition of the
relaxed transposition solution of
\eqref{op-bsystem3}. To this end, consider two
SEEs:\vspace{-0.2cm}
\begin{equation}\label{op-fsystem2}
\left\{
\begin{array}{ll}
\ds d\p_1(s) = \big[(A+J)\p_1(s) + \tilde f_1(s)\big]ds + \big(K\p_1(s) + \hat f_1(s)\big) dW(s) &\mbox{ in } (t,T],\\
\ns\ds \p_1(t)=\xi_1
\end{array}
\right.
\end{equation}
and\vspace{-0.2cm}
\begin{equation}\label{op-fsystem3}
\left\{
\begin{array}{ll}
\ds d\p_2(s) = \big[(A+J)\p_2(s) + \tilde
f_2(s)\big]ds
+ \big(K\p_2(s) + \hat f_2(s)\big) dW(s) &\mbox{ in } (t,T],\\
\ns\ds \p_2(t)=\xi_2.
\end{array}
\right.
\end{equation}
Here $t\in [0,T)$, $\xi_1,\xi_2 \in
L^4_{\cF_t}(\Omega;H)$, $\tilde f_1,\tilde f_2
\in L^2_\dbF(t,T;L^4(\Omega;H))$ and $\hat
f_1,\hat f_2 \in
L^2_\dbF(t,T;L^4(\Omega;\cL_2))$.

Write\vspace{-0.2cm}
\begin{equation*}\label{jshi1}
\!\!\!\!\!\!\begin{array}{ll}\ds
D_{\dbF,w}([0,T];L^{2}(\Omega;\cL(H)))\\
\ns\ds\= \Big\{P(\cd,\cd)\;\Big|\; P(\cd,\cd)\in
\cL\big(L^{2}_{\dbF}(0,T;L^{4}(\Omega;H)),
\;L^2_{\dbF}(0,T;L^{\frac{4}{3}}(\Omega;H))\big),\;
P(t,\o)\in\cL(H) \mbox{ for a.e. }
\\
\ns\ds\q\, (t,\o)\!\in\!
[0,T]\!\times\!\Om,\mbox{ and for every }
t\!\in\![0,T]\hb{ and }\xi\! \in\!
L^4_{\cF_t}(\Omega;H),\, P(\cd,\cd)\xi\!\in\!
D_{\dbF}([t,T];L^{\frac{4}{3}}(\Omega;H))\\
\ns\ds\q  \mbox{ and }
\|P(\cd,\cd)\xi\|_{D_{\dbF}([t,T];L^{\frac{4}{3}}(\Omega;H))}
\leq C\|\xi\|_{L^4_{\cF_t}(\Omega;H)} \Big\}
\end{array}\vspace{-0.12cm}
\end{equation*}
and\vspace{-0.32cm}
\begin{equation*}\label{jshi2}
\3n\begin{array}{ll}\ds
\dbQ[0,T]\!\=\!\Big\{\big(Q^{(\cd)},\widehat
Q^{(\cd)}\big)\;\Big|\;\mbox{For any } t\in
[0,T], \mbox{ both }Q^{(t)}\mbox{ and }\widehat
Q^{(t)}\mbox{ are bounded
linear operators}\\
\ns\ds\hspace{1.7cm}\mbox{ from
}L^4_{\cF_t}(\O;H)\times
L^2_\dbF(t,T;L^4(\Omega;H))\times
L^2_\dbF(t,T;L^4(\Omega;\cL_2)) \mbox{ to }
L^{2}_\dbF(t,T;L^{\frac{4}{3}}(\Omega;\cL_2))\\
\ns\ds \hspace{1.7cm} \mbox{ and
}Q^{(t)}(0,0,\cd)^*=\widehat
Q^{(t)}(0,0,\cd)\Big\}.
\end{array}\vspace{-0.1cm}
\end{equation*}

\begin{definition}\label{op-definition2x}
We call $\big(P(\cd),(Q^{(\cd)},\widehat
Q^{(\cd)})\big)\in D_{\dbF,w}([0,T];
L^{2}(\Omega;\cL(H)))\times \dbQ[0,T]$ a relaxed
transposition solution of \eqref{op-bsystem3} if
for every $t\in [0,T]$, $\xi_1,\xi_2\in
L^4_{\cF_t}(\Omega;H)$, $\tilde f_1(\cd), \tilde
f_2(\cd)\in L^2_{\dbF}(t,T;$ $L^4(\Omega;H))$
and $\hat f_1(\cd), \hat f_2(\cd)\in
L^2_{\dbF}(t,T; L^4(\Omega;\cL_2))$, the
following is satisfied\vspace{-0.2cm}
\begin{equation}\label{6.18eq1}
\begin{array}{ll}
\ds \q\mE\big\langle P_T \p_1(T),  \p_2(T)
\big\rangle_{H} - \mE \int_t^T \big\langle
F(s) \p_1(s), \p_2(s) \big\rangle_{H}ds\\
\ns\ds =\mE\big\langle P(t) \xi_1,\xi_2
\big\rangle_{H} + \mE \int_t^T \big\langle
P(s)\tilde f_1(s), \p_2(s)\big\rangle_{H}ds +
\mE \int_t^T \big\langle P(s)\p_1(s),
\tilde f_2(s)\big\rangle_{H}ds \\
\ns\ds \q  + \mE \int_t^T \big\langle P(s)K
(s)\p_1 (s), \hat f_2(s)\big\rangle_{\cL_2}ds +
\mE \int_t^T \big\langle  P(s)\hat
f_1(s), K (s)\p_2 (s)+ \hat f_2(s)\big\rangle_{\cL_2}ds\\
\ns\ds \q + \mE \int_t^T \big\langle \hat
f_1(s), \widehat Q^{(t)}(\xi_2,\tilde f_2 ,\hat
f_2)(s)\big\rangle_{\cL_2}ds+ \mE \int_t^T
\big\langle Q^{(t)}(\xi_1,\tilde f_1,\hat
f_1)(s), \hat f_2(s) \big\rangle_{\cL_2}ds.
\end{array}\vspace{-0.1cm}
\end{equation}
Here, $\p_1(\cd)$ and $\p_2(\cd)$ solve
\eqref{op-fsystem2} and \eqref{op-fsystem3},
respectively.
\end{definition}

\begin{lemma}\label{OP-th2}
Let {\bf (AS9)} hold. Then the equation
\eqref{op-bsystem3} admits a unique relaxed
transposition solution $\big(P(\cd),(Q^{(\cd)},$
$\widehat Q^{(\cd)})\big) \in D_{\dbF,w}([0,T];$
$ L^{2}(\Omega;\cL(H))) \times \dbQ[0,T]$.
Furthermore,\vspace{-0.21cm}
$$
\begin{array}{ll}\ds
|P|_{D_{\dbF,w}([0,T];L^{2}(\Om; \cL(H)))} +
|\big(Q^{(\cd)},\widehat Q^{(\cd)}\big)|_{\dbQ
[0,T]} \leq C\big(
|F|_{L^1_\dbF(0,T;\,L^{2}(\Om;\cL(H)))} +
|P_T|_{L^{2}_{\cF_T}(\Om;\,\cL(H))}\big).
\end{array}\vspace{-0.1cm}
$$
\end{lemma}
The proof is almost the same as the one of
\cite[Theorem 6.1]{LZ1}. The only difference is
that one should replace the inner product of $H$
by $\cL_2$ for terms involving $\hat f_1$ and
$\hat f_2$. Hence we omit it.

\vspace{0.2cm}

For $\varphi$ equal to $a$ or $b$,
let\vspace{-0.2cm}
$$
\varphi_{11}[t]=\varphi_{xx}(t,\bar{x}(t),\bar{u}(t)),\quad
\varphi_{12}[t]=\varphi_{xu}(t,\bar{x}(t),\bar{u}(t)),\quad
\varphi_{22}[t]=\varphi_{uu}(t,\bar{x}(t),\bar{u}(t)).\vspace{-0.1cm}
$$

For  $\nu_1\in T^b_{\cK_a}(\bx_{0})$, $u_1\in
T^b_{\mathcal{V}}(\bu)$,  $\nu_{2}\in
T^{b(2)}_{\cK_a}(\bx_{0},\nu_1)$ and $u_2\in
T^{b(2)}_{\mathcal{V}}(\bu,u_1)$, consider the
following second order variational
equation:\vspace{-0.2cm}
\begin{equation}\label{second order vari equ}
\left\{\!\!
\begin{array}{ll}\ds
dx_{2}(t)= \Big[Ax_{2}(t) +a_{1}[t]x_{2}(t) +
a_{2}[t]u_2(t) +\frac{1}{2}a_{11}[t]\big(x_{1}(t),x_{1}(t)\big)+ a_{12}[t]\big(x_{1}(t),u_1(t)\big)\\
\ns\ds\qquad\qquad\;+\frac{1}{2}a_{22}[t]\big(u_1(t),u_1(t)\big)\Big]dt
+\Big[b_{1}[t]x_{2}(t)+b_{2}[t]u_2(t) +\frac{1}{2}b_{11}[t]\big(x_{1}(t),x_{1}(t)\big)\\
\ns\ds\qquad\qquad\;+
b_{12}[t]\big(x_{1}(t),u_1(t)\big)
+\frac{1}{2}b_{22}[t]\big(u_1(t),u_1(t)\big)\Big]dW(t)  \qq\qq\qq\q\;\mbox{ in }(0,T],\\
x_{2}(0)=\nu_{2},
\end{array}
\right.
\end{equation}
where $x_{1}(\cd)$ is the solution of the first
order variational equation (\ref{first vari
equ}) (for $u_1(\cdot)$ and $\nu_1$ as above).

By the definition of the second order adjacent
tangent, for any $\e>0$, there exist
$\nu_{2}^\e\in H$ and $u_2^\e(\cd)\in
L^{4}_{\dbF}(0,T;H_1)$ such that\vspace{-0.2cm}
$$
\nu^{\e}_{0}\=\bar\nu_{0}+\e  \nu_1+\e^2
\nu_{2}^{\e}\in \cK_a, \qquad
u^{\e}(\cd)\=\bu(\cd)+\e u_1(\cd)+\e^2
u_2^{\e}(\cd)\in \cV\vspace{-0.2cm}
$$
and\vspace{-0.2cm}
$$
\lim_{\e\to 0^+}\nu_{2}^{\e}=\nu_{2} \mbox{ in
}H, \qq \lim_{\e\to 0^+}u_2^{\e}=u_2 \mbox{ in
}L^{4}_{\dbF}(0,T;H_1).
$$

Denote by $x^{\e}(\cd)$ the solution  of
(\ref{controlsys}) corresponding to the control
$u^{\e}(\cd)$ and the initial datum
$\nu^{\e}_{0}$. Put\vspace{-0.2cm}
$$
\delta x^{\e}(\cd)\=x^{\e}(\cd)-\bar{x}(\cd),
\qq r_2^{\e}(\cd)\=\frac{\delta x^{\e}(\cd)-\e
x_1(\cd)-\e^2 x_2(\cd)}{\e^2}.
$$

We have the following result.
\begin{lemma}\label{estimate two of varie qu}
Suppose that {\bf(AS1)}, {\bf(AS2)} and
{\bf(AS5)} hold. Then, for $\nu_1, \!\nu_{2},
\nu_{2}^{\e}\!\in\! H$ and $u_1(\cd), u_2(\cd)$,
$u_2^{\e}(\cd)\in L^{4}_{\dbF}(0,T;H_1)$ as
above, we have\vspace{-0.1cm}
$$
\|x_{2}\|_{L^\infty_\dbF(0,T;L^2(\Om;H))} \leq
C\big(|\nu_{2}|_H +|\nu_1|_H^2 +
|u_1|_{L^{4}_\dbF(0,T;H_1)}^2
+|u_2|_{L^{2}_\dbF(0,T;H_1)}
\big)\vspace{-0.1cm}
$$
and\vspace{-0.1cm}
\begin{equation}\label{r2 to 0}
\lim_{\e\to
0^+}|r_{2}^{\e}|_{L^\infty_\dbF(0,T;L^2(\Om;H))}=
0.
\end{equation}
\end{lemma}
Proof of Lemma \ref{estimate two of varie qu} is
provided in Appendix B.

Put\vspace{-0.1cm}
\begin{equation}\label{8.18-eq6}
\begin{array}{ll}\ds
\cY(\bar x,\bar u)\=\big\{ (x_1(\cd),u_1(\cd),
\nu_1)\in C_\dbF([0,T];L^4(\Om;H))\times
T_\cV^b(\bar u)\times T_{\cK_a}^b(\bar
\nu_0)\big|\ x_1(\cd)\mbox{
solves \eqref{first vari equ},}\\
\ns\ds\qq\qq\;\; x_{1}(\cd) \in {\rm cl}
\cQ^{(1)}\cap {\rm cl} \cE^{(1)}\mbox{ and }
\mE\langle g_x^j(\bar x(T)), x_1(T)\rangle_H\leq
0,\; \forall\ j\in \cI(\bar x)\big\}
\end{array}\vspace{-0.1cm}
\end{equation}
and define the critical cone\vspace{-0.1cm}
\begin{equation}\label{8.18-eq7}
\begin{array}{ll}\ds
\cZ(\bar x,\bar u)\=\Big\{(x_1(\cd),u_1(\cd),
\nu_1)\in \cY(\bar x,\bar u)\Big|\ \mE\langle
h_x (\bar x(T)), x_{1}(T)\rangle_H = 0 \Big\}.
\end{array}\vspace{-0.21cm}
\end{equation}
For a fixed $(x_1(\cd),u_1(\cd), \nu_1)\in
\cZ(\bar x,\bar u)$, let $\cW(\bar\nu_{0},
\nu_1)$ and $\cM(\bu,u_1)$ be convex subsets of
$T_{\cK_a}^{b(2)}(\bar\nu_{0}, \nu_1)$ and
$T_{\mathcal{V}}^{b(2)}(\bu,u_1)$, respectively.
Put\vspace{-0.21cm}
\begin{equation}\label{8.18-eq8}
\begin{array}{ll}\ds
\cG^{(2)}(x_1,u_1)\=\big\{x_2(\cd)\in
L^2_{\dbF}(\Om;C([0,T];H))\big|\,x_2(\cd) \mbox{
is the solution of \eqref{second order vari equ}
corresponding } \\
\ns\ds\qq\qq\qq \mbox{ to some }(\nu_2,u_2)\in
\cW(\bar\nu_0, \nu_1)\times \cM(\bar u,u_1)
\big\}.
\end{array}\vspace{-0.1cm}
\end{equation}

Let\vspace{-0.21cm}
$$
\dbI^0(\bar x,x_1)\=\big\{t\in \cI^0(\bar x)|\
\mE\langle g_x^0(\bar x(t)), x_1(t)
\rangle_H=0\big\},\vspace{-0.1cm}
$$
$$
\dbI(\bar x,x_1)\=\big\{j\in \cI(\bar x)|\
\mE\langle g_x^j(\bar x(T)), x_1(T)
\rangle_H=0\big\},
$$
\begin{equation}\label{index Tg bx y1}
\begin{array}{ll}\ds
\t^{g}(\bx)\=\big\{t\in [0,T]\; \big|\; \exists\; \{s_{k}\}_{k=1}^\infty\subset[0,T]\mbox{ such that }\lim_{k\to\infty}s_k= t,\,  \mE g^0(\bx(s_{k}))<0,\\
\ns\ds\qq\qq\qq\qq\q\mE\inner{g^0_{x}(\bx(s_{k}))}{x_{1}(s_{k})}_H>0,\;\forall\;
k=1,2,\cdots\big\},
\end{array}
\end{equation}
\begin{equation}\label{e(t)}
e(t)\=\left\{
\begin{array}{ll}
\limsup\limits_{\substack{s\to t\\ \mE~ g^0(\bx(s))<0\\
\mE\inner{g^0_{x}(\bx(s))}{x_{1}(s)}_H>0}}\dfrac{\big|\mE \inner{g^0_{x}(\bx(s))}{x_{1}(s)}_H\big|^2}{4\big|\mE g^0(\bx(s))\big|_H},  \q t\in \t^{g}(\bx),\\[+0.9em]
0,\qq\qq\qq\qq\qq\qq\qq\qq\ \mbox{otherwise},
\end{array}\right.
\end{equation}
\begin{equation}\label{Q2}
\begin{array}{ll}
\cQ^{(2)}(x_1)\=\Big\{z\in
L^2_{\dbF}(\Om;C([0,T];H))\Big|\ \mbox{for all
}t\in  \dbI^0(\bar x,x_1),\\
\ns\ds\hspace{2cm} \mE\langle g^0_x(\bar x(t)),
z(t) \rangle_H + \frac{1}{2} \mE\langle
g^0_{xx}(\bar x(t))x_1(t), x_1(t) \rangle_H+
e(t)<0\Big\},
\end{array}
\end{equation}
$$
\begin{array}{ll}\ds
\cE^{(2,j)}(x_1)\!\=\!\Big\{z\!\in\!
L^2_{\dbF}(\Om;C([0,T];H))\Big| \mE\langle
g^j_x(\bar x(T)), z(T) \rangle_H \!+
\!\frac{1}{2} \mE\langle g^j_{xx}(\bar
x(T))x_1(T), x_1(T) \rangle_H\!<\!0\Big\},
\end{array}\vspace{-0.1cm}
$$
\begin{equation}\label{E2}
\cE^{(2)}(x_1)\= \bigcap_{j\in \dbI(\bar
x,x_1)}\cE^{(2,j)}(x_1), \vspace{-0.21cm}
\end{equation}
and\vspace{-0.1cm}
\begin{eqnarray}\label{L2}
\!\!\!\cL^{(2)}\!(x_1\!)\!\=\!\Big\{\!z(\cdot)
\!\in\! L_{\dbF}^{2}(\Omega;
\!C([0,T];\!H))\Big|\, \mE \langle h_x(\bar
x(T)),z(T)\rangle_{H}\! +\! \frac{1}{2}\mE
\langle h_{xx}(\bar
x(T))x_1\!(T),x_1\!(T)\rangle_{H}\!<\!0\!\Big\}.
\end{eqnarray}
\begin{remark}
If $x_1\in \cQ^{(1)}$,  then $\dbI^0(\bar
x,x_1)=\emptyset$. Consequently,
$\cQ^{(2)}(x_1)= L_{\dbF}^{2}(\Omega; C([0,T];
H))$. In addition, if there exists  $\delta>0$
such that\vspace{-0.1cm}
$$\mE\inner{g_{x}(\bx(s))}{x_{1}(s)}_H\le 0,
\q \forall\,s\in (t-\delta,t+\delta)\cap[0,T],\;
t\in \cI^0(\bar x),
$$
then $e(t)= 0$ for any $t\in \dbI^0(\bar
x,x_1)$. In this case,\vspace{-0.1cm}
$$
\begin{array}{ll}\ds
\cQ^{(2)}(x_1)= \Big\{z(\cd)\in
L_{\dbF}^{2}(\Omega; C([0,T];
H))\,\Big|\;\mbox{For all } t\in
\dbI^0(\bar x,x_1), \\
\ns\ds\hspace{2.2cm}
\mE\inner{g^0_{x}(\bx(t))}{z(t)}_H+\frac{1}{2}\mE
\inner{g^0_{xx}(\bx(t))x_{1}(t)}{x_{1}(t)}_H<0\Big\}.
\end{array}\vspace{-0.1cm}
$$
\end{remark}
\begin{remark}\label{11.20-rmk1}
Let $z_1 \in  \mathcal{Q}^{(1)}$ and $z_2 \in
\mathcal{Q}^{(2)}(x_1)$. Then for every $t \in
\dbI^0(\bar x,x_1) \subset \cI^0(\bar x)$, we
have $ \mE\langle g^0_x(\bar x(t)), z_1(t)
\rangle_H<0 $ and $\ds \mE\langle g^0_x(\bar
x(t)), z_2(t) \rangle_H +\frac{1}{2} \mE\langle
g^0_{xx}(\bar x(t))x_1(t), x_1(t) \rangle_H+
e(t)<0$.
Therefore,\vspace{-0.21cm}
$$
\mE\langle g^0_x(\bar x(t)), z_1(t)+z_2(t)
\rangle_H + \frac{1}{2} \mE\langle g^0_{xx}(\bar
x(t))x_1(t), x_1(t) \rangle_H+
e(t)<0,\vspace{-0.1cm}
$$
which implies that $z_1+z_2\in
\mathcal{Q}^{(2)}(x_1)$. Consequently,
$\mathcal{Q}^{(1)} + \mathcal{Q}^{(2)}(x_1)
\subset \mathcal{Q}^{(2)}(x_1)$. Similarly, if
$\Phi(t,\omega)=\cC_{U}(\bar u(t,\omega))$, then
we can prove that $ \cG^{(1)} +
\cG^{(2)}(x_1,u_1) \subset \cG^{(2)}(x_1,u_1)$.
\end{remark}

Let $(y, Y)$, $\psi$ and $\lambda_{j}$, $j\in
\cI(\bar x)$ be defined as in the proof of
Theorem \ref{th1} in the case when
$\cG^{(1)}\cap\cQ^{(1)}\cap\cE^{(1)}\neq
\emptyset$ (See (\ref{1st variation set}),
(\ref{Q1}) and (\ref{Ed1}) for the definitions
of $\cG^{(1)}$, $\cQ^{(1)}$ and $\cE^{(1)}$,
respectively), where $\ds y(T)=-h_x(\bar x(T)) -
\sum_{j\in \cI(\bar x)}\l_jg_x^j(\bar x(T))$.

Let $(P(\cd),(Q^{(\cd)},\widehat Q^{(\cd)}))$ be
the relaxed transposition solution of the
equation \eqref{op-bsystem3} in which $P_T$,
$J(\cd)$, $K(\cd)$ and $F(\cd)$ are given
by\vspace{-0.1cm}
\begin{equation*}\label{zv2}
\begin{array}{ll}\ds
P_T = - h_{xx}\big(\bar x(T)\big),\q J(t) =
a_1[t],\q K(t) =b_1[t],\\
\ns\ds F(t)= -\dbH_{xx}[t]\=-\dbH_{xx}(t,\bar
x(t),\bar u(t), y(t),Y(t),\omega).
\end{array}\vspace{-0.1cm}
\end{equation*}

We have the following result.

\begin{theorem}\label{TH second order integral condition}
Suppose that {\bf(AS1)}--{\bf(AS9)} hold and
that $\cG^{(1)}\cap\cQ^{(1)}\cap\cE^{(1)}\neq
\emptyset$ for an optimal triple
$(\bx(\cd),\bu(\cd),\bar\nu_{0})$ of {\bf
Problem (OP)}. If $\cG^{(2)}(x_1,u_1)\cap
\cQ^{(2)}(x_1)\cap \cE^{(2)}(x_1)\neq
\emptyset$, then for any $x_2(\cd)\in
\cG^{(2)}(x_1,u_1)\cap {\rm cl} \cQ^{(2)}(x_1)
\cap {\rm cl} \cE^{(2)}(x_1)$ with the
corresponding  $\nu_2\in\cW(\bar \nu_0,\nu_1)$
and $u_2(\cd)\in\cM(\bar u,u_1)$, we
have\vspace{-0.2cm}
\begin{equation}\label{second order integral condition}
\begin{array}{ll}\ds
\langle y(0),\nu_2 \rangle_H+ \frac{1}{2}\langle
P(0)\nu_1,\nu_1 \rangle_H +\sum_{j\in \cI(\bar
x)}\mE
\inner{\lambda_{j}g^{j}_{x}(\bx(T))}{x_{2}(T)}_H
\\
\ns\ds  +\mE \int_0^T \(\big\langle \dbH_{u}[t],
u_2(t)\big\rangle_{H_{1}} +
\frac{1}{2}\big\langle \dbH_{uu}[t]u_1(t),
u_1(t)\big\rangle_{H_{1}}+
\frac{1}{2}\big\langle
b_{2}[t]^{*}P(t)b_{2}[t]u_1(t),
u_1(t)\big\rangle_{H_{1}}
\\
\ns\ds \qq\qq + \big\langle\big(\dbH_{xu}[t] +
a_{2}[t]^*
P(t) + b_{2}[t]^*P(t) b_{1}[t]\big)x_{1}(t), u_1(t)\big\rangle_{H_{1}} \\
\ns\ds + \frac{1}{2}\big\langle\big( \widehat
Q^{(0)} + Q^{(0)}\big)\big(0,a_{2}[t]u_1(t),
b_{2}[t]u_1(t)\big),b_{2}[t]u_1(t)
\big\rangle_{\cL_2}\)dt + \mE\int_0^T
\inner{x_2(t)}{d\psi(t)}_H \leq 0,
\end{array}
\end{equation}
where \vspace{-0.32cm}
\begin{equation*}\label{11.20-eq2}
\dbH_{uu}[t]\=\dbH_{uu}(t,\bar x(t),\bar u(t),
y(t),Y(t),\omega),\q
\dbH_{xu}[t]\=\dbH_{xu}(t,\bar x(t),\bar u(t),
y(t),Y(t),\omega).
\end{equation*}
\end{theorem}
\begin{remark}
In Theorem \ref{TH second order integral
condition}, we  take $\lambda_0=1$ and  $(y,
Y)$, $\psi$ and $\lambda_{j}$, $j\in \cI(\bx)$
as in Theorem \ref{th1}. Accordingly, the terms
$\ds\sum_{j\in \cI(\bar x)}\mE
\inner{\lambda_{j}g^{j}_{x}(\bx(T))}{x_{2}(T)}_H$
and $\ds\mE\int_0^T \inner{x_2(t)}{d\psi(t)}_H$
appear. By doing so,  our second order condition
is valid for any normal multiplier appearing in
the first order conditions.
\end{remark}
In Theorem \ref{TH second order integral
condition} we assumed that
$\cG^{(2)}(x_1,u_1)\cap \cQ^{(2)}(x_1)\cap
\cE^{(2)}(x_1)\neq \emptyset$. It seems that
this condition is not easy to verify. Let us
give a result concerning this below.
\begin{proposition}
Assume that there is $(x_1,u_1,\nu_1)\in
\cZ(\bx,\bu)$ such that the function $e(\cdot)$
defined by (\ref{e(t)}) is  bounded on
$\dbI^0(\bar x,x_1)$, and that
$T^{b(2)}_{\cK_a}(\bar\nu_0, \nu_1)$ and $
T^{b(2)}_{\cV}(\bu,u_1)$ are nonempty. If
$\cG^{(1)}\cap\cQ^{(1)}\cap\cE^{(1)}\neq
\emptyset$ (with $\cT_{\cK_{a}}(\bar\nu_0)$ and
$\cT_{\Phi}(\bu)$ being replaced by
$\cC_{\cK_{a}}(\bar\nu_0)$ and $\cC_{\cV}(\bu)$,
respectively), then $\cG^{(2)}(x_1,u_1)\cap
\cQ^{(2)}(x_1)\cap \cE^{(2)}(x_1)\neq
\emptyset$.
\end{proposition}
\begin{proof}
If $\cG^{(1)}\cap\cQ^{(1)}\cap\cE^{(1)}\neq
\emptyset$ (with $\cT_{\cK_{a}}(\bar\nu_0)$ and
$\cT_{\Phi}(\bu)$ being replaced by
$\cC_{\cK_{a}}(\bar\nu_0)$ and $\cC_{\cV}(\bu)$,
respectively), then there exists $\hat
x_{1}(\cd)\in
\cG^{(1)}\cap\cQ^{(1)}\cap\cE^{(1)}$ with the
initial datum $\hat\nu_1\in
\cC_{\cK_{a}}(\bar\nu_0)$ and the control $\hat
u_1(\cd)\in \cC_{\cV}(\bu)$.

Since $T^{b(2)}_{\cK_{a}}(\bar\nu_0, \nu_1)$ and
$T^{b(2)}_{\cV}(\bu,v)$ are nonempty, they
contain some  nonempty convex  subsets
$\cW^1(\bar\nu_0,  \nu_1)$ and $\cM^1(\bu,u_1)$,
respectively.

Put\vspace{-0.2cm}
$$
\cW(\bar\nu_0, \nu_1)\=\cC_{\cK_{a}}(\bar\nu_0)+
\cW^1(\bar\nu_0,  \nu_1), \qq
\cM(\bu,u_1)\=\cC_{\cV}(\bu)+  \cM^1(\bu,u_1).
$$
It follows from Lemma 2.4 in \cite{FHT2013} that
$ \cW(\bar\nu_0,  \nu_1)\subset
T^{b(2)}_{\cK_{a}}(\bar\nu_0, \nu_1)$ and
$\cM(\bu,u_1)\subset T^{b(2)}_{\cV}(\bu,u_1). $
Moreover, for every $\tilde\nu_{2}\in
\cW^1(\bar\nu_0, \nu_1)$, $\tilde u_2\in
\cM^1(\bu,u_1)$ and $\d\ge0$, we have $ \d\hat
\nu_1+\tilde\nu_{2}\in \cW(\bar\nu_0, \nu_1)$
and $\d\hat u_1+ \tilde u_2\in\cM(\bu,u_1).
$

Fixing $\d \ge 0$ and  letting
$x_{2,\d}(\cd)$(\resp $\tilde x_2$) be the
solution of (\ref{second order vari equ})
corresponding to $\d\hat \nu_1+\tilde \nu_2$
(\resp $\tilde \nu_2$) and $\d\hat u_1+ \tilde
u_2$ (\resp $\tilde u_2$), we have $
x_{2,\d}(\cd)=\d\hat x_{1}(\cd)+\tilde x_2(\cd).
$ It follows from Lemma \ref{estimate two of
varie qu} that\vspace{-0.2cm}
$$
|\tilde x_2|^{2}_{L^\infty_\dbF(0,T;L^2(\O;H))}
\leq C\big(|\tilde\nu_{2}|_H^2+|\hat
u_1|^{4}_{L^4_\dbF(0,T;H)}+ |\tilde
u_2|^{2}_{L^2_\dbF(0,T;H)}\big).
$$
Since $\hat x_{1}(\cd)\in
\cG^{(1)}\cap\cQ^{(1)}\cap\cE^{(1)}$, and
$\cI^0(\bar x)$ and $\dbI^0(\bar x,\hat x_1)$
are compact sets, for all sufficiently large
$\d$,\vspace{-0.2cm}
$$
\begin{array}{ll}\ds
\mE\inner{h_{x}(\bx(t))}{x_{2,\d}(t)}_H
+\frac{1}{2}\mE \inner{h_{xx}(\bx(t)) x_{1}(t)}{x_{1}(t)}_H+e(t)\\
\ns\ds= \d\mE\inner{h_{x}(\bx(t))}{\hat
x_1(t)}_H +\mE\inner{h_{x}(\bx(t))}{\tilde
x_2(t)}_H +\frac{1}{2}\mE \inner{h_{xx}(\bx(t))
x_{1}(t)}{ x_{1}(t)}_H+e(t)
\\
\ns\ds< 0, \q \forall \, t\in \dbI^0(\bar x,\hat
x_1),\vspace{-0.2cm}
\end{array}
$$
and  for every $j\in \dbI(\bar x,\hat x_1)$, and
all $\d$ sufficiently  large\vspace{-0.2cm}
$$
\begin{array}{ll}\ds
\mE\inner{g^{j}_{x}(\bx(T))}{x_{2,\d}(T)}_H +
\frac{1}{2}\mE
\inner{g^{j}_{xx}(\bx(T))x_{1}(T)}{x_{1}(T)}_H
\\
\ns\ds=\d\mE\inner{g^{j}_{x}(\bx(T))}{\hat
x_1(T)}_H +\mE\inner{g^{j}_{x}(\bx(T))}{\tilde
x_1(T)}_H +\frac{1}{2}\mE
\inner{g^{j}_{xx}(\bx(T))x_{1}(T)}{x_{1}(T)}_H<
0.
\end{array}\vspace{-0.2cm}
$$
Therefore, when $\d$ is large enough, $
x_{2,\d}(\cd)\in \cG^{(2)}(x_1,u_1)\cap
\cQ^{(2)}(x_1)\cap \cE^{(2)}(x_1)$. This yields
that $\cG^{(2)}(x_1,u_1)\cap \cQ^{(2)}(x_1)\cap
\cE^{(2)}(x_1)\neq \emptyset$
\end{proof}

\begin{proof}[Proof of Theorem \ref{TH second order integral
condition}]

If  $\dbI^0(\bar x,x_1) = \emptyset$, then
$\cQ^{(2)}(x_1)=L_{\dbF}^{2}(\Omega; C([0,T];
H))$. Hence,\vspace{-0.2cm}
$$
\cG^{(2)}(x_1,u_1)\cap \cQ^{(2)}(x_1)\cap
\cE^{(2)}(x_1)=\cG^{(2)}(x_1,u_1) \cap
\cE^{(2)}(x_1).\vspace{-0.2cm}
$$
In such case, without loss of generality, we can
ignore the constraint (\ref{constraints1}) and
put $\psi =0$. Thus, we  only need to consider
the case $\dbI^0(\bar x,x_1) \neq \emptyset$.

The proof is divided into five steps. In the
first four steps,  we deal with the special case
when $x_{2}(\cd)\in \cG^{(2)}(x_1,u_1)\cap
\cQ^{(2)}(x_1)\cap \cE^{(2)}(x_1)$. Then, in the
last step, we handle the general case.

\vspace{0.1cm}

{\bf Step 1:} Since $x_{2}(\cd)\in
\cG^{(2)}(x_1,u_1)\cap \cQ^{(2)}(x_1)\cap
\cE^{(2)}(x_1)$, $x_2(\cd)$ is a solution of the
equation (\ref{second order vari equ})
corresponding to some $(\nu_{2},u_2)\in
\cW(\bx_{0},\nu_0)\times \cM(\bu,u_1)$ such
that\vspace{-0.2cm}
$$
\mE\inner{g^0_{x}(\bx(t))}{x_2(t)}_H+\frac{1}{2}\mE
\inner{g^0_{xx}(\bx(t))x_{1}(t)}{x_{1}(t)}_H+e(t)<0,\q
\forall\; t\in \dbI^0(\bar x,x_1)\vspace{-0.2cm}
$$
and\vspace{-0.2cm}
$$
\mE\inner{g^{j}_{x}(\bx(T))}{x_2(T)}_H\!+\!\frac{1}{2}\mE
\inner{g^{j}_{xx}(\bx(T))x_{1}(T)}{x_{1}(T)}_H\!<\!0,\;\q
\forall\; j\in \dbI(\bar x,x_1).
$$

Let $\mu^\eps \in H$  and  $\eta^\eps(\cd)\in
L_{\dbF}^{4}(0,T; H_1)$  be such
that\vspace{-0.2cm}
$$
\begin{array}{ll}\ds
|\mu^\eps|=o(\eps^2),\qq\qq\q\;\;\nu^{\eps}_{0}\=\bx_{0}+\eps\nu_1
+\eps^2\nu_{2}+\mu^\eps
\in \cK_a,\\
\ns\ds |\eta^\eps|_{L_{\dbF}^{4}(0,T;
H_1)}=o(\eps^2),\qq u^{\eps}(\cd)\=\bu(\cd)+\eps
u_1(\cd)+\eps^{2}u_2(\cd)+\eta^\eps(\cd)
\in\mathcal{V}.
\end{array}\vspace{-0.2cm}
$$
Denote by $x^{\eps}(\cd)$ the solution of
\eqref{controlsys} corresponding to
$\nu^{\eps}_{0}$ and $u^{\eps}(\cd)$. By ({\bf
AS1})--({\bf AS7}) and Lemma \ref{estimate two
of varie qu}, for any $t\in [0,T]$, we
have\vspace{-0.2cm}
\begin{equation}\label{th5.1 expansion g xeps}
\begin{array}{ll}\ds
\mE g^0(x^{\eps}(t))\3n&=\ds\mE g^0(\bx(t))+\eps \mE\inner{g^0_{x}(\bx(t))}{x_{1}(t)}_H+ \eps^2\mE\inner{g^0_{x}(\bx(t))}{x_{2}(t)}_H \\
\ns&\ds\q
+\frac{\eps^2}{2}\mE\inner{g^0_{xx}(\bx(t))
x_{1}(t)}{x_{1}(t)}_H +o(\eps^2).
\end{array}\vspace{-0.2cm}
\end{equation}

{\bf Step 2:} Fix an arbitrary $\hat
t\in\dbI(\bar x,x_1)$. In this step, we prove
that  there exist $\delta(\hat t)>0$ and
$\alpha(\hat t)>0$ such that\vspace{-0.2cm}
\begin{equation}\label{xeps admi for I1}
\mE g^0(x^{\eps}(s))\le 0,\q\forall\, s\in (\hat
t-\delta(\hat t), \hat t+\delta(\hat t))\cap
[0,T],\;\;\forall\,\e \in [0,\alpha(\hat
t)].\vspace{-0.2cm}
\end{equation}

If \eqref{xeps admi for I1} is false, then for
any $\ell\in \dbN$, we can find
$\eps_{\ell}\in[0, 1/\ell]$ and $s_{\ell}\in
(\hat t- 1/\ell,\hat t+1/\ell)\cap [0,T]$ such
that\vspace{-0.2cm}
\begin{equation}\label{contodicition g xeps}
\mE
g^0(x^{\eps_{\ell}}(s_{\ell}))>0.\vspace{-0.2cm}
\end{equation}

We consider two different cases.

{\bf Case 1.1}. There exists a subsequence
$\{s_{\ell_k}\}_{k=1}^{\infty}$ of
$\{s_\ell\}_{\ell=1}^{\infty}$
satisfying\vspace{-0.2cm}
\begin{equation}\label{11.20-eq1}
\mE g^0(\bx(s_{\ell_k}))<0 \;\mbox{ and }\;
\mE\inner{g^0_{x}(\bx(s_{\ell_k}))}{x_{1}(s_{\ell_k})}_H>0,\;\forall\,
k=1,2,\cdots.\vspace{-0.2cm}
\end{equation}
By \eqref{th5.1 expansion g
xeps},\vspace{-0.2cm}
\begin{equation*}\label{th5.1 eq1}
\begin{array}{ll}\ds
\mE g^0(x^{\eps_{\ell_k}}(s_{\ell_k}))\\
\ns\ds=
\eps_{\ell_k}^2\!\Big(\mE\inner{g^0_{x}(\bx(s_{\ell_k}))}{
x_{2}(s_{\ell_k})}_H\!
+\!\frac{1}{2}\mE\inner{g^0_{xx}(\bx(s_{\ell_k}))
x_{1}(s_{\ell_k})}{x_{1}(s_{\ell_k})}_H\!-\!\dfrac{\big|\mE\inner{g^0_{x}(\bx(s_{\ell_k}))}{x_{1}(s_{\ell_k})}_H\big|^2}{4
\mE~ g^0(\bx(s_{\ell_k})) }
\\
\ns \ds  \qq\q
+\dfrac{o(\eps_{\ell_k}^2)}{\eps_{\ell_k}^2}\Big)+\mE
g^0(\bx(s_{\ell_k}))\Big(1+\dfrac{\eps_{\ell_k}\mE\inner{g_{x}(\bx(s_{\ell_k}))}{
x_{1}(s_{\ell_k})}_H}{2\mE
g^0(\bx(s_{\ell_k}))}\Big)^2.
\end{array}\vspace{-0.2cm}
\end{equation*}
Since $\hat t\in \dbI^0(\bar x,x_1)$ and
$x_2(\cd)\in \cQ^{(2)}(x_1)$, there exists
$\rho_0>0$ such that\vspace{-0.2cm}
$$
\mE\inner{g^0_{x}(\bx(\hat t))}{x_{2}(\hat t)}_H
+\frac{1}{2}\mE\inner{g^0_{xx}(\bx(\hat t))
x_{1}(\hat t)}{x_{1}(\hat t)}_H+e(\hat
t)<-\rho_0.\vspace{-0.2cm}
$$
Therefore, when $k$ is large
enough,\vspace{-0.2cm}
$$
\begin{array}{ll}\ds
\mE\inner{g^0_{x}(\bx(s_{\ell_k}))}{x_{2}(s_{\ell_k})}_H
\!+\!\frac{1}{2}\mE\inner{g^0_{xx}(\bx(s_{\ell_k}))
x_{1}(s_{\ell_k})}{ x_{1}(s_{\ell_k})}_H
\!+\!\dfrac{\big|\mE\inner{g^0_{x}(\bx(s_{\ell_k}))}{
x_{1}(s_{\ell_k})}_H\big|^2}{4\big|\mE
g^0(\bx(s_{\ell_k}))\big|}<\!-\frac{\rho_0}{2},
\end{array}\vspace{-0.2cm}
$$
which, together with (\ref{11.20-eq1}), implies
that $\mE g^0(x^{\eps_{\ell_k}}(s_{\ell_k}))\le
0$, provided that $k$ is large enough. This
contradicts (\ref{contodicition g xeps}).

\vspace{0.1cm}

{\bf Case 1.2}: There is no subsequence of
$\{s_\ell\}_{\ell=1}^\infty$ such that
(\ref{11.20-eq1}) holds.

Under this circumstance,\vspace{-0.2cm}
$$
\mE g^0(\bx(s_\ell))=0 \mbox{ or }
\mE\inner{g^0_{x}(\bx(s_\ell))}{x_{1}(s_\ell)}_H\leq0
\mbox{ for all sufficiently large
$\ell$.}\vspace{-0.2cm}
$$
If $s_\ell\notin \cI^0(\bx)$, we have $\mE
g^0(\bx(s_\ell))<0$. Thus,
$\mE\inner{g^0_{x}(\bx(s_\ell))}{x_{1}(s_\ell)}_H\le
0$. On the other hand, if $s_{\ell}\in
\cI^0(\bx)$, then $\mE g^0(\bx(s_\ell))=0$.
Since $x_{1}(\cd)\in  {\rm cl} \cQ^{(1)}$,
$\mE\inner{g^0_{x}(\bx(s_\ell))}{x_{1}(s_\ell)}_H\le
0$. In both cases,\vspace{-0.2cm}
\begin{equation}\label{th5.1 eq2}
\mE
g^0(\bx(s_\ell))+\eps_\ell\mE\inner{g^0_{x}(\bx(s_\ell))}{x_{1}(s_\ell)}_H\le0.\vspace{-0.2cm}
\end{equation}
Noting that $e(t)\ge0$ for all $t\in [0,T]$ and
$\dbI^0(\bar x,x_1)$ is compact, there exists
$\rho_2>0$ such that\vspace{-0.2cm}
$$
\mE\inner{g^0_{x}(\bx(t))}{x_{2}(t)}_H
+\frac{1}{2}\mE\inner{g^0_{xx}(\bx(t))x_{1}(t)}{
x_{1}(t)}_H<-\rho_2,\q \forall \, t\in
\dbI^0(\bar x,x_1).
$$
Since $s_\ell\to \hat t$ and $\hat t\in
\dbI(\bar x,x_1)$, when $\ell$ is large
enough,\vspace{-0.2cm}
$$
\mE\inner{g^0_{x}(\bx(s_\ell))}{x_{2}(s_\ell)}_H
+\frac{1}{2}\mE\inner{g^0_{xx}(\bx(s_\ell))
x_{1}(s_\ell)}{x_{1}(s_\ell)}_H<-\dfrac{\rho_2}{2}.
$$
Then, by (\ref{th5.1 expansion g xeps}) and
(\ref{th5.1 eq2}), for any sufficiently large
$\ell$,\vspace{-0.2cm}
$$
\begin{array}{ll}\ds
\mE g^0(x^{\eps_\ell}(s_\ell)) \3n&\ds\leq
\eps_\ell^2\mE\inner{g^0_{x}(\bx(s_\ell))}{
x_{2}(s_\ell)}_H
+\frac{\eps_\ell^2}{2}\mE\inner{g^0_{xx}(\bx(s_\ell))x_{1}(s_\ell)}{x_{1}(s_\ell)}_H +o(\eps_\ell^2)\\
\ns&\ds\leq\eps_\ell^2\Big(-\frac{\rho_2}{2}+\frac{o(\eps_\ell^2)}{\eps_\ell^2}
\Big)\le 0,\vspace{-0.2cm}
\end{array}
$$
which also contradicts (\ref{contodicition g
xeps}). This proves (\ref{xeps admi for I1}).

\vspace{0.1cm}

{\bf Step 3:} In this step, we prove that
 $(x^{\eps}(\cd), u^{\eps}(\cd))\in \cP_{ad}$,
provided that $\eps$ is sufficiently small.

By the compactness of $\dbI^0(\bar x,x_1)$, we
can find  $\{t_{\ell}\}_{\ell}^N\subset
\dbI^0(\bar x,x_1)$ ($N\in\dbN$) such
that\vspace{-0.2cm}
$$
\dbI^0(\bar x,x_1)\subset \bigcup_{\ell=1}^{N}
\big(t_{\ell}-\delta(t_\ell),
t_\ell+\delta(t_\ell)\big).\vspace{-0.2cm}
$$
Let $\eps_{1}\=\min\{\alpha(t_{\ell}),\;
\ell=1,2,\ldots,N\}$. Then we have
that\vspace{-0.2cm}
\begin{equation}\label{th5.1 eq3}
\mE g^0(x^{\eps}(s))\le 0,\q\forall\, s\in
\bigcup_{\ell=1}^{N}\big(t_{\ell}-\delta(t_\ell),
t_\ell+\delta(t_\ell)\big)\cap
[0,T],\;\forall\,\e \in
[0,\eps_1].\vspace{-0.2cm}
\end{equation}
Let $\cI_{0}^{c}\= \cI^0(\bar x)\setminus
\bigcup_{\ell=1}^{N} (t_{\ell}-\delta(t_\ell),
t_\ell+\delta(t_\ell))$. Since $\cI_{0}^{c}$ is
compact, we can find $\tilde{\delta} >0$ and
$\rho_{3}>0$ (independent of $t$) such
that\vspace{-0.2cm}
$$
\mE\inner{g^0_{x}(\bx(s))}{x_{1}(s)}_H
<-\rho_{3},\q \forall\, s\in (t-\tilde{\delta},
t+\tilde{\delta} )\cap[0,T],\q t\in
\cI_{0}^{c}.\vspace{-0.2cm}
$$
This, together with (\ref{th5.1 expansion g
xeps}), implies that there exists $\eps_2>0$
such that\vspace{-0.1cm}
\begin{equation}\label{th5.1 eq4}
\mE g^0(x^{\eps}(s)) \leq  0, \q \forall\,s\in
(t-\tilde{\delta} , t+\tilde{\delta}
)\cap[0,T],\;\,\forall\, t\in \cI_{0}^{c},\;\,
\forall\, \eps\in [0,\eps_2].\vspace{-0.1cm}
\end{equation}
Clearly,\vspace{-0.2cm}
$$
\cI^0(\bar
x)\subset\[\bigcup_{\ell=1}^{N}(t_{\ell}-\delta(t_\ell),
t_\ell+\delta(t_\ell)) \]\bigcup\[\bigcup_{t\in
\cI_{0}^{c}} (t-\tilde{\delta} ,
t+\tilde{\delta} ) \].\vspace{-0.2cm}
$$
Let $\delta_0>0$ be small enough  such
that\vspace{-0.2cm}
$$
\cI^0(\bar x)\subset\bigcup_{t\in \cI^0(\bar x)}
(t-\delta_{0}  , t+\delta_{0})
\subset\[\bigcup_{\ell=1}^{N}
(t_{\ell}-\delta(t_\ell), t_\ell+\delta(t_\ell))
\]\bigcup\[\bigcup_{t\in \cI_{0}^{c}}
(t-\tilde{\delta} , t+\tilde{\delta} )
\].\vspace{-0.2cm}
$$
Put $\eps_0=\min\{\eps_1,\eps_2\}$. It follows
from \eqref{th5.1 eq4} that \vspace{-0.2cm}
\begin{equation}\label{xeps admi for I0}
\mE g^0(x^{\eps}(s))\le 0,\q\forall\, s\in
(t-\delta_0, t+\delta_0)\cap
[0,T],\;\forall\;t\in \cI^0(\bar x),\; \forall\,
\eps\in[0,\eps_0].\vspace{-0.2cm}
\end{equation}

Set\vspace{-0.2cm}
$$
\cI^{cc}\= [0,T]\setminus\Big[\bigcup_{t\in
\cI^0(\bar x)} (t-\delta_{0}  , t+\delta_{0})
\Big].
$$
From the compactness of $\cI^{cc}$ and the
continuity if $\mE g^0(\bx(\cd))$ with respect
to $t$, we know that there exists $\rho_{4}>0$
such that\vspace{-0.2cm}
$$
\mE g^0(\bx(t))<-\rho_{4}, \q \forall\; t\in
\cI^{cc}.\vspace{-0.2cm}
$$
This, together with (\ref{th5.1 expansion g
xeps}), implies that for all sufficiently small
$\eps>0$,\vspace{-0.2cm}
\begin{equation}\label{th5.1 eq5}
\mE g^0(x^{\eps}(t)) \leq  0,\q \forall\; t\in
\cI^{cc}.\vspace{-0.2cm}
\end{equation}
Combining (\ref{xeps admi for I0}) and
(\ref{th5.1 eq5}), we conclude that
$x^{\eps}(\cd)$ satisfies the constraint
(\ref{constraints1}), provided that $\eps$ is
small enough.

\vspace{0.1cm}

By a similar argument, we can show that when
$\eps$ is small enough, $x^{\eps}(\cd)$
satisfies the constraint (\ref{constraints}).
This proves that $(x^{\eps}(\cd),
u^{\eps}(\cd))\in \cP_{ad}$, provided that
$\eps$ is sufficiently small.

\vspace{0.15cm}

{\bf Step 4:} By the optimality of
$(\bx(\cd),\bu(\cd))$ and the equality
$\mE\inner{h_{x}(\bx(T))}{ x_{1}(T)}_H=0$, we
have\vspace{-0.1cm}
\begin{equation}\label{2nd exp of phi}
\begin{array}{ll}\ds
0\3n&\le\ds\lim_{\eps\to 0^+}\frac{\mE h(x^{\eps}(T))-\mE h(\bx(T))}{\eps^2} \\
\ns&\ds=  \mE\inner{h_{x}(\bx(T))}{x_{2}(T)}_H +\frac{1}{2}\mE\inner{h_{xx}(\bx(T))x_{1}(T)}{x_{1}(T)}_H +\lim_{\eps\to 0^+}\frac{o(\eps^2)}{\eps^2} \\
\ns&\ds = \mE\inner{h_{x}(\bx(T))}{x_{2}(T)}_H
+\frac{1}{2}\mE\inner{h_{xx}(\bx(T))x_{1}(T)}{
x_{1}(T)}_H.
\end{array}\vspace{-0.2cm}
\end{equation}
From the definition of the transposition
solution of the equation \eqref{first ajoint
equ}, we get that\vspace{-0.2cm}
\begin{eqnarray}\label{dual P1y2}
&&\3n\3n\!\!\!\!
\mE \inner{y(T)}{x_{2}(T)}_H\nonumber\\
&&\3n\3n\!\!\!\!=\!
\inner{y(0)}{\nu_{2}}_H\!+\!\mE\!\int_{0}^{T}
\!\!\!\Big(\!\inner{y(t)}{a_{2}[t]u_2(t)}_H\!
+\!\!\frac{1}{2}\inner{y(t)}{a_{11}[t](x_{1}\!(t),x_{1}\!(t))}_H\!
+\!\inner{y(t)}{a_{12}[t](x_{1}\!(t),u_1\!(t))}_H \nonumber\\
&&\q\3n\3n\!\!\!\!
+\frac{1}{2}\inner{y(t)}{a_{22}[t](u_1(t),u_1(t))}_H+
\inner{Y(t)}{b_{2}[t]
u_2(t)}_{\cL_2}+\frac{1}{2} \inner{Y(t)}{
b_{11}[t](x_{1}(t),x_{1}(t))}_{\cL_2}
 \\
&&\q\3n\3n\!\!\!\! + \inner{Y(t)}{
b_{12}[t](x_{1}(t),u_1(t))}_{\cL_2} +\frac{1}{2}
\inner{Y(t)}{b_{22}[t](u_1(t),u_1(t))}_{\cL_2}\Big)dt+\mE\int_0^T\inner{x_2(t)}{d\psi(t)}_H.\nonumber
\vspace{-0.2cm}
\end{eqnarray}
This, together with the choice of $y(T)$,
implies that\vspace{-0.2cm}
\begin{eqnarray}\label{dual phix(T) y2}
&&\3n\3n\3n\!\!\ds
\mE\inner{h_{x}(\bx(T))}{x_{2}(T)}_H\nonumber
\\
&&\3n\3n\3n\!\!\ds=
-\inner{y(0)}{\nu_{2}}_H-\sum_{j\in\cI(\bar
x)}\l_j \inner{g^j_x(\bx(T))}{x_2(T)}_H
-\mE\int_{0}^{T}\inner{x_{2}(t)}{d\psi(t)}_H \nonumber\\
&&\3n\3n\3n\!\!\ds\q-\mE\int_{0}^{T}
\!\!\Big(\!\inner{y(t)}{a_{2}[t]u_2(t)}_H
\!+\!\frac{1}{2}\inner{y(t)}{
a_{11}[t](x_{1}(t),x_{1}(t))}_H\!+\!\inner{y(t)}{a_{12}[t](x_{1}(t),u_1(t))}_H
\\
&&\3n\3n\3n\!\!\ds
\qq\qq\;+\frac{1}{2}\inner{y(t)}{a_{22}[t](u_1(t),u_1(t))}_H+
\inner{Y(t)}{b_{2}[t] u_2(t)}_{\cL_2}
+\frac{1}{2} \inner{Y(t)}{
b_{11}[t](x_{1}(t),x_{1}(t))}_{\cL_2} \nonumber\\
&&\3n\3n\3n\!\!\ds \qq\qq\;+
\inner{Y(t)}{b_{12}[t](x_{1}(t),u_1(t))}_{\cL_2}
+\frac{1}{2} \inner{Y(t)}{
b_{22}[t](u_1(t),u_1(t))}_{\cL_2} \Big)dt \nonumber\\
&&\3n\3n\3n\!\!\ds= -\inner{y(0)}{\nu_{2}}_H
-\mE\int_{0}^{T}\inner{x_{2}(t)}{d\psi(t)}_H
-\mE\int_{0}^{T}\inner{\dbH_{u}[t]}{u_2(t)}_{H_1}dt \nonumber\\
&&\3n\3n\3n\!\!\ds\q-\frac{1}{2}\mE\!\int_{0}^{T}
\!\!\big(\!\inner{\dbH_{xx}[t]x_{1}(t)}{x_{1}(t)}_H
\!+\!2\!\inner{\dbH_{xu}[t]x_{1}(t)}{u_1(t)}_{H_1}\!\!+\!\inner{\dbH_{uu}[t]u_1(t)}{u_1(t)}_{H_1}
\!\big)dt.\nonumber
\end{eqnarray}
By the definition of the relaxed transposition
solution of \eqref{op-bsystem3}, we have that
\begin{equation}\label{dual P2 y1y1}
\begin{array}{ll}\ds
 \mE\inner{P(T)x_{1}(T)}{x_{1}(T)}_H\\
\ns\ds=  \inner{P(0) \nu_1}{\nu_1}_H
+\mE\int_{0}^{T}\big(2\inner{P(t)x_{1}(t)}{a_{2}[t]u_1(t)}_H+2
\inner{P(t)b_{1}[t]x_{1}(t)}
{b_{2}[t]u_1(t)}_{\cL_2}
 \\
\ns\ds\q + \inner{P(t)b_{2}[t]u_1(t)}
{b_{2}[t]u_1(t)}_{\cL_2} +\big\langle\big(
\widehat Q^{(0)} +
Q^{(0)}\big)\big(0,a_{2}[t]u_1(t),
b_{2}[t]u_1(t)\big),b_{2}[t]u_1(t)
\big\rangle_{\cL_2}
\\
\ns\ds\q
-\inner{\dbH_{xx}[t]x_{1}(t)}{x_{1}(t)}_H
\big)dt.
\end{array}\vspace{-0.2cm}
\end{equation}
This, together with (\ref{dual phix(T) y2}) and
(\ref{2nd exp of phi}), implies  (\ref{second
order integral condition}).

\vspace{0.15cm}

{\bf Step 5:} In this step, we handle the case
when $ x_{2}(\cd) \in \cG^{(2)}(x_1,u_1)\cap
{\rm cl} \cQ^{(2)}(x_1) \cap {\rm cl}
\cE^{(2)}(x_1)$.

Let $ \hat x_{2}(\cd)\in \cG^{(2)}(x_1,u_1)\cap
{\rm cl} \cQ^{(2)}(x_1) \cap {\rm cl}
\cE^{(2)}(x_1)$ with the corresponding
$\hat\nu_{2}\in \cW(\bar\nu_{0}, \nu_1)$ and
$\hat u_2(\cdot)\in \cM(\bu, u_1)$. For
$\theta\in(0,1)$, put\vspace{-0.2cm}
$$
x^{\theta}_{2}=(1-\theta)x_{2}+ \theta \hat
x_{2}.\vspace{-0.1cm}
$$
Noting that $\cW(\bar\nu_{0}, \nu_1)$ and
$\cM(\bu, u_1)$ are convex, $x^{\theta}_{2} $ is
the solution of the equation \eqref{second order
vari equ} with the initial datum\vspace{-0.2cm}
$$
\nu^{\theta}_{2}\= (1-\theta)\nu_{2}+\theta
\hat\nu_{2} \in \cW(\bar\nu_{0},
\nu_1)\vspace{-0.2cm}
$$
and the control\vspace{-0.2cm}
$$
u_2^{\theta}(\cdot)\=(1-\theta)
u_2(\cdot)+\theta \hat u_2(\cdot) \in \cM(\bu,
u_1).\vspace{-0.2cm}
$$
Then, it is easy to show that\vspace{-0.2cm}
$$
\lim_{\theta\to
0}x_2^{\theta}= x_{2} \mbox{ in }
L^{2}_{\dbF}(\Omega;C([0,T];H)).
$$
Furthermore, since $\hat x_{2}(\cd)\in
\cQ^{(2)}(x_1)\cap\cE^{(2)}(x_1)$, we have
$x^{\theta}_{2}(\cd)\in
\cQ^{(2)}(x_1)\cap\cE^{(2)}(x_1)$ for $\th\neq
0$. From \textbf{Step 1}, we deduce
that\vspace{-0.2cm}
$$
\begin{array}{ll}\ds
\langle y(0),\nu_2^{\theta} \rangle_H+
\frac{1}{2}\langle P(0) \nu_1, \nu_1 \rangle_H
+\sum_{j\in \cI(\bar x)}\mE
\big\langle\lambda_{j}g^{j}_{x}(\bx(T)),x_{2}^{\theta}(T)\big\rangle_H
\\
\ns\ds  +\mE \int_0^T \(\big\langle \dbH_{u}[t],
u_2^{\theta}(t)\big\rangle_{H_{1}} +
\frac{1}{2}\big\langle \dbH_{uu}[t]u_1(t),
u_1(t)\big\rangle_{H_{1}}+
\frac{1}{2}\big\langle
b_{2}[t]^{*}P(t)b_{2}[t]u_1(t),
u_1(t)\big\rangle_{H_{1}}
\\
\ns\ds \qq\qq + \big\langle\big(\dbH_{xu}[t] +
a_{2}[t]^*
P(t) + b_{2}[t]^*P(t) b_{1}[t]\big)x_{1}(t), u_1(t)\big\rangle_{H_{1}} \\
\ns\ds \qq\qq+ \frac{1}{2}\big\langle\big(
\widehat Q^{(0)}\! +\!
Q^{(0)}\big)\big(0,a_{2}[t]u_1(t),
b_{2}[t]u_1(t)\big),b_{2}[t]u_1(t)
\big\rangle_{\cL_2}\)dt\! + \mE\!\int_0^T\! \lan
x^{\theta}_2(t),d\psi(t)\ran_H \leq 0.
\end{array}
$$
Letting $\theta\to 0$ in the above inequality,
we obtain (\ref{second order integral
condition}). This completes the proof of Theorem
\ref{TH second order integral condition}.
\end{proof}

\begin{remark}
The second order necessary condition is only
valid for $\cY(\bar x,\bar u)$ (recall
\eqref{8.18-eq6} for the definition) being
nonempty. If $\cE^{(1)}_T \neq\emptyset$,
$U=H_1$, \eqref{first vari equ} is exactly
controllable and there are no state constraints,
then $\cY(\bar x,\bar u)\neq\emptyset$. However,
to enjoy the exact controllability property, one
needs some restrictive conditions
(e.g.\cite{Luqi4,Luqi6,LYZ}).
\end{remark}

Next, we give another second order necessary
condition.
\begin{theorem}\label{TH second order conditionbis}
Suppose that {\bf(AS1)}--{\bf(AS9)} hold and
$(\bx(\cd),\bu(\cd),\bar\nu_{0})$ be an optimal
triple of {\bf Problem (OP)}. Let
$\Phi(t,\omega)=\cC_{U}(\bar u(t,\omega))$.
Assume that $\mE |g^0_{x}(\bx(t))|_H\neq 0$ for
all $t\in \cI^0(\bar x)$. Let
$(x_{1},u_1,\nu_1)\in \Upsilon(\bx,\bu)$ and
suppose that $e(\cdot)$ (defined by
(\ref{e(t)})) is  bounded on $\dbI^0(\bar
x,x_1)$. Let $\cW(\bar\nu_{0},\nu_1)$ $\subset
T^{b(2)}_{\cK_a}(\bx_{0},\nu_1)$ and $\cM(\bu,
u_1) \subset T^{b(2)}_{\mathcal{V}}(\bu,u_1)$ be
convex. Then there exist $\lambda_0 \in\{0,1\}$,
$\lambda_{j}\ge 0$ for all $j\in \cI(\bar x)$
and $\psi \in \big(\cQ^{(1)}\big)^{-}$ such that
the solution $(y,Y)$ of (\ref{first ajoint equ})
with $\ds y_T=-\l_0h_x(\bar x(T)) - \sum_{j\in
\cI(\bar x)}\l_jg_x^j(\bar x(T))$ and ${\cal
I}(\bar x)$ replaced by $ {\dbI}(\bar x,x_1)$
satisfies the first order condition
(\ref{th2-eq1}), and for any $x_{2}(\cd)\in
\cG^{(2)}(x_1,u_1)$ with the corresponding
$\nu_{2}\in \cW(\bar\nu_{0}, \nu_1)$ and
$u_2(\cdot)\in \cM(\bu, u_1)$, the second order
necessary condition \eqref{second order integral
condition} holds true, where
$(P(\cd),Q^{(\cd)},\wh Q^{(\cd)})$ is the
relaxed transposition solution of
(\ref{op-bsystem3}) with $\ds P(T)=-\lambda_0
h_{xx}(\bx(T))-\sum_{j=1}^n \l_j
g^j_{xx}(\bx(T))$.
\end{theorem}
\begin{proof}

If either $\cW(\bar\nu_{0},\nu_1)$ or  $\cM(\bu,
u_1) $ is empty, then by  Theorem \ref{th2}, we
get the desired result. Therefore, in the rest
of the proof, we assume that these two sets are
nonempty. Put\vspace{-0.2cm}
$$
\wt{\cW}(\bx_{0},\nu_1)\= \cC_{ \cK_a }(\bx_{0})
+ \cW(\bx_{0},\nu_1),\q \wt{\cM}(\bu, u_1)\=
\cT_\Phi(\bu)+\cM(\bu, u_1),\vspace{-0.2cm}
$$
where\vspace{-0.2cm}
$$
\cT_\Phi(\bu)\= \{u \in
L^{4}_{\dbF}(0,T;H_1)\;|\; u(t,\omega) \in
\cC_U(\bu(t,\omega)) \;\; \mbox{\rm a.e. in }
[0,T] \times \Omega\}.\vspace{-0.2cm}
$$
By Lemma \ref{lm7}, $\cT_\Phi(\bu) \subset
\cC_{\mathcal{V}} (\bar u)$. Thus, by Lemma 2.4
from [17], $\wt{\cM}(\bu, u_1) \subset
T^{b(2)}_{\mathcal{V}}(\bu,u_1) $.

We divide the rest of the proof into two steps.
In \textbf{Step 1}, we handle the case when
$\dbI^0(\bar x,x_1) = \emptyset$. In
\textbf{Step 2}, we deal with the case when
$\dbI^0(\bar x,x_1)\neq\emptyset$.

\medskip

{\bf Step 1}. If  $\dbI^0(\bar x,x_1) =
\emptyset$, then
$\cQ^{(2)}(x_1)=L_{\dbF}^{2}(\Omega; C([0,T];
H))$ and \vspace{-0.2cm}
$$\cG^{(2)}(x_1,u_1)\cap
\cQ^{(2)}(x_1)\cap\cE^{(2)}(x_1)=\cG^{(2)}(x_1,u_1)
\cap\cE^{(2)}(x_1).\vspace{-0.2cm}
$$

Fix $(x_{1}(\cd),u_1(\cd),\nu_1)\in \cZ(\bar
x,\bar u)$ (recall \eqref{8.18-eq7} for the
definition of $\cZ(\bar x,\bar u)$). Consider
the following two different cases:

\vspace{0.2cm}

{\bf Case 1.1}: $\dbI(\bar x,x_{1})=\emptyset$.

In this context,\vspace{-0.2cm}
$$
\mE\langle h_x(\bar x(T)),x_{1}(T) \rangle_H=0,
\qq \mE\langle g_x^j(\bar x(T)),x_{1}(T)
\rangle_H<0,\q \forall\, j\in \cI(\bar
x).\vspace{-0.2cm}
$$
Then for any $\nu_2\in \wt\cW(\bar
\nu_0,\nu_1)$, $u_2\in\wt\cM(\bar u,u_1)$ and
$\e>0$, there exist $\nu^\e\in H$ and $v^\e \in
L^4_\dbF(0,T;H_1)$ such that\vspace{-0.2cm}
$$
|\nu^\e|_H=o(\e^2),\qq \nu_0^\e\=\bar \nu_0 +
\e\nu_1+\e^2\nu_2+\nu^\e\in \cK_a\vspace{-0.2cm}
$$
and\vspace{-0.2cm}
$$
|v^\e|_{L^4_\dbF(0,T;H_1)}=o(\e^2),\qq
u^\e\=\bar u + \e u_1+\e^2 u_2 + v^\e\in
\cV.\vspace{-0.2cm}
$$
Let $x^\e(\cd)$ be the solution of the control
system \eqref{controlsys} with the initial datum
$\nu_0^\e$ and  the control $u^\e(\cd)$.
Put\vspace{-0.2cm}
$$
h_{11}^{\e}(T)\=\int_{0}^{1}
(1-\theta)h_{xx}(\bx(T) + \theta\delta
x^{\e}(T))d\theta.\vspace{-0.2cm}
$$
By Lemma \ref{estimate two of varie qu}, there
is $\rho<0$ such that for each $j\in \cI(\bar
x)$ and all sufficiently small
$\e>0$,\vspace{-0.2cm}
\begin{equation*}\label{8.18-eq10}
\begin{array}{ll}\ds
\mE g^j(x^\e(T))\3n&\ds = \mE g^j(\bar x(T)) +
\e\mE\langle g^j_x(\bar x(T)), x_{1}(T)
\rangle_H + \e^2\mE\langle g^j_x(\bar x(T)),
x_{2}(T)
\rangle_H\\
\ns&\ds \q +\frac{\e^2}{2}\mE\langle
g^j_{xx}(\bar x(T))x_{1}(T), x_{1}(T) \rangle_H
+
o(\e^2)\\
\ns&\ds = \e \Big( \mE\langle g^j_x(\bar x(T)),
x_{1}(T) \rangle_H + \e \mE\langle g^j_x(\bar
x(T)), x_{2}(T) \rangle_H \\
\ns&\ds\q + \frac{\e}{2}\mE\langle g^j_{xx}(\bar
x(T))x_{1}(T), x_{1}(T) \rangle_H +
o(\e)\Big)<\e\rho< 0\vspace{-0.2cm}
\end{array}
\end{equation*}
and, for each $j\notin \cI(\bar x)$,  $\mE
g^j(x^\e(T))  = \mE g^j(\bar x(T)) + O(\e) \leq
\rho+ O(\e)$. Consequently,
$(x^\e(\cd),u^\e(\cd))\in\cP_{ad}$.

Direct computations yield
$$
\begin{array}{ll}\ds
\frac{\cJ(u^{\e})-\cJ(\bu)}{\e^2}\\
\ns \ds=  \frac{1}{\e^2}\mE
\Big(\inner{h_{x}(\bx(T))}{\delta
x^{\e}(T)}_H+\frac{1}{2}\inner{h_{11}^{\e}(\bx(T))\delta
x^{\e}(T)}{\delta x^{\e}(T)}_H
\Big)\\
\ns \ds =  \mE\Big(\frac{1}{\e}
\inner{h_{x}(\bar{x}(T))}{x_{1}(T)}_H +
\inner{h_{x}(\bx(T))}{x_{2}(T)}_H
+\frac{1}{2}\inner{h_{xx}(\bx(T))x_{1}(T)}{x_{1}(T)}_H\Big)
+ \rho_{2}^{\e},
\end{array}
$$
where\vspace{-0.2cm}
$$
\begin{array}{ll}\ds
\rho_{2}^{\eps} \3n&\ds=  \mE
\Big(\frac{1}{2}\Big\langle
h_{11}^{\eps}(\bx(T))\frac{\delta
x^{\eps}(T)}{\eps},\frac{\delta
x^{\eps}(T)}{\eps}\Big\rangle_{H}-\frac{1}{2}\inner{h_{xx}(\bx(T))x_{1}(T)}{x_{1}(T)}_{H}\Big).
\end{array}
$$
Similar to the proof of Lemma \ref{estimate two
of varie qu}, we can show that $\ds\lim_{\e\to
0^+}\rho_{2}^{\e}=0$. Therefore,\vspace{-0.2cm}
\begin{equation}\label{taylorexpconvex}
\begin{array}{ll}\ds
0 \leq\lim_{\eps\to 0^+}
\frac{\cJ(u^{\eps}(\cdot))-\cJ(\bu(\cdot))}{\eps^2}
= \mE\Big(\inner{h_{x}(\bx(T))}{x_{2}(T)}_H
+\frac{1}{2}\inner{h_{xx}(\bx(T))x_{1}(T)}{x_{1}(T)}_H\Big). \\
\end{array}
\end{equation}
It follows from the definition of the
transposition solution of \eqref{first ajoint
equ}  that
\begin{equation}\label{hxty2}
\begin{array}{ll}\ds
\mE\inner{y(T)}{x_{2}(T)}_H\\
\ns\ds =
\mE\inner{y(0)}{\nu_{2}}_H+\mE\int_{0}^{T}
\Big(\inner{y(t)}{a_{2}[t]u_2(t)}_H +
\frac{1}{2}\inner{y(t)}{a_{11}[t]\big(x_{1}(t),x_{1}(t)\big)}_H
\\
\ns\ds\q
+
\inner{y(t)}{a_{12}[t]\big(x_{1}(t),u_1(t)\big)}_{H_1}+
\frac{1}{2}\inner{y(t)}{a_{22}[t]\big(u_1(t),u_1(t)\big)}_{H_1}
\\
\ns\ds\q+ \inner{Y(t)}{b_{2}[t]u_2(t)}_{\cL_2}
+\frac{1}{2}\inner{Y(t)}{b_{11}[t]\big(x_{1}(t),x_{1}(t)\big)}_{\cL_2}
\\
\ns\ds\q +
\inner{Y(t)}{b_{12}[t]\big(x_{1}(t),u_1(t)\big)}_{\cL_2}+
\frac{1}{2}\inner{Y(t)}{b_{22}[t]\big(u_1(t),u_1(t)\big)}_{\cL_2}
\Big)dt.
\end{array}
\end{equation}
By the definition of the relaxed transposition
solution of \eqref{op-bsystem3}, we have
\begin{equation}\label{hxxty12}
\begin{array}{ll}\ds
\mE\inner{P(T)x_{1}(T)}{x_{1}(T)}_{H}\\
\ns\ds=
\mE\inner{P(0)\nu_1}{\nu_1}_{H}+\mE\int_{0}^{T}\Big(2\inner{P(t)x_{1}(t)}{a_{2}[t]u_1(t)}_{H}
+2\inner{P(t)b_{1}[t]x_{1}(t)}{b_{2}[t]u_1(t)}_{H}
\\
\ns\ds\q
+\inner{P(t)b_{2}[t]u_1(t)}{b_{2}[t]u_1(t)}_{H}
-\inner{\dbH_{xx}(t)x_{1}(t)}{x_{1}(t)}_{H} \Big)dt\\
\ns\ds\q +\mE \int_0^T \big\langle \widehat
Q^{(0)}(0,a_2u_1, b_2u_1)(t) + Q^{(0)}(0,a_2u_1,
b_2u_1)(t), b_2(t)u_1(t) \big\rangle_{\cL_2}dt.
\end{array}
\end{equation}
Let $\l_0=1$, $\l_j=0$  for all $j\in \cI(\bar
x)$ and $\psi=0$. It follows from
(\ref{taylorexpconvex})--(\ref{hxxty12})  that
\begin{eqnarray*} \ds
&& \3n\3n\3n 0 \geq\mE\inner{y(0)}{\nu_{2}}_{H}
+\frac{1}{2}\mE\inner{P(0)\nu_1}{\nu_1}_{H}+\mE\int_{0}^{T}\Big[
\inner{y(t)}{a_{2}[t]u_2(t)}_{H}
+\inner{Y(t)}{b_{2}[t]u_2(t)}_{\cL_2}
\\
&&\3n\3n\3n\qq
+\frac{1}{2}\Big(\inner{y(t)}{a_{22}[t]\big(u_1(t),u_1(t)\big)}_{H}\!+\!\inner{Y(t)}{b_{22}[t]\big(u_1(t),u_1(t)\big)}_{\cL_2}
\!\!+\! \inner{P(t)
b_{2}[t]u_1(t)}{b_{2}[t]u_1(t)}_{H}\!\Big)
\\
&&\3n\3n\3n\qquad +\!
\inner{y(t)}{a_{12}[t]\big(x_{1}(t),u_1(t)\big)}_{H} +\!\inner{Y(t)}{b_{12}[t]\big(x_{1}(t),u_1(t)\big)}_{\cL_2}\\
&&\3n\3n\3n\qquad
+\!\inner{a_{2}[t]^{*}P(t)x_{1}(t)}{u_1(t)}_{H_1}
\!\!+\!\inner{b_{2}[t]^{*}P(t)b_{1}[t]x_{1}(t)}{u_1(t)}_{H_1}
\!\Big]dt\\
&&\3n\3n\3n \q+\frac{1}{2}\mE \int_0^T
\big\langle \widehat Q^{(0)}(0,a_{2}u_1,
b_{2}u_1)(t) + Q^{(0)}(0,a_{2}v, b_{2}u_1)(t),
b_{2}[t]u_1(t) \big\rangle_{\cL_2}dt
\\
&&\3n\3n\3n=
\mE\inner{y(0)}{\nu_{2}}_{H}+\frac{1}{2}\mE\inner{P(0)\nu_1}{\nu_1}_{H}
\\
&&\3n\3n\3n\q  +\mE\!\int_0^T\!\! \( \big\langle
\dbH_{u}[t], u_2(t)\big\rangle_{H_{1}}\!\! +\!
\frac{1}{2}\big\langle \dbH_{uu}(t)u_1(t),
u_1(t) \big\rangle_{H_{1}}\!\! +\!
\frac{1}{2}\big\langle
b_{2}[t]^{*}P(t)b_{2}[t]u_1(t),
u_1(t)\big\rangle_{H_{1}}\!
\)dt\\
&&\3n\3n\3n \q + \mE\!\int_0^T
\big\langle\big(\dbH_{xu}(t) + a_{2}[t]^*
P(t) + b_{2}[t]^*P(t) b_{1}[t]\big)y(t), u_1(t)\big\rangle_{H_{1}} dt  \\
&&\3n\3n\3n\q +\frac{1}{2}\mE \int_0^T
\big\langle \widehat Q^{(0)}(0,a_{2}[t]u_1,
b_{2}[t]u_1)(t) + Q^{(0)}(0,a_{2}[t]u_1,
b_{2}[t]u_1)(t), b_{2}[t]u_1(t)
\big\rangle_{\cL_2}dt.
\end{eqnarray*}

\vspace{0.2cm}

{\bf Case 1.2}: $\dbI(\bar
x,x_{1})\neq\emptyset$.

First, we claim that\vspace{-0.2cm}
\begin{equation}\label{8.3-eq10}
\cE^2(x_{1})\cap
\cL^{(2)}(x_{1})\cap\cG^2(x_{1},u_1)=\emptyset.
\end{equation}
Indeed, if \eqref{8.3-eq10} was false, then
there would exist $\nu_2\in \wt\cW(\bar
\nu_0,\nu_1)$ and $u_2(\cd)\in \wt\cM(\bar
u,u_1)$ such that for some $\rho<0$, the
corresponding solution $x_{2}(\cd)$ of
\eqref{second order vari equ}
satisfies\vspace{-0.2cm}
$$
\mE\big\langle g_x^j(\bar x(T)),x_{2}(T)
\big\rangle_H + \frac{1}{2}\mE\big\langle
g_{xx}^j(\bar x(T))x_{1}(T),x_{1}(T)
\big\rangle_H<2\rho,\qq \forall\, j\in \dbI(\bar
x,x_{1})\vspace{-0.2cm}
$$
and\vspace{-0.2cm}
$$
\begin{array}{ll}\ds
\mE\big\langle h_x(\bar x(T)),x_{2}(T)
\big\rangle_H + \frac{1}{2}\mE\big\langle
h_{xx}(\bar x(T))x_{1}(T),x_{1}(T)
\big\rangle_H<2\rho.
\end{array}
$$

Let $\nu^\e\in H$ and $v^\e\in
L^4_\dbF(0,T;H_1)$ be such that\vspace{-0.2cm}
$$
|\nu^\e|_H=o(\e^2),\qq \nu_0^\e\=\bar \nu_0 +
\e\nu_1+\e^2\nu_2+\nu^\e\in \cK_a\vspace{-0.2cm}
$$
and\vspace{-0.2cm}
$$
|v^\e|_{L^4_\dbF(0,T;H_1)}=o(\e^2),\qq
u^\e(\cd)\=\bar u(\cd) + \e u_1(\cd)+\e^2
u_2(\cd)+ v^\e(\cd)\in \cV.\vspace{-0.2cm}
$$
Let $x^\e(\cd)$ be the solution of the control
system \eqref{controlsys} with the initial datum
$\nu_0^\e$ and the control $u^\e(\cd)$. Similar
to {\bf Case 1.1}, one can prove that for every
$j\notin \dbI(\bar x,x_{1})$ and  for all $\e>0$
small enough, $\mE g^j(x^\e(T))\leq 0$.
Meanwhile, by Lemma \ref{estimate two of varie
qu}, for any $j\in \dbI(\bar x,x_{1})$, and for
all sufficiently small $\e>0$,\vspace{-0.2cm}
\begin{equation*}\label{8.18-eq11.1}
\begin{array}{ll}\ds
\mE g^j(x^\e(T))\3n&\ds = \mE g^j(\bar x(T)) +
\e\mE\big\langle g^j_x(\bar x(T)), x_{1}(T)
\big\rangle_H + \e^2\mE\big\langle g^j_x(\bar
x(T)), x_{2}(T)
\big\rangle_H\\
\ns&\ds \q +\frac{\e^2}{2}\mE\big\langle
g^j_{xx}(\bar x(T))x_{1}(T), x_{1}(T)
\big\rangle_H +
o(\e^2)\\
\ns&\ds = \e^2 \Big(\mE\big\langle g^j_x(\bar
x(T)), x_{2}(T) \big\rangle_H +
\frac{1}{2}\mE\big\langle g^j_{xx}(\bar
x(T))x_{1}(T), x_{1}(T) \big\rangle_H+ \frac{o(\e^2)}{\e^{2}}\Big)\\
\ns&\ds <\e^2\rho< 0.
\end{array}\vspace{-0.2cm}
\end{equation*}
This proves that $(x^\e(\cd),u^\e(\cd))\in
\cP_{ad}$.

On the other hand, for all sufficiently small
$\e>0$,\vspace{-0.1cm}
\begin{equation*}\label{8.18-eq12}
\begin{array}{ll}\ds
\mE h(x^\e(T))\3n&\ds = \mE h(\bar x(T)) +
\e\mE\big\langle h_x(\bar x(T)), x_{1}(T)
\big\rangle_H + \e^2\mE\big\langle h_x(\bar
x(T)), x_{2}(T)
\big\rangle_H\\
\ns&\ds \q +\frac{\e^2}{2}\mE\big\langle
h_{xx}(\bar x(T))x_{1}(T), x_{1}(T)
\big\rangle_H +
o(\e^2)\\
\ns&\ds = \mE h(\bar x(T)) \!+\!\e^2 \!\Big(
\mE\big\langle h_x(\bar x(T)), x_{2}(T)
\big\rangle_H \!  +\! \frac{1}{2}\mE\big\langle
h_{xx}(\bar x(T))x_{1}(T), x_{1}(T)
\big\rangle_H \!+\!
\frac{o(\e^2)}{\e^{2}}\Big)\\
\ns&\ds<\mE h(\bar x(T)) +\e^2\rho<\mE h(\bar
x(T)).
\end{array}\vspace{-0.2cm}
\end{equation*}
This contradicts the optimality of $(\bar
x(\cd),\bar u(\cd),\bar\nu_0)$. Hence,
\eqref{8.3-eq10} holds.

\vspace{0.2cm}

Next, we consider two subcases (recall
\eqref{8.18-eq8}, \eqref{E2} and \eqref{L2} for
the definitions of $\cG^{(2)}(x_1,u_1)$,
$\cE^{(2)}(x_1)$ and $\cL^{(2)}(x_1)$).

\vspace{0.2cm}

{\bf Case 1.2.1}. $\cL^{(2)}(x_1)\cap
\cE^{(2)}(x_1)\neq\emptyset$.

\vspace{0.2cm}

Under these circumstances,
$\G\big(\cL^{(2)}(x_1)\big)\cap
\G\big(\cE^{(2)}(x_1)\big)\neq\emptyset$. Since
$\G\big(\cL^{(2)}(x_1)\big)\cap
\G\big(\cE^{(2)}(x_1)\big)\cap
\G\big(\cG^{(2)}(x_1,u_1)\big)=\emptyset$, by
the Hahn-Banach separation  theorem, we can find
a nonzero $\xi\in L^2_{\cF_T}(\Om;H)$ such
that\vspace{-0.2cm}
$$
\sup_{\a\in \G(\cL^{(2)}(x_1))\cap
\G(\cE^{(2)}(x_1))}\mE\langle
\xi,\a\rangle_H\leq \inf_{\b\in
\G(\cG^{(2)}(x_1,u_1))}\mE\langle
\xi,\b\rangle_H.\vspace{-0.2cm}
$$
By Lemma \ref{lm5}, there exists\vspace{-0.2cm}
$$
\begin{array}{ll}\ds
\a_0\in {\rm cl}
\big(\G\big(\cL^{(2)}(x_1)\big)\bigcap
\G\big(\cE^{(2)}(x_1)\big)\big) = {\rm cl}
\G\big(\cL^{(2)}(x_1)\big)
 \bigcap \big(\bigcap_{j\in \dbI(\bar x,x_{1})}
{\rm cl} \G\big(\cE^{2,j}(x_{1}(T))\big) \big)
\end{array}\vspace{-0.2cm}
$$
such that\vspace{-0.2cm}
$$
\mE \langle\xi,\a_0\rangle_H=\sup_{\a\in
\G(\cL^{(2)}(x_1))\cap \G(\cE^{(2)}(x_1))}\mE
\langle\xi,\a \rangle_H.
$$
Put\vspace{-0.2cm}
$$\dbI_0(\bar x,x_{1}) \=\Big\{j\in \dbI(\bar
x,x_{1})\Big|\,  \mE\big\langle g_x^j(\bar
x(T)), \a_0\big\rangle_H +
\frac{1}{2}\mE\big\langle g_{xx}^j(\bar
x(T))x_{1}(T),
x_{1}(T)\big\rangle_H=0\Big\}.\vspace{-0.1cm}
$$
By Lemma \ref{lm5}, for every $j\in \dbI_0(\bar
x,x_{1})$, there exists $\l_j\geq 0$ such
that\vspace{-0.2cm}
\begin{equation}\label{7.2-eq1}
\xi=\l_0 h_x(\bar x(T)) + \sum_{j\in \dbI_0(\bar
x,x_{1})}\l_j g_x^j(\bar x(T)),\vspace{-0.2cm}
\end{equation}
where $\l_0=0$ if $\ds\mE\langle h_x(\bar x(T)),
\a_0\rangle_H + \frac{1}{2}\mE\langle
h_{xx}(\bar x(T))x_{1}(T), x_{1}(T)\rangle_H<0$.
Then \eqref{7.2-eq1} yields\vspace{-0.2cm}
$$
\begin{array}{ll}\ds
\mE\langle\xi, \a_0\rangle_H \!=\!
-\frac{1}{2}\(\l_0\mE\big\langle h_{xx}(\bar
x(T))x_{1}(T), x_{1}(T)\big\rangle_H+\!\!
\sum_{j\in \dbI_0(\bar
x,x_{1})}\!\l_j\mE\big\langle g_{xx}^j(\bar
x(T))x_{1}(T), x_{1}(T)\big\rangle_H\).
\end{array}\vspace{-0.2cm}
$$
Setting\vspace{-0.2cm}
$$
y(T)=-\l_0 h_x(\bar x(T)) -\sum_{j\in
\dbI_0(\bar x,x_{1})}\l_j g_x^j(\bar
x(T))\vspace{-0.2cm}
$$
and\vspace{-0.2cm}
$$
P(T)=-\l_0 h_{xx}(\bar x(T)) -\sum_{j\in
\dbI_0(\bar x,x_{1})}\l_j g_{xx}^j(\bar
x(T)),\vspace{-0.2cm}
$$
we find that for any $x_{2}(T)\in
\cG^{(2)}_T(x_1,u_1)$,\vspace{-0.2cm}
$$
\begin{array}{ll}\ds
\frac{1}{2}\mE\langle P(T)x_{1}(T),x_{1}(T)
\rangle_H\\
\ns\ds = -\frac{1}{2}\(\l_0\mE\langle
h_{xx}(\bar x(T))x_{1}(T), x_{1}(T)\rangle_H+
\sum_{j\in \dbI_0(\bar x,x_{1})}\l_j\mE\langle
g_{xx}^j(\bar x(T))x_{1}(T), x_{1}(T)\rangle_H\)\\
\ns\ds = \mE\langle\xi, \a_0\rangle_H\leq
\mE\langle y(T),x_{2}(T) \rangle_H.
\end{array}
$$
This, together with \eqref{hxty2} and
\eqref{hxxty12}, implies \eqref{second order
integral condition}.

\vspace{0.2cm}

{\bf Case 1.2.2}.  $\cL^{(2)}_T(x_1)\cap
\cE^{(2)}_T(x_1)=\emptyset$.

\vspace{0.2cm}

For simplicity of notations, we put
$g^{n+1}(\cd)\!=\!h(\cd)$,
$\mathbf{I}\!=\!\{n+1\}\cup \dbI(\bar x,x_{1})$,
$\cE^{(2,n+1)}(x_{1})\!=\!\cL^{(2)}(x_1)$ and
$\cE^{(2,n+1)}_T(x_{1})=\cL^{(2)}_T(x_1)$.

If there exists  $j\in \mathbf{I}$ such that
$\cE^{(2,j)}_T(x_{1})=\emptyset$, then
$g_x^j(\bar x(T))=0$, $\dbP$-a.s.
and\vspace{-0.1cm}
\begin{equation}\label{8.18-eq13}
\mE \big\langle g_{xx}^j(\bar
x(T))x_{1}(T),x_{1}(T)\big\rangle_H\geq
0.\vspace{-0.1cm}
\end{equation}
Let $\l_j=1$ and $\l_k=0$ for $k\in
\mathbf{I}\setminus\{j\}$. Then
$\ds\l_jg_x^j(\bar x(T)) + \sum_{k\in
\mathbf{I}\setminus\{j\}}\l_kg_x^k(\bar
x(T))=0$. Let $y(T)=0$ and $P(T)=-g_{xx}^j(\bar
x(T))$. It is easy to see that $(y(\cd),Y(\cd))=
(0,0)$, $\dbH(\cd)= 0$, $\dbH_{xx}[\cd]= 0$ and
by \eqref{8.18-eq13}, $\mE\langle
P(T)x_{1}(T),x_{1}(T)\rangle_H\leq 0$. Then, by
the definition of the relaxed transposition
solution of \eqref{op-bsystem3}, \eqref{second
order integral condition} holds and it is
reduced to\vspace{-0.2cm}
\begin{equation*}\label{8.18-eq4}
\begin{array}{ll}\ds
\mE\langle P(0)\nu_1,\nu_1 \rangle_H \!+\!\mE\!
\int_0^T\!\!
\[\big\langle b_{2}[t]^{*}P(t)b_{2}[t]u_1(t),
u_1(t)\big\rangle_{H_{1}}\!\! +\!
2\big\langle\!\big(a_{2}[t]^*
P(t)\! +\! b_{2}[t]^*P(t) b_{1}[t]\big)x_{1}(t),u_1(t)\big\rangle_{H_{1}} \\
\ns\ds\hspace{3.76cm} + \big\langle\big(
\widehat Q^{(0)} +
Q^{(0)}\big)\big(0,a_{2}[t]u_1(t),
b_{2}[t]u_1(t)\big),b_{2}[t]u_1(t)
\big\rangle_{\cL_2}\]dt \leq 0.
\end{array}\vspace{-0.1cm}
\end{equation*}

If $\cE^{(2,j)}_T(x_{1})\neq\emptyset$ for each
$j\in \mathbf{I}$, then one can find $j_0\in
\mathbf{I}$ and a subset $\mathbf{I}^0\subset
\mathbf{I}$ with $j_0\notin \mathbf{I}^0$ such
that\vspace{-0.1cm}
$$
\bigcap_{j\in
\mathbf{I}^0}\cE^{(2,j)}_T(x_{1})\neq\emptyset,
\qq \(\bigcap_{j\in
\mathbf{I}^0}\cE^{(2,j)}_T(x_{1})\)\bigcap
\cE^{(2,j_0)}_T(x_{1})=\emptyset.\vspace{-0.1cm}
$$
By the Hahn-Banach separation theorem, there
exists a nonzero $\xi\in L^2_{\cF_T}(\Om;H)$
such that\vspace{-0.1cm}
$$
\sup_{\a\in
\cE^{(2,j_0)}_T(x_{1})}\mE\langle\xi,\a\rangle_H\leq
\inf_{\b\in \cap_{j\in
\mathbf{I}^0}\cE^{(2,j)}_T(x_{1})}\mE\langle\xi,\b\rangle_H.\vspace{-0.2cm}
$$
By Lemma \ref{lm5}, we can find $\a_0\in {\rm
cl} \cE^{(2,j_0)}_T(x_{1})$ and $\b_0\in
\bigcap_{j\in \mathbf{I}^0} {\rm cl}
\cE^{(2,j)}_T(x_{1})$ such that
\begin{equation}\label{8.18-eq14}
\mE\langle \xi,\a_0\rangle_H = \sup_{\a\in
\cE^{(2,j_0)}_T(x_{1})}\mE\langle\xi,\a\rangle_H\leq
\inf_{\b\in \cap_{j\in
\mathbf{I}^0}\cE^{(2,j)}_T(x_{1})}\mE\langle\xi,\b\rangle_H
= \mE\langle\xi,\b_0\rangle_H.\vspace{-0.2cm}
\end{equation}
It follows from Lemma \ref{lm5} that there
exists $\l_{j_0}>0$ such that
$\xi=\l_{j_0}g_x^{j_0}(\bar x(T))$ and
\vspace{-0.1cm}
\begin{equation}\label{8.18-eq15}
0=\mE\big\langle g_x^{j_0}(\bar
x(T)),\a_0\big\rangle_H +
\frac{1}{2}\mE\big\langle g_{xx}^{j_0}(\bar
x(T))x_{1}(T),x_{1}(T)\big\rangle_H.\vspace{-0.1cm}
\end{equation}
Denote by $\mathbf{I}^1$ the set of all indices
$j\in \mathbf{I}^0$ satisfying\vspace{-0.2cm}
\begin{equation}\label{8.18-eq16}
0=\mE\big\langle g_x^{j}(\bar
x(T)),\a_0\big\rangle_H +
\frac{1}{2}\mE\big\langle g_{xx}^{j}(\bar
x(T))x_{1}(T),x_{1}(T)\big\rangle_H.\vspace{-0.2cm}
\end{equation}
Then, by Lemma \ref{lm5} once more, for each
$j\in \mathbf{I}^1$, there exists $\l_j\geq 0$
such that\vspace{-0.1cm}
\begin{equation}\label{8.18-eq17}
-\xi=-\l_{j_0}g_x^{j_0}(\bar x(T))=\sum_{j\in
\mathbf{I}^1}\l_jg_x^{j}(\bar
x(T)).\vspace{-0.2cm}
\end{equation}
Combing \eqref{8.18-eq14}--\eqref{8.18-eq17}, we
obtain that\vspace{-0.2cm}
$$
0\leq \l_{j_0}\mE\big\langle g_{xx}^{j_0}(\bar
x(T))x_{1}(T),x_{1}(T)\big\rangle_H + \sum_{j\in
\mathbf{I}^1}\mE\big\langle g_{xx}^{j}(\bar
x(T))x_{1}(T),x_{1}(T)\big\rangle_H.\vspace{-0.2cm}
$$
Let $ y(T)=0$ and
$P(T)=-\l_{j_0}g_{xx}^{j_0}(\bar
x(T))-\sum_{j\in \mathbf{I}^1} g_{xx}^{j}(\bar
x(T))$. Then\vspace{-0.2cm}
$$
(y(\cd),Y(\cd))=(0,0), \q \dbH(\cd)=0, \q
\dbH_{xx}[\cd]=0, \q\mE\langle P(\bar
x(T))x_{1}(T),x_{1}(T)\rangle_H\leq
0.\vspace{-0.2cm}
$$
Applying the same argument as before, we obtain
\eqref{second order integral condition} with
$\psi= 0$.

\vspace{0.2cm}

{\bf Step 2}. In this step, we deal with the
case that $\dbI^0(\bar x,x_1)\neq\emptyset$.

From $\mE|g_{x}(\bx(t))|_H\neq 0$ for any $t\in
\cI^0(\bar x)$ and $e(\cdot)$ (recall
\eqref{e(t)} for the definition of $e(\cdot)$)
is bounded on $\dbI^0(\bar x,x_1)$, we get that
$-g_{x}(\bx(\cdot))\in \cQ^{(1)}$ and $-\d
g_{x}(\bx(\cdot))\in \cQ^{(2)}(x_1)$ when $\d$
($>0$) is large enough. Thus, $\cQ^{(1)} \neq
\emptyset$ and $\cQ^{(2)}(x_1) \neq \emptyset$.

Let  $x_2(\cd) \in \cG^{(2)}(x_1,u_1)$ and
$(y(\cd),Y(\cd))$ be the transposition solution
to (\ref{first ajoint equ}). We deduce from
(\ref{dual P1y2}) that\vspace{-0.2cm}
$$
\begin{array}{ll}\ds
\mE \inner{y(T)}{x_2(T)}_H \\
\ns\ds=  \mE \inner{y(0)}{\nu_{2}}_H
+\mE\int_{0}^{T}\inner{x_{2}(t)}{d\psi(t)}_H
+\mE\int_{0}^{T}\inner{\dbH_{u}[t]}{u_2(t)}_{H_1}dt \\
\ns\ds\q+\frac{1}{2}\mE\int_{0}^{T}
\big(\!\inner{\dbH_{xx}[t]x_{1}(t)}{x_{1}(t)}_H
+2\inner{\dbH_{xu}[t]x_{1}(t)}{u_1(t)}_{H_1}+\inner{\dbH_{uu}[t]u_1(t)}{u_1(t)}_{H_1}
\!\big)dt.
\end{array}\vspace{-0.2cm}
$$

If $\cE^{(2,j_1)}(x_{1}) =\emptyset$ for some
$j_1\in \mathbf{I}$,   then
$g^{j_1}_{x}(\bx(T))=0$, $\dbP$-a.s. and $ \mE
\langle
g^{j_1}_{xx}(\bx(T))x_{1}(T),x_{1}(T)\rangle_H
\geq 0$.
Therefore, by setting $ \psi(\cd) = 0$,
$\lambda_{j_1}=1$ and $\lambda_j =0$   for all
$j_1 \neq j \in \mathbf{I}$,
we get $(y(\cd),Y(\cd))=(0,0)$,
$P(T)=-g_{xx}^{j_1}(\bar x(T)) $, $
\mE\inner{P(T)x_{1}(T)}{x_{1}(T)}_H \leq 0 $
and\vspace{-0.2cm}
$$
\mE\int_0^T\inner{\dbH_{xx}[t]x_{1}(t)}{x_{1}(t)}_H
dt=0.\vspace{-0.2cm}
$$
These facts, together with (\ref{dual P2 y1y1}),
imply \eqref{second order integral condition}.

Next, assume that $\cE^{(2,j)}(x_{1}) \neq
\emptyset$ for every $j \in \mathbf{I}$. We
claim that\vspace{-0.1cm}
\begin{equation}\label{11.20-eq5}
\cG^{(2)}(x_1,u_1) \bigcap
\mathcal{Q}^{(2)}(x_1) \bigcap \(\bigcap_{j \in
\mathbf{I}} \cE^{(2,j)}(x_{1})\)=
\emptyset.\vspace{-0.2cm}
\end{equation}
Indeed, if
$$
\cG^{(2)}(x_1,u_1) \bigcap
\mathcal{Q}^{(2)}(x_1) \bigcap \(\bigcap_{j \in
\dbI(\bar x,x_1)} \cE^{(2,j)}(x_{1})\)=
\emptyset,\vspace{-0.2cm}
$$
then \eqref{11.20-eq5} holds. Otherwise, for any
$$
x_2\in \cG^{(2)}(x_1,u_1) \bigcap
\mathcal{Q}^{(2)}(x_1) \bigcap \(\bigcap_{j \in
\dbI(\bar x,x_1)} \cE^{(2,j)}(x_{1})\),
$$
from \eqref{L2} and \eqref{2nd exp of phi}, we
see that $x_2\notin
\cL^{(2)}(x_1)=\cE^{(2,n+1)}(x_{1})$. This also
yields \eqref{11.20-eq5}.

It follows from Lemma \ref{lm9} that there exist
$x^*, x_j^* \in L_{\dbF}^{2}(\Omega; C([0,T];
H))^*$ for all $j \in \mathbf{I}$,  which do not
vanish simultaneously, such that for $\ds \k^* =
-\(x^* + \sum_{j\in \mathbf{I}} x_j^*\)$,
\vspace{-0.2cm}
\begin{equation}\label{new}
\inf_{z\in \cG^{(2)} (x_1,u_1)} \k^*(z) +
\inf_{z\in\mathcal{Q}^{(2)}(x_1)} x^*(z) +
\sum_{j\in \mathbf{I}} \inf_{z\in
\cE^{(2,j)}(x_{1})} {x_j^*}(z) \geq
0.\vspace{-0.2cm}
\end{equation}
If $g^{j}_{x}(\bx(T))=0$  for some $j \in
\mathbf{I}$, then
$\cE^{(2,j)}(x_{1})=L_{\dbF}^{2}(\Omega;
C([0,T]; H))$. This, together with (\ref{new}),
yields $x_j^*=0$.

For each $j\in \mathbf{I}$ with
$g^{j}_{x}(\bx(T)) \neq 0$, put\vspace{-0.2cm}
\begin{equation}\label{Ed2n}
R_j \= \big\{z_T\in  L^{2}_{\cF_{T}}(\O;
H)\,\big|\, \mE\inner{g^{j}_{x}(\bx(T))}{z_T}_H
\leq 0\; \big\}.\vspace{-0.2cm}
\end{equation}
Then $R_j$ is a closed convex cone and $(R_j)^-
= \dbR_+ g^{j}_{x}(\bx(T))$.

Let $\Gamma$ be given by (\ref{gamma C to L2}).
It is easy to show that\vspace{-0.2cm}
$$
\Gamma^{-1}(R_j) + \cE^{(2,j)}(x_{1}) \subset
\cE^{(2,j)}(x_{1}) \; \mbox{ for every }\;j\in
\mathbf{I}\vspace{-0.2cm}
$$
and that $\Gamma^{-1}(R_j)$ is a cone.  Hence,
by (\ref{new}), $-x_j^* \in
\left(\Gamma^{-1}(R_j)\right)^- $. Noting that
$\Gamma$ is surjective, by the well known result
of  convex analysis,
$\left(\Gamma^{-1}(R_j)\right)^- =
\Gamma^*(R_j^-)$ (see for instance
\cite[Corollary 22, p. 144]{Aubin84} applied  to
the closed convex cone $R_j$ and the set-valued
map $\Gamma^{-1}$  whose graph is a closed
subspace of $ L^{2}_{\cF_{T}}(\Omega;H) \times
L_{\dbF}^{2}(\Omega; C([0,T]; H))$).
Therefore,\vspace{-0.1cm}
$$
-x_j^* = \Gamma^* (\lambda_j g^{j}_{x}(\bx(T)) )
\mbox{ for some }\lambda_j \geq
0.\vspace{-0.1cm}
$$
If $x_j^*=0$, then we put $\lambda_j=0$.  By
normalizing, we may assume that $\lambda_0
\in\{0,1\}$.

Since the map $\Gamma$ is surjective, we have
that\vspace{-0.1cm}
$$
\sup_{z \in \cE^{(2,j)}(x_{1})}
\left(-x_j^*\right)(z) = \sup_{z \in
\cE^{(2,j)}(x_{1})} \mE\big\langle \lambda_j
g^{j}_{x}(\bx(T)) , \Gamma (z) \big\rangle_H =
\sup_{z_T \in \cE^{(2,j)}_T(x_{1})}
\mE\big\langle \lambda_j g^{j}_{x}(\bx(T)), z(T)
\big\rangle_H.\vspace{-0.1cm}
$$
By the definition of $\cE^{(2,j)}(x_{1})$, for
any $j \in \mathbf{I}$ with $g^{j}_{x}(\bx(T))
\neq 0$,\vspace{-0.2cm}
$$
\sup_{z_T \in \cE_T^{(2,j)}(x_{1})} \mE\langle
\lambda_j g^{j}_{x}(\bx(T)) , z_T \rangle_H =
-\frac {\lambda_j}{2}\mE
\inner{g^{j}_{xx}(\bx(T))x_{1}(T)}{x_{1}(T)}_H.\vspace{-0.2cm}
$$

From (\ref{new}) (by setting $d\psi =-x^*$), we
deduce that\vspace{-0.2cm}
\begin{equation}\label{new4}
\3n \sup_{x_2 \in \cG^{(2)}(x_1,v)}\!
(-\k^*)(x_2) +\! \sup_{\alpha \in
\mathcal{Q}^{(2)}(x_1) } \mE\!\int_0^T \langle
\alpha (t), d\psi(t) \rangle_H\! -\! \frac12
\mE \sum_{j \in \mathbf{I}} \inner{\lambda_j
g^{j}_{xx}(\bx(T))x_{1}(T)}{x_{1}(T)}_H \leq 0.
\end{equation}
Recalling Remark \ref{11.20-rmk1} for the
inclusions $ \mathcal{Q}^{(1)} +
\mathcal{Q}^{(2)}(x_1) \subset
\mathcal{Q}^{(2)}(x_1) $ and $ \cG^{(1)} +
\cG^{(2)}(x_1,u_1) \subset \cG^{(2)}(x_1,u_1)$,
we get from \eqref{new4} that $ d\psi \in
(\mathcal{Q}^{(1)})^-$ and $-\k^* \in
(\cG^{(1)})^-$.

Put $\ds y(T)= -\sum_{j\in
\mathbf{I}}\lambda_{j} g_{x}^{j}(\bx(T))$ and
let $(y(\cd),Y(\cd))$ be the solution  to
\eqref{first ajoint equ} with ${\cal I}(\bar x)$
replaced by $ {\dbI}(\bar x,x_1)$. Let
$P(T)=-\lambda_0h_{xx}(\bx(T))$ and
$(P(\cd),Q^{(\cd)},\widehat Q^{(\cd)})$ be the
relaxed solution of (\ref{op-bsystem3}). By
(\ref{new4}), for every $x_2(\cd) \in
\cG^{(2)}(\bar x,x_1)$,\vspace{-0.2cm}
\begin{equation*}\label{new1}
\begin{array}{ll}\ds
-\mE\int_{0}^{T}\inner{x_2 (t)}{d \psi(t)}_H -
\sum_{j \in \mathbf{I}}\mE\langle \lambda_j
g^{j}_{x}(\bx(T)), x_{2}(T)  \rangle_H
\\
\ns\ds  +\frac12 \mE\langle P(T) x_1(T), x_1(T)
\rangle_H  + \sup_{\alpha \in
\mathcal{Q}^{(2)}(x_1) } \int_0^T \langle \alpha
(t), d\psi(t) \rangle_H  \leq 0. \vspace{-0.2cm}
\end{array}
\end{equation*}
From the above inequality, using  (\ref{dual
P1y2})  and (\ref{dual P2 y1y1}), we complete
the proof.
\end{proof}

\appendix


\section{Proof of Lemma \ref{estimate one of varie qu}}


We first recall the following result. Its proof
can be found in  \cite[Chapter 7]{Prato}.
\begin{lemma}\label{estimatelinearsde}
Assume that \textbf{(AS1)} holds. Then,  for any
$\nu_{0}\in H$, $p\geq 1$ and $u(\cd)\in
L_{\dbF}^{p}(\Omega; L^{2}(0, T; $ $H_1))$, the
equation (\ref{controlsys}) admits a unique
solution $x(\cd) \in C_{\dbF}([0,
T];L^{p}(\Omega; H))$, and for any $t\in
[0,T]$,\vspace{-0.2cm}
\begin{equation}\label{estimateofx}
\begin{array}{ll}\ds
\sup_{s\in[0,t]}\mE\Big(|x(s)|_H^{p}\Big) \le
C\mE \Big[|\nu_{0}|_H^{p}
+\Big(\int_{0}^{t}|a(s,0,u(s))|_Hds\Big)^{p}
+\Big(\int_{0}^{t}|b(s,0,u(s))|_{\cL_2}^{2}ds\Big)^{\frac{p}{2}}\Big].
\end{array}\vspace{-0.2cm}
\end{equation}
Moreover, if  $\tilde{x}$ is the  solution of
(\ref{controlsys}) corresponding to $(\tilde
\nu_{0}, \tilde{u})\in H\times
L_{\dbF}^{p}(\Omega; L^{2}(0, T; H_1))$, then,
for any $t\in [0,T]$,\vspace{-0.32cm}
\begin{equation}\label{estimateof delta x}
\sup_{s\in[0,t]}\mE\Big(|x(s)-\tilde{x}(s)|_H^{p}\Big)
\le C\mE~\Big[|\nu_{0}-\tilde
\nu_{0}|_H^{p}+\Big(\int_{0}^{t}|u(s)
-\tilde{u}(s)|_{H_1}^{2}ds
\Big)^{\frac{p}{2}}\Big].\vspace{-0.2cm}
\end{equation}
\end{lemma}

\begin{proof}[Proof of Lemma \ref{estimate one of varie
qu}] From  (\ref{first vari equ}) and  Lemma
\ref{estimatelinearsde} we deduce
that\vspace{-0.2cm}
$$
\begin{array}{ll}\ds
\mE\Big(|x_{1}(t)|_H^{p}\Big)\3n&\ds  \le \!C
\mE\Big[\big|\nu_1\big|_H^p\!+\!
\Big(\int_{0}^{t}|a_{2}[s]u_1(s)|_Hds\Big)^{p}
+ \(\int_{0}^{t}| b_{2}[s]u_1(s)|_{\cL_2}^{2}ds\)^{\frac{p}{2}} \Big]\\
\ns&\ds\le C\mE\Big[|\nu_1|^{p} +
\(\int_{0}^{t}|u_1(t)|_{H_1}^2dt\)^{\frac{p}{2}}\Big].
\end{array}\vspace{-0.2cm}
$$
This  implies that\vspace{-0.2cm}
\begin{equation*}
\sup_{t\in[0,T]}\mE\Big(|x_{1}(t)|_H^{p}\Big)
\leq C\mE\Big[|\nu_1|^{p}+
\(\int_{0}^{T}|u_1(t)|_{H_1}^2dt
\)^{\frac{p}{2}}\Big],\vspace{-0.2cm}
\end{equation*}
which yields \eqref{8.18-eq41}.

\medskip

Since\vspace{-0.32cm}
$$
\lim_{\e\to 0^+}\nu_1^{\e}=\nu_1 \mbox{ in }H,\q
\lim_{\e\to 0^+}u_1^{\e}(\cdot)=u_1(\cdot)
\mbox{ in } L^p_{\dbF}(\Om;L^2(0,T;H_1)),
$$
it follows from (\ref{estimateof delta x})
that\vspace{-0.2cm}
\begin{equation*}
\sup_{t\in[0,T]}\mE\Big(|\delta
x^{\e}(t)|_H^{p}\Big) \le
C\mE\Big(\e^{p}|\nu^{\e}_{0}|_{H}^{p}+
\(\int_{0}^{T}|\e
u_1^{\e}(t)|_{H_1}^{2}ds\)^{\frac{p}{2}}\Big)
=O(\e^p).
\end{equation*}
This implies \eqref{8.18-eq42}.

\medskip

Let\vspace{-0.2cm}
$$
\begin{cases}\ds
\tilde{a}_{1}^{\e}(t)\=\int_{0}^{1}a_{x}(t,\bx(t)
+\theta\delta x^{\e}(t),\bu(t)+\theta\e
u_1^{\e}(t))d\theta,\\
\ns\ds
\tilde{a}_{2}^{\e}(t)\=\int_{0}^{1}a_{u}(t,\bx(t)
+\theta\delta x^{\e}(t),\bu(t)+\theta\e
u_1^{\e}(t))d\theta,\\
\ns\ds
\tilde{b}_{1}^{\e}(t)\=\int_{0}^{1}b_{x}(t,\bx(t)
+\theta\delta x^{\e}(t),\bu(t)+\theta\e
u_1^{\e}(t))d\theta,\\
\ns\ds
\tilde{b}_{2}^{\e}(t)\=\int_{0}^{1}b_{u}(t,\bx(t)
+\theta\delta x^{\e}(t),\bu(t)+\theta\e
u_1^{\e}(t))d\theta.
\end{cases}\vspace{-0.2cm}
$$
Then, $\delta x^{\e} (\cdot)$ is the solution of
the following SEE:\vspace{-0.1cm}
\begin{equation*}
\left\{\2n
\begin{array}{ll}\ds
d\delta x^{\e}(t)\!=\! \big(A \delta
x^{\e}(t)\!+\!\tilde{a}_{1}^{\e}(t)\delta
x^{\e}(t) \!+\!\e\tilde{a}_{2}^{\e}(t)
u_1^{\e}(t)\big)dt
+\big(\tilde{b}_{1}^{\e}(t)\delta x^{\e}(t)
\!+\!\e\tilde{b}_{2}^{\e}(t)
u_1^{\e}(t)\big)dW(t) \q \mbox{ in }(0,T],\\
\ns\ds\delta x^{\e}(0)=\e \nu_1^{\e},
\end{array}\vspace{-0.2cm}
\right.
\end{equation*}
and $r_{1}^{\e}(\cdot) $ solves\vspace{-0.2cm}
\begin{equation}\label{11.20-eq11}
\left\{
\begin{array}{ll}
\ds dr_{1}^{\e}(t)=
\Big[Ar_{1}^{\e}(t)+\tilde{a}_{1}^{\e}(t)r_{1}^{\e}(t)
+\big(\tilde{a}_{1}^{\e}(t)-a_{1}[t]\big)x_{1}(t)
+\tilde{a}_{2}^{\e}(t)\big(u_1^{\e}(t)-u_1 (t)\big)\\
\ns\ds\qquad\qquad
+\big(\tilde{a}_{2}^{\e}(t)-a_{2}[t]\big)u_1(t)\Big]dt
+\Big[\tilde{b}_{1}^{\e}(t)r_{1}^{\e}(t)
+\big(\tilde{b}_{1}^{\e}(t)-b_{1}[t]\big)x_{1}(t)\\
\ns\ds\qquad\qquad
+\tilde{b}_{2}^{\e}(t)\big(u_1^{\e}(t)-u_1
(t)\big)
+\big(\tilde{b}_{2}^{\e}(t)-b_{2}[t]\big)u_1(t)\Big]dW(t) \qq\qq  \mbox{ in }(0,T],\\
\ns\ds r_{1}^{\e}(0)=\nu^{\e}_{1}-\nu_1.
\end{array}
\right.
\end{equation}

For any sequence $\{\e_{j}\}_{j=1}^\infty$ of
positive numbers satisfying $\ds\lim_{j\to
\infty}\e_{j}=0$, we can find a subsequence
$\{j_k\}_{k=1}^\infty\subset\mn$ such
that\vspace{-0.2cm}
$$
\begin{cases}\ds
\lim_{k\to \infty}\sup_{t\in[0,T]}|\delta
x^{\e_{j_k}}(t)|_H \to
0, &\dbP\mbox{-a.s.,  } \\
\ns\ds \lim_{k\to \infty}\e_{j_k}
|u_1^{\e_{j_k}}(t)|_{H_1}=0, &\dbP\mbox{-a.s.
for a.e. } t \in [0,T].
\end{cases}\vspace{-0.2cm}
$$
Hence,\vspace{-0.2cm}
$$
\big|(\tilde{a}_{1}^{\e_j}(\cdot)-a_{1}(\cdot))x_{1}(\cdot)
\big|_H \to 0\; \mbox{ in  measure, as
}j\to\infty.\vspace{-0.2cm}
$$

From {\bf (AS2)}, we see that\vspace{-0.2cm}
$$
\lim_{k\to
\infty}\big|(\tilde{a}_{1}^{\e_{j_k}}(t)-a_{1}[t])x_{1}(t)
\big|_H = 0, \qq  \dbP\mbox{-a.s. for a.e. }t
\in [0,T].\vspace{-0.2cm}
$$
Then, it follows from  Lebesgue's dominated
convergence theorem that\vspace{-0.2cm}
\begin{equation}\label{11.20-eq12}
\lim_{j\to\infty}\mE
\int_{0}^{T}\big|\big(\tilde{a}_{1}^{\e_{j}}(t)
-a_{1}[t]\big)x_{1}(t)\big|_H^pdt= 0.
\end{equation}
A similar argument implies that\vspace{-0.2cm}
\begin{equation}\label{11.20-eq13}
\begin{array}{ll}\ds
\lim_{j\to\infty}
\mE\Big[\(\int_{0}^{T}\big|\big(\tilde{a}_{2}^{\e_{j}}(t)
-a_{2}[t]\big)u_1(t)\big|_{H}^2dt\)^{\frac{p}{2}}
+
\(\int_{0}^{T}\big|\big(\tilde{b}_{1}^{\e_{j}}(t)
-b_{1}[t]\big)x_{1}(t)\big|_{\cL_2}^2dt\)^{\frac{p}{2}}  \\
\ns\ds \qq\q+
\(\int_{0}^{T}\big|\big(\tilde{b}_{2}^{\e_{j}}(t)
-b_{2}[t]\big)u_1(t)\big|_{\cL_2}^2dt\)^{\frac{p}{2}}\Big]
= 0.
\end{array}\vspace{-0.2cm}
\end{equation}
On the other hand,\vspace{-0.2cm}
$$
\begin{array}{ll}\ds
\lim_{j\to\infty}
\mE\Big[\(\int_{0}^{T}|\tilde{a}_{2}^{\e_j}(t)
\big(u_1^{\e_j}(t)-u_1(t)\big)|_{H_1}^2dt\)^{\frac{p}{2}}
+ \(\int_{0}^{T}|\tilde{b}_{2}^{\e_j}(t)
\big(u_1^{\e_j}(t)-u_1(t)\big)|_{\cL_2}^2dt\)^{\frac{p}{2}} \]\\
\ns\ds\leq C\lim_{j\to\infty}\mE\(\int_{0}^{T}
|u_1^{\e_j}(t)-u_1(t)|^2_{H_1}dt\)^{\frac{p}{2}}
\to 0.
\end{array}\vspace{-0.2cm}
$$
Therefore, by Lemma \ref{estimatelinearsde} and
\eqref{11.20-eq11}--\eqref{11.20-eq13}, we
obtain that
$$
\begin{array}{ll}\ds
\lim_{j\to\infty}\sup_{t\in[0,T]}\mE\Big(|r_{1}^{\e_{j}}(t)|_H^{p}\Big)\\
\ns\ds\le\!
C\!\!\lim_{j\to\infty}\!\mE\Big[|\nu^{\e_{j_k}}_{1}\!\!-\!\nu_1|_H^{p}\!
+\!\Big(\!\int_{0}^{T}\!\!\big|\big(\tilde{a}_{1}^{\e_{j}}\!(t)\!-\!a_{1}[t]\big)x_{1}\!(t)
\!+\!\tilde{a}_{2}^{\e_{j}}\!(t)\big(u_1^{\e_{j}}\!(t)\!-\!u_1(t)\big)
\!+\!
\big(\tilde{a}_{2}^{\e_{j}}\!(t)\!-\!a_{2}[t]\big)u_1(t)\big|_Hdt\Big)^{p}\\
\ns\ds\qq\qq
+\Big(\int_{0}^{T}\!\big|\big(\tilde{b}_{1}^{\e_{j}}(t)
-b_{1}[t]\big)x_{1}(t) +
\tilde{b}_{2}^{\e_{j}}(t)\big(u_1^{\e_{j_k}}(t)\!-\!u_1(t)\big)
+\big(\tilde{b}_{2}^{\e_{j}}(t)-b_{2}[t]\big)u_1(t)
\big|_{\cL_2}^2dt\Big)^{\frac{p}{2}} \Big]= 0.
\end{array}
$$
Since the sequence $\{\e_{j}\}_{j=1}^\infty$ is
arbitrary, the proof is complete.
\end{proof}

\section{Proof of Lemma \ref{estimate two of varie qu}}

\begin{proof}
By Lemma \ref{estimate one of varie qu} (applied
with $p=4$), we obtain\vspace{-0.2cm}
\begin{equation}\label{8.18-eq30}
\sup_{t\in[0,T]}\mE\Big(|x_{1}(t)|_H^{4}\Big)
\le C\mE\Big[ |\nu_1|_H^{4}
+\Big(\int_{0}^{T}|u_1(t)|_{H_1}^2dt\Big)^{2}\Big].
\vspace{-0.2cm}
\end{equation}
By  (\ref{second order vari equ}),
\eqref{8.18-eq30} and H\"older's inequality, we
have that\vspace{-0.2cm}
$$
\begin{array}{ll}
\ds\sup_{t\in[0,T]}\mE\Big(|x_{2}(t)|_H^{2}\Big)\\
\ns\ds\le \! C\mE \Big[|\nu_{2}|_H^{2}\!+\!
\Big(\!\int_{0}^{T}\!\!\big|2a_{2}[t]u_2(t)
\!+\!a_{11}[t]\big(x_{1}\!(t),x_{1}\!(t)\big)\!
+\!2 a_{12}[t]\big(x_{1}\!(t),u_1\!(t)\big)\!+\!
a_{22}[t]\big(u_1\!(t),u_1\!(t)\big)\big|_Hdt\Big)^{2}
\\
\ns\ds\qquad\;\; + \int_{0}^{T}|2b_{2}[t]u_2(t)
+ b_{11}[t]\big(x_{1}(t),x_{1}(t)\big)  +2
b_{12}[t]\big(x_{1}(t),u_1(t)\big)
+ b_{22}[t]\big(u_1(t),u_1(t)\big)|_{\cL_2}^{2}dt \Big]\\
\ns\ds\le C \mE\Big(|\nu_{2}|_H^{2}\!+\!
\int_{0}^{T}\!|u_2(t)|_{H_1}^2dt
\!+\! \int_{0}^{T}\!|u_1(t)|_{H_1}^4dt \Big)
\!+\!\sup_{t\in [0,T]}\(\mE|x_{1}(t)|_H^{4} +
\mE|x_{1}(t)|_H^{2}
\mE\int_{0}^{T}|u_1(t)|_{H_1}^2dt\)
\\
\ns\ds\leq
C\mE\Big(|\nu_{2}|_H^{2}+|\nu_1|_H^{4} +
\int_{0}^{T}|u_2(t)|_{H_1}^2dt +
\int_{0}^{T}|u_1(t)|_{H_1}^4dt \Big).
\end{array}\vspace{-0.2cm}
$$

Let\vspace{-0.2cm}
$$
\begin{cases}\ds
\tilde{a}_{11}^{\e}(t)\=\int_{0}^{1}(1-\theta)a_{xx}(t,\bx(t)
+\theta\delta x^{\e}(t),\bu(t)+\theta\delta
u^{\e}(t))d\theta,\\
\ns\ds\tilde{a}_{12}^{\e}(t)\=\int_{0}^{1}(1-\theta)a_{xu}(t,\bx(t)
+\theta\delta x^{\e}(t),\bu(t)+\theta\delta
u^{\e}(t))d\theta,\\
\ns\ds\tilde{a}_{22}^{\e}(t)\=\int_{0}^{1}(1-\theta)a_{uu}(t,\bx(t)
+\theta\delta x^{\e}(t),\bu(t)+\theta\delta
u^{\e}(t))d\theta,\\
\ns\ds\tilde{b}_{11}^{\e}(t)\=\int_{0}^{1}(1-\theta)b_{xx}(t,\bx(t)
+\theta\delta x^{\e}(t),\bu(t)+\theta\delta
u^{\e}(t))d\theta,\\
\ns\ds\tilde{b}_{12}^{\e}(t)\=\int_{0}^{1}(1-\theta)b_{xu}(t,\bx(t)
+\theta\delta x^{\e}(t),\bu(t)+\theta\delta
u^{\e}(t))d\theta,\\
\ns\ds\tilde{b}_{22}^{\e}(t)\=\int_{0}^{1}(1-\theta)b_{uu}(t,\bx(t)
+\theta\delta x^{\e}(t),\bu(t)+\theta\delta
u^{\e}(t))d\theta.
\end{cases}
$$
Then, $\delta x^{\e}$ solves\vspace{-0.2cm}
\begin{equation}\label{8.18-eq31}
\left\{
\begin{array}{l}
d\delta x^{\e}(t)= \Big[A\delta x^{\e}(t)
+a_1[t]\delta x^{\e}(t) +a_{2}[t]\delta
u^{\e}(t)
+ \tilde{a}_{11}^{\e}(t)\big(\delta x^{\e}(t),\delta x^{\e}(t)\big)\\[+0.4em]
\qq \qq \q \ +2\delta
\tilde{a}_{12}^{\e}(t)\big(x^{\e}(t),\delta
u^{\e}(t)\big)
+\tilde{a}_{22}^{\e}(t)\big(\delta
u^{\e}(t),\delta u^{\e}(t)\big) \Big]dt\\
\qq \qq \q  +\Big[b_{1}[t]\delta x^{\e}(t)
+b_{2}[t]\delta u^{\e}(t)
+\tilde{b}_{11}^{\e}(t)\big( \delta
x^{\e}(t),\delta
x^{\e}(t)\big)\\[+0.4em]
\qq \qq \q +2\tilde{b}_{12}^{\e}(t)\big(\delta
x^{\e}(t), \delta u^{\e}(t)\big)
+\tilde{b}_{22}^{\e}(t)\big( \delta
u^{\e}(t),\delta
u^{\e}(t)\big)\Big]dW(t) \qq \mbox{ in }(0,T],\\[+0.4em]
\delta x^{\e}(0)=\e \nu_1+\e^2\nu_{2}^{\e}.
\end{array}\vspace{-0.2cm}
\right.
\end{equation}
Consequently,   $r_{2}^{\e}$
solves\vspace{-0.2cm}
\begin{eqnarray}\label{8.18-eq32}
\3n\left\{\3n
\begin{array}{ll}\ds
dr_{2}^{\e}(t)\!= \!\Big\{Ar_{2}^{\e}(t) +
a_{1}[t]r_{2}^{\e}(t) + a_{2}[t]\big(u_2^{\e}(t)
- u_2(t)\big)
 + \[\tilde{a}_{11}^{\e}(t)\Big(\frac{\delta
x^{\e}\!(t)}{\e}, \frac{\delta
x^{\e}\!(t)}{\e}\Big) \\
\ns\ds\qquad\qquad- \frac{1}{2}
a_{11}[t]\big(x_{1}\!(t),x_{1}\!(t)\big)\Big]
+\Big[2\tilde{a}_{12}^{\e}(t)\big(\frac{\delta
x^{\e}(t)}{\e},\frac{\delta
u^{\e}(t)}{\e}\big)-a_{12}[t]\big(x_{1}(t),u_1(t)\big)\Big]\\
\ns\ds\qquad\qquad +\Big[
\tilde{a}_{22}^{\e}(t)\Big(\frac{\delta
u^{\e}(t)}{\e},\frac{\delta
u^{\e}(t)}{\e}\Big)-\frac{1}{2}
a_{22}[t]\big(u_1(t),u_1(t)\big)\Big]\Big\}dt\\
\ns\ds\qquad\qquad +\Big\{b_{1}[t]r_{2}^{\e}(t)
+b_{2}[t]\big(h_{\e}(t)\!-\!h(t)\big) +\big[
\tilde{b}_{11}^{\e}(t)\big(\frac{\delta
x^{\e}(t)}{\e},\frac{\delta x^{\e}(t)}{\e}\big)
-\frac{1}{2}
b_{11}[t]\big(x_{1}(t),x_{1}(t)\big)\big]\\
\ns\ds\qquad\qquad +\big[2
\tilde{b}_{12}^{\e}(t)\big(\frac{\delta
x^{\e}(t)}{\e},\frac{\delta u^{\e}(t)}{\e}\big)-
b_{12}[t]\big(x_{1}(t),u_1(t)\big)\big]\\
\ns\ds\qquad\qquad +\big[
\tilde{b}_{22}^{\e}(t)\big(\frac{\delta
u^{\e}(t)}{\e},\frac{\delta
u^{\e}(t)}{\e}\big)-\frac{1}{2}
b_{22}[t]\big(u_1(t),u_1(t)\big)\big]\Big\}dW(t)
\qq \mbox{ in }(0,T], \\
\ns\ds r_{2}^{\e}(0)=\nu_{2}^{\e}-\nu_{2}.
\end{array}\vspace{-0.2cm}
\right.
\end{eqnarray}
Since $u_2^{\e}(\cdot)$ converges to
$u_2(\cdot)$ in $L^{4}_{\dbF}(0,T;H_1)$, we
have\vspace{-0.2cm}
\begin{equation}\label{8.18-eq33}
\lim_{\e\to
0^+}\mE\Big(\int_{0}^{T}\Big|a_{2}[t]\big(u_2^{\e}(t)-u_2(t)\big)
\Big|_Hdt\Big)^{2} + \lim_{\e\to
0^+}\mE\Big(\int_{0}^{T}\Big|b_{2}[t]\big(u_2^{\e}(t)-u_2(t)\big)
\Big|_{\cL_2}^2dt\Big)=0.\vspace{-0.2cm}
\end{equation}
By H\"older's inequality,\vspace{-0.2cm}
\begin{equation}\label{8.18-eq34}
\begin{array}{ll}\ds
\mE\Big(\int_{0}^{T}\Big|
\tilde{a}_{11}^{\e}(t)\(\frac{\delta
x^{\e}(t)}{\e},\frac{\delta x^{\e}(t)}{\e}\Big)
-\frac{1}{2}
a_{11}[t]\big(x_{1}(t),x_{1}(t)\big)\Big|_Hdt\Big)^{2}
\\
\ns\ds\le  C \mE \Big(\int_{0}^{T}\Big|
\tilde{a}_{11}^{\e}(t)\(\frac{\delta
x^{\e}(t)}{\e},\frac{\delta x^{\e}(t)}{\e}\)
-\frac{1}{2}
a_{11}[t]\big(x_{1}(t),x_{1}(t)\big)\Big|_H^2dt\Big)
\\
\ns\ds\le C\mE\Big[\int_{0}^{T}\Big|
\big(\tilde{a}_{11}^{\e}(t)
-\frac{1}{2}a_{11}[t]\big)\(\frac{\delta x^{\e}(t)}{\e},\frac{\delta x^{\e}(t)}{\e}\)\Big|_H^2dt\Big]  \\
\ns\ds\q +C\mE\Big[\sup_{t \in
[0,T]}\Big|\frac{\delta x^{\e}(t)}{\e}
-x_{1}(t)\Big|_H^2 \Big( \sup_{t \in
[0,T]}\Big|\frac{\delta x^{\e}(t)}{\e}\Big|_H^2
+
\sup_{t \in [0,T]}|x_{1}(t)|_H^2\Big)\Big]  \\
\ns\ds\leq C \Big[\mE \Big(\sup_{t \in [0,T]}
\Big|\frac{\delta x^{\e}(t)}{\e}\Big|_H^{4}
\Big) \Big]^{1/2} \Big[\mE
\Big(\int_{0}^{T}\Big|\tilde{a}_{11}^{\e}(t)
-\frac{1}{2}a_{11}[t]\Big|_{\cL(H\times H;H)}^4 \Big|\frac{\delta x^{\e}(t)}{\e}\Big|_H^{4}dt\Big) \Big]^{1/2} \\
\ns\ds\q+ C\Big[\mE\Big(\sup_{t \in
[0,T]}\Big|\frac{\delta x^{\e}(t)}{\e}
-x_{1}(t)\Big|_H^{4}\Big)\Big]^\frac{1}{2} \Big[
\mE\Big(\sup_{t \in [0,T]}\Big|\frac{\delta
x^{\e}(t)}{\e}\Big|_H^{4} + \sup_{t \in
[0,T]}|x_{1}(t)|_H^{4}\Big)\Big]^\frac{1}{2}.
\end{array}\vspace{-0.2cm}
\end{equation}
Since\vspace{-0.2cm}
$$
\lim_{\e\to 0^+}\nu_{2}^{\e}=\nu_{2} \mbox{ in
}H,\qq \lim_{\e\to
0^+}u_2^{\e}(\cd)=u_2(\cd)\mbox{ in
}L^{4}_{\dbF}(0,T;H_1),\vspace{-0.2cm}
$$
by Lemma \ref{estimatelinearsde},\vspace{-0.2cm}
$$
\sup_{t\in[0,T]}\mE\Big(|\delta
x^{\e}(t)|_H^{4}\Big) \le C\mE\Big[|\e
\nu_1+\e^2\nu_{2}^{\e}|_H^{4}+
\Big(\int_{0}^{T}|\e u_1(t)+\e^2
u_2^{\e}(t)|_H^{2}dt\Big)^{2}\Big]
=O(\e^{4}).\vspace{-0.2cm}
$$
%
As the proof of (\ref{r1 to 0}) in Lemma
\ref{estimate one of varie qu}, we obtain
that\vspace{-0.2cm}
$$
\lim_{\e\to
0^+}\sup_{t\in[0,T]}\mE\Big|\frac{\delta
x^{\e}(t)}{\e}-x_{1}(t)\Big|_H^{4}=
0.\vspace{-0.2cm}
$$
For any sequence $\{\e_{j}\}_{j=1}^\infty$ of
positive numbers converging to $0$ as $j\to
\infty$, one can show that\vspace{-0.1cm}
\begin{equation}\label{12.20-eq14}
a_{xx}(\cdot,\bx(\cdot) +\theta\delta
x^{\e_j}(\cdot),\bu(\cdot)+\theta\delta
u^{\e_j}(\cdot)) -a_{11}[\cdot] \to 0,\;\;\;
\mbox{ in measure, as
}j\to\infty.\vspace{-0.2cm}
\end{equation}
Since\vspace{-0.32cm}
$$
\tilde{a}_{11}^{\e_j}(t) -\frac{1}{2}a_{11}[t] =
\int_{0}^{1}(1-\theta)\big(a_{xx}(t,\bx(t)
+\theta\delta x^{\e_j}(t),\bu(t)+\theta\delta
u^{\e_j}(t))
-a_{11}[t]\big)d\theta,\vspace{-0.2cm}
$$
it follows from (\ref{8.18-eq34}),
\eqref{12.20-eq14} and the Lebesgue dominated
convergence theorem that\vspace{-0.2cm}
\begin{equation}\label{8.18-eq35}
\lim_{j\to \infty}\mE\Big(\int_{0}^{T}\Big|
\tilde{a}_{11}^{\e_{j}}(t)\Big(\frac{\delta
x^{\e_{j}}(t)}{\e_{j}},\frac{\delta
x^{\e_{j}}(t)}{\e_{j}}\Big) -\frac{1}{2}
a_{11}[t]\big(x_{1}(t),x_{1}(t)\big)\Big|_Hdt\Big)^{2}
= 0.\vspace{-0.2cm}
\end{equation}
Since,\vspace{-0.2cm}
$$
\begin{array}{ll}\ds
\mE\Big(\int_{0}^{T}\Big|2
\tilde{a}_{12}^{\e_{j}}(t)\Big(\frac{\delta
x^{\e_{j}}(t)}{\e_{j}},\frac{\delta
u^{\e_{j}}(t)}{\e_{j}}\Big)
- a_{12}[t]\big(x_{1}(t),u_1(t)\big)\Big|_Hdt\Big)^{2} \\
\ns\ds\leq  C\mE\Big(\int_{0}^{T}\Big|2
\tilde{a}_{12}^{\e_{j}}(t)\(\frac{\delta
x^{\e_{j}}(t)}{\e_{j}},\frac{\delta
u^{\e_{j}}(t)}{\e_{j}}\)
- a_{12}[t]\big(x_{1}(t),u_1(t)\big)\Big|_H^2dt\Big) \\
\ns\ds\le C \sup_{t \in
[0,T]}\Big(\mE\Big|\frac{\delta
x^{\e_{j}}(t)}{\e_{j}} \Big|_H^{4}
\Big)^{\frac12}  \Big( \mE\int_{0}^{T}
\Big|\tilde{a}_{12}^{\e_{j}}(t)
-\frac{1}{2}a_{12}[t]\Big|_{\cL(H \times H_1;H_1)}^4\Big|\frac{\delta u^{\e_{j}}(t)}{\e_{j}}\Big|_{H_1}^4dt\Big)^{\frac12} \\
\ns\ds\q + C \sup_{t \in [0,T]} \Big(
\mE\Big|\frac{\delta x^{\e_{j}}(t)}{\e_{j}}
-x_{1}(t)\Big|_H^{4}\Big)^{\frac12}
\Big(\mE\int_{0}^{T}\Big|\frac{\delta u^{\e_{j}}(t)}{\e_{j}}\Big|_{H_1}^4 dt \Big)^{\frac12} \\
\ns\ds\q +   C  \sup_{t \in
[0,T]}\Big(\mE|x_{1}(t)|_H^{4} \Big)^{\frac12}
\Big( \mE\int_{0}^{T}\Big|\frac{\delta
u^{\e_{j}}(t)}{\e_{j}} -u_1(t)\Big|_{H_1}^4 dt
\Big)^{\frac12}.
\end{array}\vspace{-0.2cm}
$$
Similar to the proof of \eqref{8.18-eq35}, we
have that\vspace{-0.2cm}
\begin{equation}\label{8.18-eq36}
\lim_{j\to \infty}\mE\Big(\int_{0}^{T}\Big|2
\tilde{a}_{12}^{\e_{j}}(t)\(\frac{\delta
x^{\e_{j}}(t)}{\e_{j}},\frac{\delta
u^{\e_{j}}(t)}{\e_{j}}\) -
a_{12}[t]\(\frac{\delta
x^{\e_{j}}(t)}{\e_{j}},u_1(t)\)\Big|_Hdt\Big)^{2}
= 0.\vspace{-0.2cm}
\end{equation}
Similarly,\vspace{-0.2cm}
\begin{eqnarray}\label{8.18-eq37}
& &\q\lim_{j\to\infty}\mE\Big(\int_{0}^{T}\Big|
\tilde{a}_{22}^{\e_{j}}(t)\(\frac{\delta
u^{\e_{j}}(t)}{\e_{j}},\frac{\delta
u^{\e_{j}}(t)}{\e_{j}}\)
-\frac{1}{2} a_{22}[t]\big(u_1(t),u_1(t)\big)\Big|_Hdt\Big)^{2}\nonumber\\
&&\le
C\lim_{j\to\infty}\mE\Big(\int_{0}^{T}\Big|\(\tilde{a}_{22}^{\e_{j}}(t)
-\frac{1}{2}a_{22}[t]\)\(\frac{\delta
u^{\e_{j}}(t)}{\e_{j}},\frac{\delta u^{\e_{j}}(t)}{\e_{j}}\)\Big|_H^2dt\Big) \nonumber\\
&
&\q+C\lim_{j\to\infty}\mE\Big[\int_{0}^{T}\Big|\frac{\delta
u^{\e_{j}}(t)}{\e_{j}}
-u_1(t)\Big|^{2}_{H_1}\Big(\Big|\frac{\delta
u^{\e_{j}}(t)}{\e_{j}}\Big|_{H_1}
^2+|u_1(t)|_{H_1}^2\Big)dt\Big] \\
&&\le C\lim_{j\to\infty}\mE
\int_{0}^{T}\Big|\frac{\delta
u^{\e_{j}}(t)}{\e_{j}}\Big|_{H_1}^{4}
\Big|\tilde{a}_{22}^{\e_{j}}(t)
-\frac{1}{2}a_{22}[t]\Big|_{\cL(H_1\times H_1;H_1)}^2dt \nonumber\\
& &\q+C\lim_{j\to\infty}\mE
\int_{0}^{T}\big|\e_{j}
u_2^{\e_{j}}(t)\big|_{H_1}^{2}\cdot\Big(\big|
u_1(t)+ \e_{j} u_2^{\e_{j}}(t)\big|_{H_1}^{2}+|
u_1(t)|_{H_1}^2\Big)dt
=0.\nonumber\vspace{-0.2cm}
\end{eqnarray}
Similar to the above argument, we
obtain\vspace{-0.2cm}
\begin{equation}\label{8.18-eq38}
\lim_{j\to\infty}\mE \int_{0}^{T}\Big|
\tilde{b}_{11}^{\e_{j}}(t)\(\frac{\delta
x^{\e_{j}}(t)}{\e_{j}},\frac{\delta
x^{\e_{j}}(t)}{\e_{j}}\)
-\frac{1}{2}b_{11}[t]\big(x_{1}(t),x_{1}(t)\big)\Big|_{\cL_2}^2dt
=0,
\end{equation}
\begin{equation}\label{8.18-eq39}
\lim_{j\to\infty}\mE \int_{0}^{T}\Big|2
\tilde{b}_{12}^{\e_{j}}(t)\(\frac{\delta
x^{\e_{j}}(t)}{\e_{j}},\frac{\delta
u^{\e_{j}}(t)}{\e_{j}}\) -
b_{12}[t]\big(x_{1}(t),u_1(t)\big)\Big|_{\cL_2}^2dt
=0,\vspace{-0.2cm}
\end{equation}
and\vspace{-0.12cm}
\begin{equation}\label{8.18-eq40}
\lim_{j\to\infty}\mE \int_{0}^{T}\Big|
\tilde{b}_{22}^{\e_{j}}(t)\(\frac{\delta
u^{\e_{j}}(t)}{\e_{j}},\frac{\delta
u^{\e_{j}}(t)}{\e_{j}}\)
-\frac{1}{2}b_{22}[t]\big(u_1(t),u_1(t)\big)\Big|_{\cL_2}^2dt
= 0.\vspace{-0.12cm}
\end{equation}
By Lemma \ref{estimatelinearsde},  and using
(\ref{8.18-eq32}), (\ref{8.18-eq33}) and
(\ref{8.18-eq35})--(\ref{8.18-eq40}),\vspace{-0.1cm}
\begin{eqnarray*}
\lim_{j\to\infty}\sup_{t\in
[0,T]}\mE|r_{2}^{\e_{j}}(t)|_H^{2}=0.
\end{eqnarray*}
The desired result follows from the fact that
the sequence $\{\varepsilon_j\}_{j=1}^\infty$ is
arbitrary.
\end{proof}

\section{Proof of Lemma \ref{lm8}}

\begin{proof}[Proof of Lemma \ref{lm8}] We borrow some idea from
\cite{WZ}. The proof is divided into three
steps.

\medskip

{\bf Step 1}. For any $\tilde v(\cd)\in
\cC_{\wt\cU_{ad}^{\nu_0}}(\bu(\cdot))$, we know
that $\tilde v(\cd)\in \cL^{2}_{\dbF}(0,T;H_1)$.
By Lemma \ref{lm3}, there exists a
$\cG$-measurable function $v(\cd)$ on
$[0,T]\times\Omega$ such that $\tilde
v(s,\o)=v(s,\o)$, $ \wt\mu$-a.e. Therefore,
\vspace{-0.2cm}
\begin{equation}\label{v*-v}
\int_{[0,T]\times\Omega}\!|\tilde
v(s,\o)-v(s,\o)|_{H_1}^2d\wt\mu(s,\o)=0
\vspace{-0.2cm}
\end{equation}
and\vspace{-0.12cm}
\begin{equation*}\label{v*-v.1}
|v(\cd)|_{L^2_{\dbF}(0,T;H_1)}^2=\int_{[0,T]\times\Omega}|v(s,\o)|_{H_1}^2d\wt\mu(s,\o)=\int_{[0,T]\times\Omega}|\tilde
v(s,\o)|_{H_1}^2d\wt\mu(s,\o)<\infty.\vspace{-0.12cm}
\end{equation*}

Since $\tilde v(\cd)\in
\cC_{\wt\cU_{ad}^{\nu_0}}(\bu(\cdot))$,  we
have\vspace{-0.21cm}
$$
\lim_{\hat v\to \bar u, \e \to
0^+}\frac{1}{\e}\inf_{\tilde u\in
\wt\cU_{ad}^{\nu_0}}\(\mE\int_0^T|\hat
v(t)+\e\tilde v(t)-\tilde
u(t)|_{H_1}^2dt\)^{\frac{1}{2}}=0.\vspace{-0.1cm}
$$
This, together with \eqref{v*-v}, implies
that\vspace{-0.1cm}
\begin{equation}\label{11.20-eq7}
\begin{array}{ll}\ds
\lim_{\hat v\to \bar u, \e\to
0^+}\frac{1}{\e}\inf_{\tilde u\in
\wt\cU_{ad}^{\nu_0}}\(\mE\int_0^T|\hat v(t)+\e
v(t)-\tilde u(t)|_{H_1}^2dt\)^{\frac{1}{2}}\\
\ns\ds\leq \lim_{\hat v\to \bar u, \e \to
0^+}\frac{1}{\e}\inf_{\tilde u\in
\wt\cU_{ad}^{\nu_0}}\(\mE\int_0^T|\hat
v(t)+\e\tilde v(t)-\tilde
u(t)|_{H_1}^2dt\)^{\frac{1}{2}}=0.\vspace{-0.1cm}
\end{array}
\end{equation}
For any $\tilde u\in \wt\cU_{ad}^{\nu_0}\subset
\cL^{2}_{\dbF}(0,T;H_1)$, by Lemma \ref{lm3},
there exists a $\cG$-measurable function
$u(\cd)$ on $[0,T]\times\Omega$ such that
$\tilde u(s,\o)=u(s,\o)$, $\wt\mu$-a.e.
Hence,\vspace{-0.1cm}
\begin{equation}\label{11.20-eq6}
\int_{[0,T]\times\Omega}\!|\tilde
u(s,\o)-u(s,\o)|_{H_1}^2d\wt\mu(s,\o)=0.
\vspace{-0.2cm}
\end{equation}
Consequently, $u\in \cU_{ad}^{\nu_0}$. This,
together with \eqref{11.20-eq7} and
\eqref{11.20-eq6}, implies that\vspace{-0.1cm}
\begin{equation}\label{11.20-eq8}
\begin{array}{ll}\ds
\lim_{\hat v\to \bar u, \e \to
0^+}\frac{1}{\e}\inf_{u\in
\cU_{ad}^{\nu_0}}\(\mE\int_0^T|\hat v(t)+\e
v(t)-
u(t)|_{H_1}^2dt\)^{\frac{1}{2}}\\
\ns\ds\leq\lim_{\hat v\to \bar u, \e \to
0^+}\frac{1}{\e}\inf_{\tilde u\in
\wt\cU_{ad}^{\nu_0}}\(\mE\int_0^T|\hat v(t)+\e
v(t)- \tilde
u(t)|_{H_1}^2dt\)^{\frac{1}{2}}=0.\vspace{-0.1cm}
\end{array}
\end{equation}
Therefore,  $v(\cd)\in
\cC_{\cU_{ad}^{\nu_0}}(\bar u(\cd))$ and
\vspace{-0.2cm}
\begin{equation}\label{integral-v*}
\begin{array}{ll}\ds
\int_{[0,T]\times\Omega}\!\lan F(t,\o),\tilde v(t,\o)\ran_{H_1} d\wt\mu(t,\o)\\
\ns\ds=  \int_{[0,T]\times\Omega}\!\lan
F(t,\o),v(t,\o)\ran_{H_1}
d\wt\mu(t,\o)\!=\!\dbE\int_0^T\!\lan
F(t),v(t)\ran_{H_1} dt\leq 0.
\end{array}\vspace{-0.2cm}
\end{equation}

\vspace{0.2cm}

{\bf Step 2}. In this step, we prove that the
set\vspace{-0.2cm}
\begin{equation}\label{cA-bar-u}
\cA_{\bar
u}\=\big\{(t,\omega)\in[0,T]\times\Omega\ \big|\
\lan F(t),v\ran_{H_1} \leq 0, \ \ \forall\ v\in
\cC_{U}(\bu(t))\big\}\in\wt\cG.\vspace{-0.2cm}
\end{equation}
We achieve this goal by showing that
\begin{equation}\label{cA-bar-u.1}
\cA_{\bar u}^c
=\big\{(t,\omega)\in[0,T]\times\Omega\ \big|\
\exists\ v\in \cC_{U}(\bu(t)),\;\lan
F(t),v\ran_{H_1} > 0
\big\}\in\wt\cG.\vspace{-0.2cm}
\end{equation}
For $k\in\dbN$, let\vspace{-0.32cm}
$$
\cB_{\bar u,k}
\=\Big\{(t,\omega)\in[0,T]\times\Omega\ \big|\
\exists\ v\in \cC_{U}(\bu(t)),\;\lan
F(t),v\ran_{H_1} \geq
\frac{1}{k}\Big\}.\vspace{-0.2cm}
$$
Clearly,\vspace{-0.2cm}
\begin{equation}\label{11.20-eq9}
\cA_{\bar u}^c =\bigcup_{k=1}^\infty\cB_{\bar
u,k}.
\end{equation}

By Corollary  \ref{Lem-adjacent} the set-valued
map $\cC_{U}(\bar
u(\cd)):[0,T]\times\Omega\rightsquigarrow H_1$
is $\wt\cG$-measurable. It follows from Lemma
\ref{lm2} that\vspace{-0.2cm}
$$
\big\{(t,\o,v)\in[0,T]\times\Omega\times H_1
\big|\ v\in \cC_U(\bar
u(t,\o))\big\}\in\wt\cG\otimes\cB(H_1).\vspace{-0.2cm}
$$
By the assumption on $F(\cd)$, we have
that\vspace{-0.2cm}
\begin{equation}\label{Graph-1}
\Big\{(t,\o,v)\in[0,T]\times\Omega\times H_1
\big|\ \inner{F(t)}{v}_{H_1}\geq \frac{1}{k},\;
v\in \cC_U(\bar u(t,\o))
\Big\}\in\wt\cG\otimes\cB(H_1).\vspace{-0.2cm}
\end{equation}
Define a set-valued map
$\L_k(\cd,\cd):[0,T]\times\Omega\rightsquigarrow
H_1$ as\vspace{-0.2cm}
$$
\L_k(t,\o)\=\Big\{v\in H_1\big|\
\inner{F(t)}{v}_{H_1}\geq \frac{1}{k},\;\, v\in
\cC_U(\bar u(t,\o)) \Big\},\qq
(t,\o)\in[0,T]\times\Omega.\vspace{-0.2cm}
$$
It follows from Lemma \ref{lm2} and
(\ref{Graph-1}) that  $\L_k$ is
$\wt\cG$-measurable. Then  $\cB_{\bar u,k}=
\L_k^{-1}(H_1)\in\wt\cG$. This, together with
\eqref{11.20-eq9}, implies \eqref{cA-bar-u.1}.
Consequently, we have \eqref{cA-bar-u}.

\vspace{0.2cm}

{\bf Step 3}. \rm In this step we  prove that
$\wt\mu(\cA_{\bar u})=T$.

For $k,\ m=1,2,\cdots$,  let\vspace{-0.2cm}
$$
B(0,m)\=\{v\in H_1||v|_{H_1}\leq
m\}\vspace{-0.2cm}
$$
and\vspace{-0.2cm}
$$
\cB_{\bar u,k,m}\=\Big\{(t,\omega)\in
[0,T]\times \Omega \,\Big|\, \exists\ v\in
\cC_{U}(\bu(t))\cap  B(0,m),\ s.t.\ \
\inner{F(t)}{v }_{H_1} \geq  \frac{1}{k}
\Big\}.\vspace{-0.2cm}
$$
It is clear that\vspace{-0.2cm}
$$
\cA^c_{\bar
u}=\bigcup_{k\geq1}\bigcup_{m\geq1}\cB_{\bar
u,k,m}.\vspace{-0.2cm}
$$
Similar to the proof of $\cB_{\bar u,k}\in
\wt\cG$, one can show that $\cB_{\bar
u,k,m}\in\wt\cG$.

Now we only need to prove that $\wt\mu(\cB_{\bar
u,k,m})=0$  for every $k,\ m\geq 1$. Let us do
this by a contradiction argument.

Suppose that there exist $k$ and $m$ such that
$\wt\mu(\cB_{\bar u,k,m})>0$. Define the
set-valued map $\Upsilon^{k,m}:\cB_{\bar
u,k,m}\rightsquigarrow H_1$ by\vspace{-0.482cm}
$$\Upsilon^{k,m}(t,\omega)\=\Big\{v\in \cC_{U}(\bu(t))\cap
B(0,m)\ \Big|\ \inner{F( t)}{v }_{H_1} \geq
\frac{1}{k} \Big\}.\vspace{-0.2cm}$$
Obviously, $\Upsilon^{k,m}(t,\o)$ is
closed-valued. Similar to \eqref{Graph-1},
\vspace{-0.28cm}
\begin{equation}
\begin{array}{ll}
\ns\ds \Big\{(t,\o,v)\in[0,T]\times\Omega\times
H_1 \Big|\ v\in \cC_U(\bar u(t,\o))\cap B(0,m),\
\inner{F(t)}{v}_{H_1}\geq \frac 1 k
\Big\}\in\wt\cG\otimes\cB(H_1).
\end{array}
\end{equation}
This, together with Lemma \ref{lm2}, implies
that $\Upsilon^{k,m}$ is $\wt\cG$-measurable.
Then by Lemma \ref{Pro-mea-sel} there exists a
$\wt\cG$-measurable selection
$v^{k,m}(\cdot)$ on $\cB_{\bar u,k,m}$,
i.e.,\vspace{-0.2cm}
$$
v^{k,m}(t,\omega)\in
\Upsilon^{k,m}(t,\omega)\subset\big[
\cC_{U}(\bu(t))\cap  B(0,m)\big], \q \forall\,
(t,\omega)\in \cB_{\bar u,k,m}.\vspace{-0.2cm}
$$
By Lemma \ref{lm7},\vspace{-0.32cm}
\begin{equation*}\label{cT-bar-u*.1}
\big\{v(\cdot)\in \cL^{2}_{\dbF}(0,T;H_1)\
\big|\ v(t)\in \cC_{U}(u(t)),\
\wt\mu\hb{-a.e.}\big\}  \subset
\cC_{\wt\cU_{ad}^{\nu_0}}(u(\cdot)).\vspace{-0.2cm}
\end{equation*}
Let $\wt v^{k,m}(\cd)\= v^{k,m}(\cd
)\chi_{\cB_{\bar u,k,m}}(\cd )$. Then
\vspace{-0.2cm}
\begin{equation}\label{negative-1}
\wt\mu\Big\{( t,\omega)\in[0,T]\,\Big|\ \lan
F(t),\wt v^{k,m}(t)\ran_{H_1} \geq
\frac{1}{k}\Big\}\geq \wt\mu(\cB_{\bar
u,k,m})>0.\vspace{-0.2cm}
\end{equation}
Therefore,\vspace{-0.2cm}
\begin{equation}\label{add eq4.27}
\int_{[0,T]}\int_\O\lan F(t,\o),\wt
v^{k,m}(t,\o)\ran_{H_1} d\wt\mu(t,\o)\geq
\frac{1}{k}\wt\mu(\cB_{\bar
u,k,m})>0.\vspace{-0.2cm}
\end{equation}
On the other hand, by Corollary
\ref{Lem-adjacent}, one has
$v^{k,m}(\cdot)\in\cT_{\bar u}\subset
\cC_{\wt\cU_{ad}^{\nu_0}}(\bar u(\cd))$. It
follows from (\ref{integral-v*})
that\vspace{-0.2cm}
$$
\int_{[0,T]}\int_\O\lan F(t,\o),\wt
v^{k,m}(t,\o)\ran_{H_1}
d\wt\mu(t,\o)\le0,\vspace{-0.2cm}
$$
which contradicts to (\ref{add eq4.27}).
Therefore, $\wt\mu(\cB_{\bar u,k,m})=0$.
Consequently, $\wt\mu(\cA_{\bar u}^c)=0$. Since
$\cA^c_{\bar u}\in\wt\cG$,  there exists a
$\cG$-measurable set $\cE_{\bar u}$ satisfying
$\cA^c_{\bar u}\subset \cE_{\bar u}$ and
$\wt\mu(\cA^c_{\bar u})=\mu(\cE_{\bar u})=0$.
Thus, $ \cE_{\bar u}^c \subset\cA_{\bar u}$ and
$[\mathbf{m}\times\dbP](\cE^c_{\bar u})=T$. This
completes the proof.
\end{proof}

\section{Proof of Lemma \ref{lm12}}

The case that $H$ is finite dimensional was
studied in \cite{FZZ2}. The proof for the
general case is similar. We give it here for the
sake of completeness.

\begin{proof}[Proof of Lemma \ref{lm12}]
Obviously, $L^{2}_{\dbF}(\Omega; C([0,T]; H))$
is a linear subspace of $L^{2}_{\dbF}(\Omega;
D([0,T]; H))$. For a given $\Lambda \in
L^{2}_{\dbF}(\Omega; C([0,T]; H))^*$, by the
Hahn-Banach theorem, there is an extension
$\widetilde\Lambda\in L^{2}_{\dbF}(\Omega;
D([0,T]; H))^*$ such that\vspace{-0.2cm}
\begin{equation}\label{(4)}
|\widetilde\Lambda|_{L^{2}_{\dbF}(\Omega;
D([0,T]; H))^*} =|\Lambda|_{L^{2}_{\dbF}(\Omega;
C([0,T]; H))^*}\vspace{-0.2cm}
\end{equation}
and\vspace{-0.2cm}
\begin{equation}\label{(3)}
\widetilde\Lambda(x(\cdot))
=\Lambda(x(\cdot)),\qquad \forall\; x(\cdot)\in
L_{\dbF}^{2}(\Omega; C([0,T];
H)).\vspace{-0.2cm}
\end{equation}

\vspace{0.2cm}

Recall that $\{e_k\}_{k=1}^\infty$ is an
orthonormal basis of $H$ and $\G_k$ the
projective operator from $H$ to
$H_k\=\span\{e_k\}$. Let $\L_k=\L \G_k$ and
$\wt\L_k=\wt\L\G_k$. Clearly,\vspace{-0.2cm}
$$
\begin{array}{ll}\ds
\L_k \in L^{2}_{\dbF}(\Omega; C([0,T];
H_k))^*=L^{2}_{\dbF}(\Omega; C([0,T];
\dbR))^*,\\
\ns\ds \wt\L_k \in L^{2}_{\dbF}(\Omega; D([0,T];
H_k))^*=L^{2}_{\dbF}(\Omega; D([0,T]; \dbR))^*,
\end{array}\vspace{-0.2cm}
$$
and\vspace{-0.32cm}
\begin{equation}\label{6.18-eq2}
\wt\Lambda(x(\cdot))
=\lim_{m\to\infty}\sum_{k=1}^m\wt\Lambda_k(x(\cdot)),\q
\forall\, x(\cdot)\in L_{\dbF}^{2}(\Omega;
D([0,T]; H)).\vspace{-0.2cm}
\end{equation}

For each $k\in\dbN$, from the proof of
\cite[Theorem 65, p. 254]{Meyer82}, we deduce
that, there exist two $\dbR$-valued processes
$\psi^+_k(\cdot)$ and $\psi^-_k(\cdot)$ of
bounded variations such that $\psi^+_k(\cdot)$
is optional and purely discontinuous,
$\psi^-_k(\cdot)$  is predictable with
$\psi_k^-(0)=0$,\vspace{-0.1cm}
\begin{equation}\label{6.18-eq1}
|\widetilde\Lambda_k|^2_{L^{2}_{\dbF}(\Omega;
D([0,T];
\dbR))^*}=\mE\Big|\int_{(0,T]}d|\psi_k^-(t)|
+\int_{[0,T)}d|\psi_k^+(t)|\Big|^2\vspace{-0.2cm}
\end{equation}
and, for any $x(\cdot)\in L_{\dbF}^{2}(\Omega;
D([0,T]; H))$,\vspace{-0.1cm}
\begin{equation}\label{6.18-eq3}
\widetilde\Lambda_k(x(\cdot)) =\mE\(\int_{(0,T]}
\G_k x_-(t) d\psi_k^-(t) +\int_{[0,T)} \G_k x(t)
d\psi_k^+(t)\),\vspace{-0.2cm}
\end{equation}
where $x_-(\cd)$ is the predictable modification
of $x(\cd)$, which equals $x(\cd)$ when
$x(\cd)\in\! L_{\dbF}^{2}(\Omega; C([0,T]; H))$.

Define two $H$-valued processes $\psi^+(\cdot)$
and $\psi^-(\cdot)$ as follows:\vspace{-0.1cm}
$$
\psi^+(\cdot)=\sum_{k=1}^\infty
\psi^+_k(\cdot)e_k,\qq
\psi^-(\cdot)=\sum_{k=1}^\infty
\psi^-_k(\cdot)e_k.\vspace{-0.2cm}
$$
Then\vspace{-0.32cm}
$$
\int_{(0,T]} \langle x_-(t), d\psi^-(t)\rangle_H
= \sum_{k=1}^\infty \int_{(0,T]}\G_k x_-(t)
d\psi_k^-(t)\vspace{-0.2cm}
$$
and\vspace{-0.12cm}
$$
\int_{[0,T)} \langle x(t), d\psi^+(t)\rangle_H =
\sum_{k=1}^\infty \int_{[0,T)}\G_k x(t)
d\psi_k^+(t).\vspace{-0.2cm}
$$
It follows from \eqref{6.18-eq2} and
\eqref{6.18-eq3} that\vspace{-0.1cm}
\begin{equation}\label{6.18-eq4}
\wt\Lambda(x(\cdot)) = \int_{(0,T]} \langle
x_-(t), d\psi^-(t)\rangle_H+\int_{[0,T)} \langle
x(t), d\psi^+(t)\rangle_H,\q \forall\,
x(\cdot)\in L_{\dbF}^{2}(\Omega; D([0,T]; H)).
\end{equation}
This, together with the arbitrariness of
$x(\cdot)\in L_{\dbF}^{2}(\Omega; D([0,T]; H))$,
implies that $\psi^+(\cdot)$ and $\psi^-(\cdot)$
are  functions of bounded variation
and\vspace{-0.1cm}
\begin{equation}\label{6.18-eq5}
|\widetilde\Lambda |^2_{L^{2}_{\dbF}(\Omega;
D([0,T];
H))^*}=\mE\Big|\int_{(0,T]}d|\psi^-(t)|_H
+\int_{[0,T)}d|\psi^+(t)|_H\Big|^2.
\end{equation}

Put $\psi^* \=  \psi^-  +\psi^+$. By
\eqref{(3)}, we have\vspace{-0.1cm}
$$
\Lambda(x(\cdot))=\mE~\int_0^T\inner{x(t)}{d\psi^*(t)}_H,\qquad
\forall\; x(\cdot)\in L_{\dbF}^{2}(\Omega;
C([0,T]; H)).\vspace{-0.2cm}
$$
Letting  $\psi =  \psi^*  -\psi^*(0)$, we obtain
(\ref{(1)}). (\ref{(2)}) follows from
(\ref{(1)}).
\end{proof}

\section{Proof of Lemma
\ref{lm13}}\label{sec-ap-E}

Before proving Lemma \ref{lm13}, we first recall
the following Riesz-type Representation Theorem
(See \cite[Corollary 2.3 and Remark 2.4]{LYZ}).

\begin{lemma}\label{lemma1}
Fix $t_1$ and $t_2$ satisfying $0\leq t_2 < t_1
\leq T$.  Assume that $\cY$ is a reflexive
Banach space. Then, for any $r,s\in [1,\infty)$,
it holds that\vspace{-0.2cm}
$$
\left(L^r_\dbF(t_2,t_1;L^s(\O;\cY))\right)^*=L^{r'}_\dbF(t_2,t_1;L^{s'}(\O;\cY^*)),\vspace{-0.2cm}
$$
where
$$
s'=\left\{
\begin{array}{ll}\ds
s/(s-1), &\mbox{ if } s\not=1,\\
\ns\ds \infty &\mbox{ if } s=1;
\end{array}
\right. \qq r'=\left\{
\begin{array}{ll}\ds
r/(r-1), &\mbox{ if } r\not=1,\\
\ns\ds \infty &\mbox{ if } r=1.
\end{array}
\right.
$$
\end{lemma}

Next, we recall the following result.
\begin{lemma}\cite[Lemma 2.5]{LZ1}\label{lemma2.1}
Assume that   $f_1\in L^2_{\dbF}(0,T;H)$ and
$f_2\in L^2_{\dbF}(0,T;H)$. Then there exists a
decreasing sequence $\{\e_n\}_{n=1}^\infty$ of
positive numbers such that
$\ds\lim_{n\to\infty}\e_n=0$, and\vspace{-0.2cm}
\begin{equation}\label{2.21}
\lim_{n\to\infty}\frac{1}{\e_n}\int_t^{t+\e_n}\mathbb{E}
\langle f_1(t),f_2(\tau)\rangle_{H}
d\tau=\mathbb{E} \langle
f_1(t),f_2(t)\rangle_{H},\qq\ae\, t\in
 [0,T].
 \end{equation}
\end{lemma}

\begin{proof}[Proof of Lemma \ref{lm13}]  It
suffices to consider a particular case for
\eqref{first ajoint equ}:
\begin{eqnarray}\label{bsystem1-1}
\left\{
\begin{array}{lll}
\ds dy (t)= -  A^* y(t) dt + f(t)dt + d\psi(t) + Y (t)dW (t)&\mbox{ in }[0,T),\\
\ns\ds y(T) = y_T,
\end{array}
\right.
\end{eqnarray}
where $y_T \in L^p_{\cF_T}(\O;H)$ and
$f(\cdot)\in L^1_{\dbF}(0,T; L^2(\O;H))$. The
general case follows from the well-posedness of
(\ref{bsystem1-1}) and the standard fixed point
technique.

We divide the proof into several steps. Since
the proof is very similar to that of
\cite[Theorem 3.1]{LZ1}, we give below only a
sketch.

\medskip

{\bf Step 1.} For any $t\in [\tau,T]$, let us
define a linear functional $\mathfrak{F}$
(depending on $t$) on the Banach space
$L^1_{\dbF}(t,T;L^2(\O;H))\times
L^2_{\dbF}(t,T;\cL_2)\times L^2_{\cF_t}(\O;H)$
as follows:\vspace{-0.2cm}
\begin{equation}\label{vw-th1-eq1.1}
\begin{array}{ll}
\ds\mathfrak{F}\big(f_1(\cdot),
f_2(\cdot),\eta\big) = \mathbb{E}\big\langle
\phi(T),y_T\big\rangle_H - \mathbb{E}\int_t^T
\big\langle \phi(s),f(s)\big\rangle_H ds -
\mE\int_t^T \langle \phi(s),d\psi(s) \rangle_H,
\\
\ns\ds \qq\qq\qq\forall\; \big(f_1(\cdot),
f_2(\cdot),\eta\big)\in
L^1_{\dbF}(t,T;L^2(\O;H))\times
L^2_{\dbF}(t,T;\cL_2)\times L^2_{\cF_t}(\O;H),
\end{array}\vspace{-0.2cm}
\end{equation}
where $\phi(\cdot)\in
L^{2}_{\dbF}(\O;C([t,T];H))$ is the mild
solution of the equation \eqref{fsystem2}. It is
an easy matter to show that $\mathfrak{F}$ is a
bounded linear functional. By Lemma
\ref{lemma1}, there exists a
triple\vspace{-0.12cm}
$$
\big(y^t(\cdot), Y^t(\cdot), \xi^t\big)\in
L^\infty_{\dbF}(t,T;L^2(\O;H))\times
L^2_{\dbF}(t,T;\cL_2)\times L^2_{\cF_t}(\O;H)
$$
such that\vspace{-0.2cm}
\begin{equation}\label{vw-th1-eq2}
\begin{array}{ll}\ds
\mathbb{E}\big\langle \phi(T),y_T\big\rangle_H -
\mathbb{E}\int_t^T \big\langle \phi(s),f(s)\big\rangle_H \,ds - \mE\int_t^T \langle \phi(s),d\psi(s) \rangle_H\\
\ns\ds =  \mathbb{E}\int_t^T \big\langle
f_1(s),y^t(s)\big\rangle_H \,ds + \mathbb{E}
\int_t^T\big\langle
f_2(s),Y^t(s)\big\rangle_{\cL_2} \,ds
+\mathbb{E} \big\langle\eta,\xi^t\big\rangle_H.
\end{array}\vspace{-0.2cm}
\end{equation}
It is clear that $\xi^T=y_T$.
Furthermore,\vspace{-0.1cm}
\begin{equation}\label{vw-th1-eq3}
\begin{array}{ll}\ds
|(y^t(\cdot), Y^t(\cdot),\xi^t)|_{
L^\infty_{\dbF}(t,T;L^2(\O; H)) \times L^2_{\dbF}(t,T;H)\times L^2_{\cF_t}(\O;H)} \\
\ns\ds \leq C\big(|f(\cdot)|_{
L^1_{\dbF}(t,T;L^2(\O;H))}+|y_T|_{
L^2_{\cF_T}(\O;H)}+
|\psi|_{L^2_\dbF(\O;BV(0,T;H))}\big),
\qq\q\forall\;t\in [\tau,T].
\end{array}\vspace{-0.2cm}
\end{equation}
\medskip

{\bf Step 2.} Note that $(y^t(\cdot),
Y^t(\cdot))$ obtained in Step 1 may depend on
$t$. Now we  show the time consistency of
$(y^t(\cdot), Y^t(\cdot))$, that is, for any
$t_1$ and $t_2$ satisfying $0\leq t_2 \leq t_1
\leq T$, it holds that
\begin{equation}\label{vw-th1-eq4}
\big(y^{t_2} (s,\o),Y^{t_2} (s,\o)\big)=\big(
y^{t_1}(s,\o), Y^{t_1}(s,\o)\big),\qq \ae (s,\o)
\in [t_1,T]\times\O,
\end{equation}
for a suitable choice of the $\eta$, $f_1$ and
$f_2$ in \eqref{fsystem2}. In fact, for any
fixed $\varrho_1 (\cdot)\in
L^1_{\dbF}(t_1,T;L^2(\O;H))$ and $\varrho_2
(\cdot)\in L^2_{\dbF}(t_1,T;\cL_2)$, we choose
first $t=t_1$, $\eta = 0$, $f_1(\cdot)=\varrho_1
(\cdot)$ and $f_2(\cdot) = \varrho_2(\cdot)$ in
\eqref{fsystem2}. From \eqref{vw-th1-eq2}, we
get that\vspace{-0.2cm}
\begin{equation}\label{vw-th1-eq5}
\begin{array}{ll}\ds \mathbb{E}\big\langle
\phi^{t_1}(T),y_T\big\rangle_H -
\mathbb{E}\int_{t_1}^T \big\langle
\phi^{t_1}(s),f(s)\big\rangle_H ds-
\mathbb{E}\int_{t_1}^T \big\langle
\phi^{t_1}(s),d\psi(s)\big\rangle_H  \\
\ns\ds = \mathbb{E}\int_{t_1}^T
\big\langle\varrho_1 (s),y^{t_1}(s)\big\rangle_H
ds+\mathbb{E}\int_{t_1}^T \big\langle\varrho_2
(s),Y^{t_1}(s)\big\rangle_{\cL_2} ds.
\end{array}
\end{equation}
Next, we choose $t=t_2$, $\eta = 0$, $f_1(\cd) =
\chi_{[t_1,T]}(\cd) \varrho_1 (\cd)$ and
$f_2(\cd) = \chi_{[t_1,T]}(\cd) \varrho_2(\cd)$
in \eqref{fsystem2}. It follows from
\eqref{vw-th1-eq2} that\vspace{-0.2cm}
\begin{equation}\label{vw-th1-eq6}
\begin{array}{ll}\ds \mathbb{E}\big\langle
\phi^{t_1}(T),y_T\big\rangle_H -
\mathbb{E}\int_{t_1}^T \big\langle
\phi^{t_1}(s),f(s)\big\rangle_H ds-
\mathbb{E}\int_{t_1}^T \big\langle
\phi^{t_1}(s),d\psi(s)\big\rangle_H\\
\ns\ds = \mathbb{E}\int_{t_1}^T
\big\langle\varrho_1 (s),y^{t_2}(s)\big\rangle_H
ds+\mathbb{E}\int_{t_1}^T \big\langle\varrho_2
(s),Y^{t_2}(s)\big\rangle_{\cL_2} ds.
\end{array}\vspace{-0.2cm}
\end{equation}
Combining  \eqref{vw-th1-eq5} and
\eqref{vw-th1-eq6}, we get\vspace{-0.1cm}
$$
\begin{array}{ll}\ds
\mathbb{E}\int_{t_1}^T \big\langle\varrho_1
(s),y^{t_1}(s)-y^{t_2}(s)\big\rangle_H ds
+\mathbb{E}\int_{t_1}^T \big\langle\varrho_2
(s),Y^{t_1}(s)-Y^{t_2}(s)\big\rangle_{\cL_2}\,
ds=0,\\
\ns\ds\qq \qq\qq\qq\qq\forall\;
\varrho_1(\cdot)\in
L^1_{\dbF}(t_1,T;L^2(\O;H)),\q \varrho_2
(\cdot)\in L^2_{\dbF}(t_1,T;\cL_2).
\end{array}\vspace{-0.2cm}
$$
This yields the desired equality
\eqref{vw-th1-eq4}.

Put\vspace{-0.2cm}
\begin{equation}\label{vw-th1-eq7}
y(t,\o)=y^\tau(t,\o),\qq Y (t,\o)=
Y^\tau(t,\o),\qq \forall\;(t,\o) \in
[\tau,T]\times\O.\vspace{-0.1cm}
\end{equation}
From  \eqref{vw-th1-eq4}, we see
that\vspace{-0.1cm}
\begin{equation}\label{vw-th1-eq8}
 \big(y^t (s,\o),Y^t
(s,\o)\big)=\big( y(s,\o),Y (s,\o)\big), \qq
\ae(s,\o) \in [t,T]\times\O.\vspace{-0.2cm}
\end{equation}
Combining \eqref{vw-th1-eq2} and
\eqref{vw-th1-eq8}, we deduce
that\vspace{-0.2cm}
\begin{equation}\label{vw-th1-eq9}
\begin{array}{ll}
\ds \mathbb{E}\big\langle
\phi(T),y_T\big\rangle_H - \mathbb{E}\int_t^T
\big\langle \phi(s),f(s)\big\rangle_H ds-
\mathbb{E}\int_t^T \big\langle
\phi(s),d\psi(s)\big\rangle_H \\
\ns\ds = \mathbb{E} \big\langle\eta,\xi^t
\big\rangle_H+\mathbb{E}\int_t^T \big\langle
f_1(s),y(s)\big\rangle_H ds +\mathbb{E} \int_t^T
\big\langle f_2(s),Y(s)\big\rangle_{\cL_2}
ds,\\\ns\ds \qq\q\forall\; \big(f_1(\cdot),
f_2(\cdot),\eta\big)\in
L^1_{\dbF}(t,T;L^2(\O;H))\times
L^2_{\dbF}(t,T;\cL_2)\times L^2_{\cF_t}(\O;H).
\end{array}\vspace{-0.2cm}
\end{equation}

\medskip

{\bf Step 3.} We show in this step that $\xi^t$
has a c\`adl\`ag modification.  The detail is
lengthy and very similar to Step 3 in the proof
of \cite[Theorem 3.1]{LZ1}, and hence we omit it
here.

First of all, we claim that, for each $t\in
[0,T]$,\vspace{-0.1cm}
\begin{equation}\label{6e1}
\dbE\Big(S^*(T-t) y_T - \int_t^T S^*(s-t)
f(s)ds- \int_t^T S^*(s-t) d\psi(s)
\;\Big|\;\cF_t\Big) = \xi^t,\ \
\dbP\mbox{-}\as\vspace{-0.2cm}
\end{equation}
To prove this, we note that for any $\eta\in
L^2_{\cF_t}(\O;H)$, $f_1=0$ and $f_2=0$, the
corresponding solution of \eqref{fsystem2} is
given by $\phi(s)=S(s-t)\eta$ for $s\in [t,T]$.
Hence, by \eqref{vw-th1-eq9}, we obtain
that\vspace{-0.1cm}
\begin{equation}\label{M2}
\begin{array}{ll}\ds
\dbE\big\langle S(T\!-\!t)\eta,y_T\big\rangle_H
\!-\! \dbE \langle \eta,\xi^t\rangle_H \!=\!
\dbE\!\int_t^T\! \big\langle
S(s\!-\!t)\eta,f(s)\big\rangle_H ds \!+\!
\dbE\!\int_t^T\! \big\langle
S(s-t)\eta,d\psi(s)\big\rangle_H.
\end{array}\vspace{-0.2cm}
\end{equation}
Noting that\vspace{-0.2cm}
$$
\dbE\big\langle
S(T-t)\eta,y_T\big\rangle_H=\dbE\big\langle
\eta,S^*(T-t)y_T\big\rangle_H=\dbE
\big\langle\eta,\dbE(S^*(T-t)y_T\;|\;\cF_t)\big\rangle_H,
$$
$$
\begin{array}{ll}\ds
\dbE\int_t^T \!\big\langle
S(s\!-\!t)\eta,f(s)\big\rangle_H ds\! =\!\dbE
\Big\langle \eta,\int_t^TS^*(s\!-\!t)
f(s)ds\Big\rangle_H \! =\!\dbE \Big\langle
\eta,\dbE\Big(\int_t^T\! S^*(s\!-\!t)
f(s)ds\;\Big|\;\cF_t\Big) \Big\rangle_H,
\end{array}
$$
and
$$
\begin{array}{ll}\ds
\dbE\int_t^T \big\langle
S(s-t)\eta,d\psi(s)\big\rangle_H
 =\dbE\Big\langle \eta,\int_t^TS^*(s-t)
d\psi(s)\Big\rangle_H
 =\dbE \Big\langle \eta,\dbE\Big(\int_t^T
S^*(s-t) d\psi(s)\;\Big|\;\cF_t\Big)
\Big\rangle_H,
\end{array}
$$
by (\ref{M2}), we conclude that for every
$\eta\in L^2_{\cF_t}(\O;H)$,\vspace{-0.1cm}
\begin{equation}\label{M1}
\dbE
\Big\langle\eta,\dbE\Big(S^*(T-t)y_T-\int_t^T
S^*(s-t) f(s)ds-\int_t^T S^*(s-t) d\psi(s)
\;\Big|\;\cF_t\Big)-\xi^t
\Big\rangle_H=0.\vspace{-0.1cm}
\end{equation}
Clearly, (\ref{6e1}) follows from (\ref{M1})
immediately.

In the rest of this step, we show that the
process\vspace{-0.2cm}
$$
\Big\{\dbE\Big(S^*(T-t) y_T - \int_t^T S^*(s-t)
f(s)ds
\;\Big|\;\cF_t\Big)\Big\}_{t\in[0,T]}\vspace{-0.2cm}
$$
has a c\`adl\`ag modification.

Recall that for any $\l\in\rho(A)$, the bounded
operator $A_\l$ (resp. $A^*_\l$) generates a
$C_0$-group $\{S_\l(t)\}_{t\in\dbR}$ (resp.
$\{S^*_\l(t)\}_{t\in\dbR}$) on $H$.

For each  $t\in [0,T]$, put\vspace{-0.2cm}
\begin{equation}\label{6e1.1z}
\xi_{\l}^t\=\dbE\Big(S^*_\l(T-t) y_T - \int_t^T
S^*_\l(s-t) f(s)ds- \int_t^T S^*_\l(s-t)
d\psi(s) \;\Big|\;\cF_t\Big)\vspace{-0.2cm}
\end{equation}
and\vspace{-0.2cm}
\begin{equation}\label{X}
\Phi_\l(t)\=S^*_\l(t)\xi_{\l}^t - \int_0^t
S^*_\l(s) f(s)ds-\int_t^T S^*_\l(s-t)
d\psi(s).\vspace{-0.2cm}
\end{equation}
We claim that $\{\Phi_\l(t)\}$ is an $H$-valued
$\dbF$-martingale. In fact, for any $\tau_1,
\tau_2 \in [0,T]$ with $\tau_1 \leq \tau_2$, it
follows from \eqref{6e1.1z} and \eqref{X}
that\vspace{-0.2cm}
$$
\!\!\!\!\begin{array}{ll}
\ns\ds\q\dbE(\Phi_\l(\tau_2)\;|\;\cF_{\tau_1})\\
\ns\ds =
\dbE\Big(S^*_\l(\tau_2)\xi_{\l}^{\tau_2} -
\int_0^{\tau_2}S^*_\l(s) f(s)ds-
\int_0^{\tau_2}S^*_\l(s)
d\psi(s) \;\Big|\;\cF_{\tau_1}\Big) \\
\ns\ds = \dbE \[\dbE\Big(S^*_\l(T)y_T\! -\!
\int_{\tau_2}^T\! S^*_\l(s)\!f(s)ds-
\int_{\tau_2}^T\! S^*_\l(s)d\psi(s) \Big|
\cF_{\tau_2}\Big)\! - \!\!\int_0^{\tau_2}\!\!
S^*_\l(s)f(s)ds\!-\!\! \int_0^{\tau_2}\!\!
S^*_\l(s)d\psi(s) \Big| \cF_{\tau_1}\]\\
\ns\ds  = \dbE\Big(S^*_\l(T)y_T-\int_0^{T}
S^*_\l(s)f(s)ds-\int_0^{T}
S^*_\l(s)d\psi(s) \;\Big|\;\cF_{\tau_1}\Big)\\
\ns\ds =
S^*_\l(\tau_1)\dbE\Big(S^*_\l(T-\tau_1)y_T-\int_{\tau_1}^{T}
S^*_\l(s-\tau_1)f(s)ds-\int_{\tau_1}^{T}
S^*_\l(s-\tau_1)d\psi(s) \;\Big|\;\cF_{\tau_1}\Big)\\
\ns\ds\q -\int_0^{\tau_1}S^*_\l(s)f(s)ds-\int_0^{\tau_1}S^*_\l(s)d\psi(s) \\
\ns\ds = S^*_\l(\tau_1)\xi_{\l}^{\tau_1}- \int_0^{\tau_1}S^*_\l(s)f(s)ds- \int_0^{\tau_1}S^*_\l(s)d\psi(s) \\
\ns\ds = X_\l(\tau_1), \q \dbP\mbox{-a.s.},
\end{array}
$$
as desired.

Now, since $\{X_\l(t)\}_{0\leq t\leq T}$ is an
$H$-valued $\dbF$-martingale, it enjoys a
c\`adl\`ag modification, and hence so does the
following process
$$\{\xi_{\l}^t\}_{0\leq t \leq
T}=\Big\{S^*_\l(-t)\Big[X_\l(t)+ \int_0^t
S^*_\l(s) f(s)ds+ \int_0^t S^*_\l(s)
d\psi(s)ds\Big]\Big\}_{0\leq t \leq T}.
$$
Here we have used the fact that
$\{S_\l^*(t)\}_{t\in\dbR}$ is a $C_0$-group on
$H$. We still use $\{\xi_{\l}^t\}_{0\leq t \leq
T}$ to stand for its c\`adl\`ag modification.

From \eqref{6e1} and \eqref{6e1.1z}, it follows
that
\begin{eqnarray}\label{th1eq3.1}
&&\3n\3n\3n\3n\ds \q
\lim_{\l\to\infty} |\xi^\cd - \xi_{\l}^{\cd}|_{L^\infty_{\dbF}(0,T;L^2(\O;H))} \nonumber\\
&&\3n\3n\3n\3n\ds =  \lim_{\l\to\infty} \Big|
\dbE\Big(S^*(T-\cd)) y_T - \int_\cd^T
S^*(s-\cd)) f(s)ds-
\int_\cd^T S^*(s-\cd))d\psi(s) \;\Big|\;\cF_\cd\Big) \nonumber\\
&&\3n\3n\3n\3n\ds \q - \dbE\Big(S^*_\l(T-\cd))
y_T- \int_\cd^T S^*_\l(s-\cd) f(s)ds -
\int_\cd^T S^*_\l(s-\cd)d\psi(s) \;\Big|\;\cF_\cd\Big) \Big|_{L^\infty_{\dbF}(0,T;L^2(\O;H))}\\
&&\3n\3n\3n\3n\ds \leq  \lim_{\l\to\infty}   \Big| S^*(T-\cd) y_T - S^*_\l(T-\cd) y_T \Big|_{L^\infty_{\dbF}(0,T;L^2(\O;H))} \nonumber\\
&&\3n\3n\3n\3n\ds \q +  \lim_{\l\to\infty} \Big|
\int_\cd^T S^*(s-\cd) f(s)ds - \int_\cd^T
S^*_\l(s-\cd)
f(s)ds \Big|_{L^\infty_{\dbF}(0,T;L^2(\O;H))}\nonumber\\
&&\3n\3n\3n\3n\ds \q +  \lim_{\l\to\infty} \Big|
\int_\cd^T S^*(s-\cd)d\psi(s) - \int_\cd^T
S^*_\l(s-\cd) d\psi(s)
\Big|_{L^\infty_{\dbF}(0,T;L^2(\O;H))}.\nonumber
\end{eqnarray}

Let us prove the right hand side of
\eqref{th1eq3.1} equals zero. First, we
prove\vspace{-0.1cm}
\begin{equation}\label{z-x1}
\lim_{\l\to\infty}   \Big| S^*(T-\cd) y_T -
S^*_\l(T-\cd) y_T
\Big|_{L^\infty_{\dbF}(0,T;L^2(\O;H))} =
0.\vspace{-0.1cm}
\end{equation}
By the property of Yosida approximations, we
deduce that for any $\a\in H$, it holds
that\vspace{-0.1cm}
$$
\lim_{\l\to\infty}|S^*(T-\cd)\a -
S^*_\l(T-\cd)\a|_{L^\infty(0,T;H)}=0\vspace{-0.2cm}
$$
and that\vspace{-0.2cm}
$$
\big| S^*(T-\cd) y_T - S^*_\l(T-\cd) y_T
\big|_{H}\leq C|y_T|_H.\vspace{-0.1cm}
$$
Thus, by Lebesgue's dominated convergence, we
obtain \eqref{z-x1}.

Similarly, we can prove that\vspace{-0.1cm}
\begin{equation}\label{z-x2}
\lim_{\l\to\infty}\Big| \int_\cd^T S^*(s-\cd)
f(s)ds - \int_\cd^T S^*_\l(s-\cd) f(s)ds
\Big|_{L^\infty_{\dbF}(0,T;L^2(\O;H))}=0\vspace{-0.2cm}
\end{equation}
and\vspace{-0.1cm}
\begin{equation}\label{z-x2.1}
\lim_{\l\to\infty}\Big| \int_\cd^T S^*(s-\cd)
d\psi(s)  - \int_\cd^T S^*_\l(s-\cd)d\psi(s)
\Big|_{L^\infty_{\dbF}(0,T;L^2(\O;H))}=0.
\vspace{-0.1cm}
\end{equation}

By \eqref{th1eq3.1}, \eqref{z-x1}, \eqref{z-x2}
and \eqref{z-x2.1}, we obtain that
$\ds\lim_{m\to\infty}\lim_{\l\to\infty} |\xi^\cd
-
\xi_{\l,m}^{\cd}|_{L^\infty_{\dbF}(0,T;L^2(\O;H))}=0$.
Recalling that $\xi_{\l}^{\cd}\in
D_{\dbF}([0,T];L^2(\O;H))$, we deduce that
$\xi^\cd$ enjoys a c\'{a}dl\'{a}g modification.

\vspace{0.15cm}

{\bf Step 4.} In this step, we show that, for
a.e. $t\in [0,T]$,\vspace{-0.1cm}
\begin{equation}\label{6e3}
\xi^t= y(t),\ \ \dbP\mbox{-}\as\vspace{-0.1cm}
\end{equation}
Choosing $t=t_2$, $f_1(\cd) = 0$, $f_2(\cd) = 0$
and $\eta = (t_1-t_2)\gamma$ in
\eqref{fsystem2}, utilizing \eqref{vw-th1-eq9},
we obtain that
\begin{equation}\label{eq6xz}
\begin{array}{ll}\ds
\mathbb{E}\big\langle S(T-t_2)(t_1-t_2)\gamma,
y_T \big\rangle_H - \mathbb{E}\big\langle
(t_1-t_2)\gamma, \xi^{t_2} \big\rangle_H \\
\ns\ds = \mathbb{E}\int_{t_2}^T \big\langle
S(\tau-t_2) (t_1-t_2)\gamma,
f(\tau)\big\rangle_H
d\tau+\mathbb{E}\int_{t_2}^T \big\langle
S(\tau-t_2) (t_1-t_2)\gamma,
d\psi(\tau)\big\rangle_H.
\end{array}\vspace{-0.2cm}
\end{equation}
Choosing $t=t_2$, $f_1(\tau,\o) =
\chi_{[t_2,t_1]}(\tau)\gamma(\o)$, $f_2(\cd) =
0$ and $\eta = 0$ in (\ref{fsystem2}), utilizing
\eqref{vw-th1-eq9} again, we find
that\vspace{-0.2cm}
\begin{equation}\label{eq7zx}
\begin{array}{ll}\ds
\q\mathbb{E}\Big\langle \int_{t_2}^T S(T-s) \chi_{[t_2,t_1]}(s)\gamma ds, y_T \Big\rangle_H \\
\ns \ds= \mathbb{E}\int_{t_2}^{t_1}\Big\langle
\int_{t_2}^\tau S(\tau-s)\gamma ds,
f(\tau)\Big\rangle_H d\tau +
\mathbb{E}\int_{t_1}^T \Big\langle S(\tau-t_1)
\int_{t_2}^{t_1}S(t_1-s)\gamma ds,
f(\tau)\Big\rangle_H d\tau \\
\ns\ds\q + \mathbb{E}\int_{t_2}^{t_1}\Big\langle
\int_{t_2}^\tau S(\tau-s)\gamma ds,
d\psi(\tau)\Big\rangle_H +
\mathbb{E}\int_{t_1}^T \Big\langle S(\tau-t_1)
\int_{t_2}^{t_1}S(t_1-s)\gamma ds,
d\psi(\tau)\Big\rangle_H \\
\ns\ds \q + \mathbb{E}\int_{t_2}^{t_1}\langle
\gamma,y(\tau)\rangle_H d\tau.
\end{array}\vspace{-0.2cm}
\end{equation}
It follows from (\ref{eq6xz}) and (\ref{eq7zx})
that\vspace{-0.1cm}
\begin{eqnarray}\label{eq7.1}
&&\ds
\q\mathbb{E}\langle \gamma, \xi^{t_2}\rangle_H \nonumber\\
&&\ds =\!
\frac{1}{t_1\!-\!t_2}\!\int_{t_2}^{t_1}\!\!\mathbb{E}\lan\!
\gamma,y(\tau) \rangle_H
d\tau \! +\! \mathbb{E}\big\langle S(T\!-\!t_2) \gamma, y_T \big\rangle_H \!-\! \frac{1}{t_1\!-\!t_2}\mathbb{E}\Big\langle \int_{t_2}^T\!\! S(T\!-\!\tau) \chi_{[t_2,t_1]}(\tau)\gamma d\tau, y_T \Big\rangle_H \nonumber\\
&&\ds \q  - \mathbb{E}\int_{t_2}^T \langle
S(\tau-t_2) \gamma, f(\tau)\rangle_H d\tau +
\frac{1}{t_1-t_2}\mathbb{E}\int_{t_2}^{t_1}\Big\langle
\int_{t_2}^\tau S(\tau-s)\gamma,
f(\tau)\Big\rangle_H d\tau\\
&&\ds \q +
\frac{1}{t_1-t_2}\mathbb{E}\int_{t_1}^T
\Big\langle S(\tau-t_1)
\int_{t_2}^{t_1}S(t_1-s)\gamma ds,
f(\tau)\Big\rangle_H d\tau-
\mathbb{E}\int_{t_2}^T \langle S(\tau-t_2)
\gamma, d\psi(\tau)\rangle_H \nonumber\\
&&\ds\q +
\frac{1}{t_1\!-\!t_2}\mathbb{E}\!\int_{t_2}^{t_1}\!\!\Big\langle\!
\int_{t_2}^\tau\! S(\tau\!-\!s)\gamma,
d\psi(\tau)\Big\rangle_H\! +\!
\frac{1}{t_1\!-\!t_2}\mathbb{E}\!\int_{t_1}^T\!\!
\Big\langle\! S(\tau\!-\!t_1)\!
\int_{t_2}^{t_1}\!\!S(t_1\!-\!s)\gamma ds,
d\psi(\tau)\Big\rangle_H.\nonumber
\end{eqnarray}
Now we analyze the terms in the right hand side
of \eqref{eq7.1}. First, it is easy to show
that\vspace{-0.2cm}
\begin{equation}\label{eq7.11}
\begin{array}{ll}\ds
\lim_{t_1\to t_2+0}
\frac{1}{t_1-t_2}\mathbb{E}\int_{t_2}^{t_1}\Big\langle
\int_{t_2}^\tau S(s-t_2)\gamma,
f(\tau)\Big\rangle_H d\tau \\
\ns\ds + \lim_{t_1\to t_2+0}\frac{1}{t_1 -
t_2}\mathbb{E} \int_{t_2}^{t_1} \Big\langle
\int_{t_2}^\tau S(\tau - s)\gamma,
d\psi(\tau)\Big\rangle_H =0,\qq\forall\;\gamma
\in L^2_{\cF_{t_2}}(\O;H).
\end{array}\vspace{-0.2cm}
\end{equation}
 Further,\vspace{-0.2cm}
\begin{equation}\label{eq7.13}
 \begin{array}{ll}\ds
\q\lim_{t_1\to t_2+0}   \frac{1}{t_1-t_2}\mathbb{E}\Big\langle \int_{t_2}^T S(T-\tau) \chi_{[t_2,t_1]}(\tau)\gamma d\tau, y_T \Big\rangle_H \\
\ns\ds = \lim_{t_1\to t_2+0}  \frac{1}{t_1-t_2}
\mathbb{E} \Big\langle \int_{t_2}^{t_1}
S(T-\tau) \gamma d\tau, y_T \Big\rangle_H =
\mathbb{E}\big\langle S(T-t_2) \gamma, y_T
\big\rangle_H.
\end{array}\vspace{-0.1cm}
\end{equation}
Utilizing the semigroup property of
$\{S(t)\}_{t\geq 0}$, we have\vspace{-0.2cm}
\begin{equation}\label{eq7.14}
\begin{array}{ll}
\ds\lim_{t_1\to t_2+0}
\frac{1}{t_1-t_2}\mathbb{E}\[\int_{t_1}^T
\Big\langle S(\tau-t_1)
\int_{t_2}^{t_1}S(t_1-s)\gamma ds,
f(\tau)\Big\rangle_H d\tau\\
\ns\ds\qq\qq\qq\q + \int_{t_1}^T \Big\langle
S(\tau-t_1) \int_{t_2}^{t_1}S(t_1-s)\gamma ds,
d\psi(\tau)\Big\rangle_H  \] \\
\ns\ds = \mathbb{E}\int_{t_2}^T \big\langle
S(\tau-t_2) \gamma, f(\tau)\big\rangle_H d\tau +
\mathbb{E}\int_{t_2}^T \big\langle S(\tau-t_2)
\gamma, d\psi(\tau)\big\rangle_H.\vspace{-0.2cm}
\end{array}
\end{equation}
From \eqref{eq7.1}, \eqref{eq7.11},
\eqref{eq7.13} and \eqref{eq7.14}, we arrive
at\vspace{-0.2cm}
\begin{equation}\label{eq8xz}
\lim_{t_1\to t_2+0}
\frac{1}{t_1-t_2}\int_{t_2}^{t_1}\mathbb{E}\langle
\gamma,y(\tau) \rangle_H
d\tau=\mathbb{E}\big\langle \gamma,
\xi^{t_2}\big\rangle_H,\q \forall\;\gamma \in
L^2_{\cF_{t_2}}(\O;H), \ t_2\in
[0,T).\vspace{-0.2cm}
\end{equation}
Now, by \eqref{eq8xz}, we conclude that, for
$\ae t_2\in (0,T)$\vspace{-0.2cm}
\begin{equation}\label{ezq10}
\lim_{t_1\to t_2+0}
\frac{1}{t_1-t_2}\int_{t_2}^{t_1}\mathbb{E}\big\langle
\xi^{t_2}-y(t_2),y(\tau) \big\rangle_H
d\tau=\mathbb{E}\big\langle \xi^{t_2}-y(t_2),
\xi^{t_2}\big\rangle_H.
\end{equation}
By Lemma \ref{lemma2.1}, we can find a monotonic
sequence $\{\e_n\}_{n=1}^\infty$ of positive
numbers with $\ds\lim_{n\to\infty}\e_n=0$, such
that\vspace{-0.2cm}
\begin{equation}\label{eq12}
\lim_{n\to\infty}
\frac{1}{\e_n}\int_{t_2}^{t_2+\e_n}\!\mathbb{E}\langle
\xi^{t_2}-y(t_2),y(\tau) \rangle_H
d\tau=\mathbb{E}\langle
\xi^{t_2}\!-\!y(t_2),y(t_2)\rangle_H,\q\ \ae
t_2\in [0,T).\vspace{-0.2cm}
\end{equation}
By (\ref{ezq10})--(\ref{eq12}), we arrive
at\vspace{-0.1cm}
\begin{equation}\label{eq14}
\mathbb{E}\langle \xi^{t_2}-y(t_2),
\xi^{t_2}\rangle_H=\mathbb{E}\langle
\xi^{t_2}-y(t_2),y(t_2)\rangle_H,\q\ \ae t_2\in
[0,T].\vspace{-0.1cm}
\end{equation}
By (\ref{eq14}), we find that $\mathbb{E}\left|
\xi^{t_2}-y(t_2)\right|_H^2=0$ for $ t_2\in
[0,T]$ a.e., which implies (\ref{6e3}).

This completes the proof of Lemma \ref{lm13}.

\end{proof}

\end{document}